\documentclass[11pt]{article}

\usepackage[utf8]{inputenc}
\usepackage[T1]{fontenc}

\usepackage[a4paper,left=21mm,right=21mm,top=24mm,bottom=28mm]{geometry}

\usepackage{amsmath,amsfonts,amssymb,amsthm,mathtools}

\usepackage{graphicx}
\graphicspath{{figs/}}

\usepackage{booktabs}
\usepackage{bm}
\usepackage{array}
\usepackage{tabularx}

\usepackage{verbatim}
\usepackage{float}
\usepackage{xcolor}
\usepackage{xurl}
\usepackage{microtype}

\usepackage{hyperref}
\hypersetup{
  colorlinks=true,
  linkcolor=blue,
  citecolor=blue,
  urlcolor=blue
}

\usepackage{fancyhdr}
\usepackage{tikz}
\usepackage{subcaption}

\newcommand{\R}{{\mathbb R}}

\newcommand{\vep}{\varepsilon}

\newcommand{\ka}{\kappa}

\newtheorem{thm}{Theorem}[section]

\newtheorem{lem}{Lemma}[section]

\newtheorem{exa}{Example}[section]

\theoremstyle{definition}

\theoremstyle{remark}
\newtheorem{rem}{Remark}[section]

\numberwithin{equation}{section}

\newcommand{\be}{\begin{eqnarray}}
\newcommand{\en}{\end{eqnarray}}
\newcommand{\ben}{\begin{eqnarray*}}
\newcommand{\enn}{\end{eqnarray*}}

\newcommand{\beq}{\begin{equation}}
\newcommand{\eeq}{\end{equation}}

\newcommand{\bef}{\begin{figure}[H]}
\newcommand{\eef}{\end{figure}}

\newcommand{\bet}{\begin{tikzpicture}}
\newcommand{\eet}{\end{tikzpicture}}

\def\bx{\mathbf{x}}

\def\bv{\mathbf{v}}

\def\bn{\mathbf{n}}

\newcommand{\PMLDE}{PML-DE}
\newcommand{\PMLDES}{\PMLDE{} system}

\newcommand{\CN}{CN}

\title{\bf A Structure-Preserving PML-Domain-Embedding Method for
  Acoustic Wave Scattering
by Moving Objects}

\author{Xuelong Gu and Qi Wang}
\date{}

\begin{document}
\maketitle

\begin{abstract}
	We develop a structure-preserving computational framework for acoustic wave scattering by moving objects, comprising a new PML-domain-embedding model and a compatible numerical approximation. The model couples a perfectly matched layer (PML), used to truncate the acoustic wave equation, with a domain-embedding formulation that represents moving objects on a fixed computational domain. The resulting PML-domain-embedding (\PMLDE{}) system enables moving-boundary scattering problems to be solved without remeshing.
	Using matched asymptotic expansions, we show that the diffuse-interface formulation converges to the corresponding sharp-interface system as the interface thickness tends to zero. We then construct an energy-dissipation-rate-preserving finite-difference scheme for the \PMLDE{} system. To improve computational efficiency, the scheme is combined with hierarchical local refinement informed by the moving-object location, the fixed PML region, and the evolving wave dynamics, all within the fixed computational domain.
	Numerical experiments demonstrate the accuracy of the computed scattering solutions, the effectiveness of the absorbing layer and object-embedding strategy, and the efficiency of the adaptive algorithm. The proposed framework provides a practical and robust computational approach for engineering applications involving complex acoustic wave-scattering problems.
\end{abstract}

\noindent\textbf{Keywords:}
Acoustic wave scattering; perfectly matched layer; domain embedding;
structure-preserving schemes; moving objects.

\section{Introduction}

\noindent \indent Acoustic wave scattering by moving objects arises in a broad range of
applications, including underwater acoustics and sonar target
identification, aeroacoustic diagnostics and noise mitigation, medical
ultrasound, ultrasonic nondestructive evaluation, wave-based imaging,
and inverse scattering for target localization and reconstruction
\cite{Petkov1989Scattering,intro:01,intro:02,intro:image,intro:03,intro:04}.
In such problems, the incident field interacts with a time-dependent
geometry and gives rise to reflection, diffraction, shadowing, and
Doppler effects
\cite{intro:doppler,Petkov1989Scattering,intro:03,intro:04}.
Accurate time-domain simulation is therefore essential not only for
resolving transient scattering phenomena, but also for quantifying the
cumulative influence of object motion on the scattered field.

Analytical aspects of wave propagation in time-dependent domains and
moving-geometry settings have been studied in
\cite{Lazzaroni2019,LiuGao2019,intro:03,intro:04}. From a numerical
perspective, the main difficulty is not only that the physical domain
is an unbounded exterior one and the object boundary evolves in
time, but also that the exterior truncation adopted in practical computations, the geometric
representation of the computational domain, and the time discretization of the truncated equation should be coordinated in a
manner compatible with the underlying mathematical structure of the model. This structural
viewpoint is particularly important for robust and faithful long-time simulation and
for the construction of structure-preserving algorithms.

For the implementation of the unbounded exterior domain, a variety of
nonreflecting techniques have been developed, ranging from local
absorbing boundary conditions and exact or approximate transparent
boundary conditions to perfectly matched layers (PMLs)
\cite{engquist1977absorbing,Grote&K1996,Alpert2002,FallettaMonegato2014,Berenger1994,DiazJoly2006,PML_2009,ChenWu2012,Joly2012}.
In the time domain, a PML introduces auxiliary variables and
damping mechanisms whose dissipative structure is not immediately
transparent. Thus, before introducing the moving boundary into the embedded model, we first make the  energy in the PML and dissipation explicit for a frozen object configuration. This fixed-geometry calculation only serves to identify the energy of the coupled system later. Motivated by previous energy and stability
analyses of time-domain PMLs
\cite{DiazJoly2006,ErvedozaZuazua2008,ChenWu2012,Joly2012},
we aim to derive such a formulation for general wave speed $c>0$ and use it
as the starting point of the present work.

For the geometric representation of moving objects, body-fitted
finite element, discontinuous Galerkin, and boundary element methods
remain powerful tools for wave propagation and scattering in complex
domains
\cite{ihlenburg1998finite,ihlenburg1995finite,ihlenburg1997finite,adjerid_sisc_dg,ralf_anm_2014,wang_aa_2018,bem_cmame_2020,bem_nm_2021,falletta2018bem,bae_2012}.
In moving-geometry problems, however, they typically require
remeshing, mesh deformation, or repeated boundary discretization as
the object evolves. Immersed boundary, fictitious-domain, and
diffuse-domain approaches alleviate this difficulty by replacing sharp
moving boundaries with fixed-grid formulations equipped with
penalties, Lagrange multipliers, or smooth indicators
\cite{immersed_2003,immersed_wave_scattering,error_fictitious_1995,FallettaMonegato2015,lowen-2009-commun,lowen-cms-2009,lowen-jcp-2011,lowen-jcp-2020}.
Diffuse-domain ideas have also been used in a variety of
complex-geometry and evolving-interface settings, including multiphase
flows, flow-structure interactions, and interfacial transport problems
\cite{lowen-2010-diffuse-chns,lowen-2014-diffuse-tumor,lowen-2019-diffuse-tumor,lowen-2021-diffuse-jfm,lowen-jcp-2016,HongWang2024,HongWang2025}.
Recent fixed-domain, thermodynamically consistent,  domain-embedding formulations for dissipative PDEs
\cite{Yu-Q-W-2025-1,Yu-Q-W-2025-2,Gu2026,GuJiWang2026} suggest that moving interfaces can
be incorporated through smooth, time-dependent coefficients while
still admitting compatible energy-based formulations. For acoustic scattering by moving objects, the recent domain-embedding strategy with radiation boundary conditions  provides a closely related fixed-domain reference point \cite{Gu2026}.
For arbitrarily shaped moving objects, the domain-embedding strategy provides a highly versatile engineering solution for the complex computational challenge.

For the present problem of wave scattering by a moving object, the fixed-computational-domain viewpoint has an additional computational
advantage: once the moving interface of the object and the absorbing PML are both
represented on a fixed Cartesian domain, the resulting system becomes
naturally amenable to finite-difference discretizations
and hierarchical local refinement, without geometry-dependent
remeshing or repeated boundary quadrature on evolving interfaces.
Recent structure-preserving algorithms for conservative and
dissipative PDEs suggest that
discretizations of the reformulated system can often be designed so as to inherit continuous
	{energy laws} at the discrete level
\cite{GongCaiWang2014,CaiZhangWang2017,CaiZhangWang2017SEM,MuGongCaiWang2018,Hong&W&G2023,Gu2026JCP}.

Motivated by the above, we develop a structure-preserving approximation to the
PML-domain-embedding (\PMLDE{}) system
for acoustic wave scattering by moving objects in a fixed computational
domain. The key idea is not merely to combine a PML with a domain-embedding
technique for the moving object in the model reformulation, but  to explore the gradient flow structure of the governing system of equations to show its energy dissipation property and devise efficient structure-preserving algorithms.  This provides a unified fixed-domain
framework in which absorption, geometry evolution, and
structure-preserving computation are handled coherently.

The main contributions of the paper are summarized as follows.
\begin{enumerate}
	\item We reformulate the {PML-truncated acoustic wave equation} for general
	      wave speed $c>0$ in a fixed computational domain as a gradient-flow system
	      with a quadratic energy. This reformulation reveals the
	      intrinsic energy dissipation mechanism of the truncated wave problem and
	      provides the structural basis for the subsequent reduction,
	      coupling, domain-embedding, and discretization.

	\item Based on the reformulation, we construct
	      structure-preserving time discretizations. In particular, we show that the
	      structure-preserving Crank--Nicolson (\CN{}) scheme with a semi-discrete dissipation
	      law  is algebraically equivalent to a leapfrog scheme of a reduced
	      two-field PML formulation, whose staggered energy can
	      be reconstructed systematically from {the auxiliary}
	      variables.

	\item We couple the reduced PML model {directly at the two-field level} with a diffuse domain
	      embedding of a moving sound-soft object to obtain {the \PMLDES{}}
	      in a fixed computational domain. We also introduce an additional damping mechanism within the moving object to enforce the sound-soft boundary condition. For the
	      coupled \PMLDES{} system, we establish a continuous, semi-discrete, and fully
	      discrete weighted energy law  {with an additional energy rate of change term that vanishes for a static object}, show formally by matched asymptotic
	      expansions that the diffuse model recovers the sharp-interface PML
	      model as $\vep\to 0$, and develop an adaptive
	      finite-difference realization with local refinement near the
	      diffuse interface, the PML layer, and regions of strong wave
	      activity.
\end{enumerate}
The theoretical results and asymptotic analyses presented in this work apply only to the problem with sound-soft boundary conditions. For sound-hard boundary conditions, we devise an analogous numerical scheme in the absence of normal acceleration and use it to simulate wave scattering by a rigid object undergoing uniform translational motion. The problem of object motion with sound-hard boundary conditions and nonzero normal acceleration is beyond the scope of this study.

The remainder of the paper is organized as follows.
In \S~\ref{sec:PML-wave}, we reformulate the acoustic PML model as a
gradient-flow system and show its energy dissipative property. {\S~\ref{sec:structure-discretization} is devoted to the development of the structure-preserving time discretization in the leap-frog form.} In \S~\ref{sec:pml-de-system}, we formulate the \PMLDES{} for moving objects, establish its formal consistency with the sharp-interface case under the sound-soft boundary condition, derive the corresponding weighted energy law,  and present the spatially adaptive algorithm. In \S~\ref{sec:numerical}, we present some numerical results for fixed and
moving objects to show the effectiveness of the computational platform and solution accuracy. Concluding remarks are given in \S~\ref{sec:conclusion}. Some
technical algebraic derivations are summarized in the appendix.

\section{Reformulation of the wave equation with a perfectly matched layer (PML)}\label{sec:PML-wave}

\noindent \indent We reformulate the wave scattering problem for a moving object in an infinite exterior domain as a problem posed on a finite computational domain equipped with a perfectly matched layer (PML). By introducing suitable auxiliary variables, we cast the governing equations in the computational domain as a gradient-flow system with a quadratic free energy, and explicitly identify the associated mobility matrix. We then establish the energy-dissipative property of the model in the case of a static object.

\subsection{Reformulation of the wave equation  with a PML and its
	dissipative properties}

\noindent \indent Let $D(t)\subset\mathbb{R}^2$ be a  moving object with a smooth boundary and
$\bn_\star$ denote the unit normal vector on $\partial D(t)$ pointing
from the object into the wave scattering region. For the sound-soft boundary condition at the object's boundary, the
acoustic pressure satisfies the following initial boundary value wave equation system:
\begin{equation}\label{eq:main_prob}
	\left\lbrace
	\begin{aligned}
		 & p_{tt} - c^2\Delta p = f(\mathbf{x}, t),
		 &                                              & (\bx,t)\in \mathbb{R}^2\setminus\overline{D(t)}\times(0,T], \\
		 & p(\bx,t)=0,
		 &                                              & (\bx,t)\in \partial D(t)\times(0,T],                        \\
		 & p(\bx,0)=p_0(\bx),\quad p_t(\bx,0)=p_1(\bx),
		 &                                              & \bx\in \mathbb{R}^2\setminus\overline{D(0)},
	\end{aligned}
	\right.
\end{equation}
where $p (\mathbf{x}, t)$ represents the pressure disturbance and $D(t)$
the moving object region. Here, we assume that the motion of the object
is subsonic, namely $|\mathbf{v}_{D(t)} \cdot \mathbf{n}_\star
	|\leq c$ (see \cite{Lazzaroni2019, LiuGao2019}).  For the sound-hard
case, the boundary condition on $\partial D(t)$ is replaced by
$\nabla p\cdot\bn_\star = 0$. In either case, a practical numerical
treatment on the unbounded exterior region would require a finite domain
truncation together with a nonreflecting or radiation boundary condition at the truncated domain boundary.
\begin{rem}
	In this study, we focus our analysis on moving objects subject to a sound-soft boundary condition. The corresponding wave scattering problem with a sound-hard boundary condition at the interface of the moving object is more subtle, as one must distinguish between cases with and without normal acceleration. In the former case, the analysis and algorithm developed for the sound-soft problem can be extended straightforwardly to the sound-hard setting. Accordingly, in the sound-hard numerical tests presented in this paper, the moving objects are prescribed to undergo rigid translations with constant velocity, so that the normal acceleration vanishes. The accelerated sound-hard case is therefore not addressed in the present study due to its added complication.
\end{rem}

The source term in the scattering problem is given by  $f(\mathbf{x}, t) = c^2 g(\mathbf{x})\gamma(t)$, where
$g(\mathbf{x})$ and $\gamma(t)$ are prescribed as
\begin{equation}
	g(\mathbf{x}) = \tfrac{1}{\sqrt{2\pi} \eta} {\rm exp} \left(-
	\tfrac{|\mathbf{x} - \mathbf{x}_0|^2}{2\eta}\right), \quad \chi(t)
	= \sin{(w t)} {\rm exp}(-\sigma t^2),
\end{equation}
$\mathbf{x}_0$ is the center of source, $\eta, w, \sigma$
are parameters characterizing the spatial width and temporal profile
of the source.

In this study, we replace the far-field radiation boundary condition at the finite computational domain
by a perfectly matched layer in which the wave is damped severely.
For example, we let the wave scattering region, $\Omega_{\rm phy} = (-a_1, a_1) \times
	(-a_2, a_2)$, be surrounded by a perfectly matched
layer, $\Omega_{\rm PML}$. Together they serve as the
computational domain. In the PML, the wave equation is replaced by an overdamped
equation  so that the resulting, composite PDE system can approximate the wave equation in the scattering
domain without inflicting any wave reflection at the PML interface.

Since the wave equation is conservative while the overdamped equation in the PML  is dissipative, we want to analyze the energy property of the composite equation in the computational domain
so that it will later serve as the reference for the embedded model to be presented in the next section. {In the following analysis, the object is regarded as stationary so that the effect of motion is not included until the time-dependent embedding profile is introduced in Section~\ref{sec:pml-de-system}.}  Following \cite{arxiv-pml},
the two-dimensional PML reformulation of the wave equation in
\eqref{eq:main_prob} within the domain,  $\Omega = \Omega_{\rm phys}
	\cup \Omega_{\rm PML}$, is given by
\begin{equation}\label{eq:wave_pml}
	\left\lbrace
	\begin{aligned}
		                                             & p_{tt} + a p_t + b p
		= c^2 \Delta p + \nabla \cdot \bm{\phi} + f, &                                                             & (\mathbf{x}, t)
		\in \Omega \setminus \overline{D(t)} \times (0, T],                                                                                                                                           \\
		                                             & \bm{\phi}_t + \varGamma_1 \bm{\phi}
		= c^2 \varGamma_2 \nabla p,                  &                                                             & (\mathbf{x}, t)
		\in \Omega \setminus \overline{D(t)} \times (0, T],                                                                                                                                           \\
		                                             & p(\mathbf{x}, t) = 0,                                       &                                              & (\mathbf{x}, t) \in \partial D(t)
		\times (0, T],                                                                                                                                                                                \\
		                                             & p(\mathbf{x}, t) = 0, \quad \bm{\phi} \cdot \mathbf{n} = 0, &                                              &
		(\mathbf{x} , t)  \in \partial \Omega \times (0, T],                                                                                                                                          \\
		                                             & p(\bx,0)=p_0(\bx),\quad p_t(\bx,0)=p_1(\bx), \quad
		\bm{\phi}(\mathbf{x}, 0) = 0,
		                                             &                                                             & \bx\in \mathbb{R}^2\setminus\overline{D(0)},
	\end{aligned}
	\right.
\end{equation}
where
\begin{equation*}
	\varGamma_1(\mathbf{x}) =
	\begin{pmatrix}
		\xi_1 & 0     \\
		0     & \xi_2
	\end{pmatrix}, \quad
	\varGamma_2(\mathbf{x}) =
	\begin{pmatrix}
		\xi_2  - \xi_1 & 0             \\
		0              & \xi_1 - \xi_2
	\end{pmatrix}, \quad a = {\rm tr}(\varGamma_1), \quad b = {\rm
			det}(\varGamma_1).
\end{equation*}
Here $\xi_1$ and $\xi_2$ are the PML damping profiles in the $x$- and
$y$-directions, respectively. In this paper, we set
\begin{equation}
	\left\lbrace
	\begin{aligned}
		 & \xi_i(x_i) = 0,                                    &                                & |x_i| < a_i, &           & i = 1, 2 \\
		 & \overline{\xi}_i \left( \tfrac{|x_i - a_i|}{L_i} -
		\tfrac{\sin{( \tfrac{2\pi |x_i - a_i|}{L_i}  )}}{2 \pi}  \right),
		 &                                                    & a_i \leq |x_i| \leq a_i + L_i, &              & i = 1, 2.
	\end{aligned}
	\right.
\end{equation}
where $\overline{\xi}_i$ is a user defined parameter dependent on the
discretization and the thickness of the layer, $\bm{\phi}$ is the
standard auxiliary vector
field introduced by the PML construction in \cite{arxiv-pml}. Throughout
the paper, $(\bullet, \bullet)$ denotes the $L^2$ inner product,
$\|\bullet\|$ the corresponding norm, and
$(\bullet, \bullet)_\omega = (\omega \bullet, \bullet)$ together with
$\|\bullet \|_\omega^2 = (\bullet, \bullet)_\omega$ for a symmetric positive
semidefinite matrix field $\omega$.

\subsection{Dissipative property of the PML wave system for a static object}

\noindent \indent In this section, we show the energy property of the PML formulation
in a fixed-domain, i.e.,  we temporarily
freeze the geometry and assume that $D(t)$ in \eqref{eq:wave_pml} is
time-independent and $f = 0$. {Therefore, the dissipative structure derived below should be read as a fixed-object, or frozen-geometry, result.} In \cite{baffet-2019}, this result was proved for
the special case $c = 1$. To extend the result to the general case $c
	> 0$ while maintaining
algebraic simplicity, we introduce the following auxiliary variables:
\beq
\bm{\lambda} = \tfrac{1}{c} \bm{\phi}, \quad
\bm{\chi} = c \nabla p + \tfrac{1}{c}\bm{\phi} = c\nabla p + \bm{\lambda},
\quad
q = p_t + a p.
\eeq
Then, the first and second equation in \eqref{eq:wave_pml} become
\begin{equation}\label{eq:q-eq}
	q_t = - b p + c \nabla \cdot \bm{\chi},
\end{equation}
and
\begin{equation}\label{eq:lambda-pre}
	\bm{\lambda}_t + \varGamma_1 \bm{\lambda} = c \varGamma_2 \nabla p.
\end{equation}

We next derive the evolution equation for $\bm{\chi}$. Differentiating
its definition with respect to time yields
\beq
\bm{\chi}_t
= c \nabla p_t +  \bm{\lambda}_t
= c \nabla (q - ap) + \bm{\lambda}_t
= c \nabla q - ca \nabla p + \bm{\lambda}_t.
\eeq
Substituting \eqref{eq:lambda-pre} into the above identity, using
$\nabla p = \tfrac{1}{c}(\bm{\chi} - \bm{\lambda})$, and observing
that $\varGamma_2 - a I = - 2 \varGamma_1$, we obtain
\begin{equation}
	\begin{aligned}
		\bm{\chi}_t
		 & = c \nabla q
		- ca \nabla p
		- \varGamma_1 \bm{\lambda}
		+ c \varGamma_2 \nabla p   \\
		 & = c \nabla q
		+ (\varGamma_2 - a I)\bm{\chi}
		- (\varGamma_2 - a I)\bm{\lambda}
		- \varGamma_1 \bm{\lambda} \\
		 & = c \nabla q
		- 2 \varGamma_1 \bm{\chi}
		+ \varGamma_1 \bm{\lambda}.
	\end{aligned}
\end{equation}
Similarly, the $\bm{\lambda}$-equation can be rewritten as
\begin{equation}\label{eq:phi-eq}
	\bm{\lambda}_t
	= - \widetilde{\varGamma}_1 \bm{\lambda}
	+ \varGamma_2 \bm{\chi},
	\quad
	\widetilde{\varGamma}_1 := \varGamma_1 + \varGamma_2 .
\end{equation}
Collecting \eqref{eq:q-eq}--\eqref{eq:phi-eq}, we obtain the
first-order system
\begin{equation}\label{eq:first-order-system}
	\left\lbrace
	\begin{aligned}
		p_t & = q - ap,                                \\
		q_t & = - b p + c \nabla \cdot \bm{\chi},      \\
		\bm{\chi}_t
		    & = c \nabla q
		- 2 \varGamma_1 \bm{\chi}
		+ \varGamma_1 \bm{\lambda},                    \\
		\bm{\lambda}_t
		    & = - \widetilde{\varGamma}_1 \bm{\lambda}
		+ \varGamma_2 \bm{\chi}.
	\end{aligned}
	\right.
\end{equation}

Let $\Phi = (p, q, \bm{\chi}, \bm{\lambda})^\top$. Then
\eqref{eq:first-order-system} can be written in the compact form
\begin{equation} \label{eq:wave-pml-compact}
	\Phi_t
	= \mathcal{M} \frac{\delta \mathcal{E}}{\delta \Phi}, \quad
	\mathcal{E} = \frac12\Big(
	\|q\|^2
	+ \|\bm{\chi}\|^2
	+ \|p\|_b^2
	+ \|\bm{\lambda}\|_{a^{-1}\varGamma_1}^2
	\Big).
\end{equation}
where
\begin{equation*}
	\mathcal{M} =
	\begin{pmatrix}
		-\tfrac{a}{b} & 1       & 0              & 0                                         \\
		-1            & 0       & c\nabla\!\cdot & 0                                         \\
		0             & c\nabla & -2\varGamma_1  & a I                                       \\
		0             & 0       & \varGamma_2    & -a\widetilde{\varGamma}_1\varGamma_1^{-1}
	\end{pmatrix}.
\end{equation*}
Using $\varGamma_1 - \widetilde{\varGamma}_1 = -\varGamma_2$ and
$\varGamma_1 + \widetilde{\varGamma}_1 = a I$, we further split
$\mathcal{M} = \mathcal{M}_{sym} +
	\mathcal{M}_{skw}$ into its symmetric and skew-symmetric parts. This
decomposition makes the dissipative structure of the
PML system explicit:
\begin{equation*}
	\mathcal{M}_{sym} =
	\begin{pmatrix}
		-\tfrac{a}{b} & 0 & 0                       & 0                          \\
		0             & 0 & 0                       & 0                          \\
		0             & 0 & -2\varGamma_1           & \widetilde{\varGamma}_1    \\
		0             & 0 & \widetilde{\varGamma}_1 & -a \widetilde{\varGamma}_1
		\varGamma_1^{-1}
	\end{pmatrix}
	, \quad
	\mathcal{M}_{skw} =
	\begin{pmatrix}
		0  & 1       & 0              & 0           \\
		-1 & 0       & c \nabla \cdot & 0           \\
		0  & c\nabla & 0              & \varGamma_1 \\
		0  & 0       & -\varGamma_1   & 0
	\end{pmatrix}.
\end{equation*}

\begin{thm}\label{thm:dissipative-pml-continuous}
	{Under the fixed-geometry assumption stated above,} for every sufficiently smooth solution of \eqref{eq:first-order-system}, the
	quadratic energy $\mathcal{E}$ defined in \eqref{eq:wave-pml-compact}
	obeys
	\begin{equation}\label{eq:pml-cont-energy-law}
		\frac{d\mathcal{E}}{dt} = -\mathcal{D}(t),
		\quad
		\mathcal{D}(t)
		= \|p\|_{ab}^2
		+ \|\bm{\chi}\|_{\varGamma_1 + a^{-1}\varGamma_1^2}^2
		+ \|\bm{\chi} - \bm{\lambda}\|_{a^{-1}b}^2 \ge 0.
	\end{equation}
\end{thm}

\begin{proof}
	Taking the inner product of \eqref{eq:wave-pml-compact} with
	$\delta\mathcal{E}/\delta\Phi$ gives
	\[
		\frac{d\mathcal{E}}{dt}
		= \left(\frac{\delta\mathcal{E}}{\delta\Phi},
		\mathcal{M}\frac{\delta\mathcal{E}}{\delta\Phi}\right).
	\]
	Since $\mathcal{M}_{\mathrm{skw}}$ contributes no quadratic form,
	\begin{align*}
		\frac{d\mathcal{E}}{dt}
		 & = \left(\frac{\delta\mathcal{E}}{\delta\Phi},
		\mathcal{M}_{\mathrm{sym}}\frac{\delta\mathcal{E}}{\delta\Phi}\right) \\
		 & = -(abp,p) - 2(\bm{\chi},\varGamma_1\bm{\chi})
		+ 2(a^{-1}\varGamma_1\bm{\lambda},\widetilde{\varGamma}_1\bm{\chi})
		- (a^{-1}\varGamma_1\bm{\lambda},\widetilde{\varGamma}_1\bm{\lambda}) \\
		 & = -\|p\|_{ab}^2
		- \|\bm{\chi}\|_{\varGamma_1 + a^{-1}\varGamma_1^2}^2
		- \|\bm{\chi} - \bm{\lambda}\|_{a^{-1}b}^2,
	\end{align*}
	where we have used
	$\widetilde{\varGamma}_1\varGamma_1=bI$ and
	$\varGamma_1^2-a\varGamma_1+bI=0$. This proves
	\eqref{eq:pml-cont-energy-law}.
\end{proof}

For the special case, $c=1$, the corresponding dissipation property was
proved in \cite{baffet-2019}. The analysis above shows that the same
structure persists for arbitrary wave speed $c>0$.

\section{Structure-preserving discretization}\label{sec:structure-discretization}
\subsection{Structure-preserving discretization of reformulated wave equation with PML}
\label{sec:reduction}

\noindent \indent We next consider the temporal discretization of the PML wave
system as a gradient flow. Our goal is to preserve, at the semi-discrete level, the same
	{fixed-geometry} dissipative mechanism identified in
\eqref{eq:wave-pml-compact}. Let $T>0$ be a final time. We partition
the interval $[0,T]$ into $N$ subintervals with time-step size
$\tau=T/N$ and introduce the temporal difference operators:
\begin{equation}\label{eq:time-ops}
	\begin{aligned}
		D_t^+\phi^n  & = \frac{\phi^{n+1}-\phi^n}{\tau},
		             & D_t^-\phi^n                                     & = \frac{\phi^n-\phi^{n-1}}{\tau},
		             & D_{2t}\phi^n                                    & = \frac{\phi^{n+1}-\phi^{n-1}}{2\tau},     \\
		D_t^2\phi^n  & = \frac{\phi^{n+1}-2\phi^n+\phi^{n-1}}{\tau^2},
		             & A_t^+\phi^n                                     & = \frac{\phi^{n+1}+\phi^n}{2},
		             & A_t^-\phi^n                                     & = \frac{\phi^n+\phi^{n-1}}{2},             \\
		A_{2t}\phi^n & = \frac{\phi^{n+1}+\phi^{n-1}}{2},
		             & A_t^2\phi^n                                     & = \frac{\phi^{n+1}+2\phi^n+\phi^{n-1}}{4}.
	\end{aligned}
\end{equation}

A natural structure-preserving discretization of
\eqref{eq:wave-pml-compact} is the Crank--Nicolson scheme
\begin{equation*}
	D_t^+ \Phi^n = \mathcal{M} \frac{\delta \mathcal{E}}{\delta
		\Phi}[A_t^+ \Phi^n].
\end{equation*}
Its componentwise form is
\begin{equation}\label{eq:wave-pml-semi-discrete-component}
	\left\lbrace
	\begin{aligned}
		 & D_t^+ p^n = A_t^+ q^n - a  A_t^+ p^n,                        \\
		 & D_t^+ q^n = - b A_t^+ p^n + c \nabla \cdot A_t^+\bm{\chi}^n, \\
		 & D_t^+ \bm{\chi}^n = c \nabla A_t^+ q^n - 2 \varGamma_1 A_t^+
		\bm{\chi}^n + \varGamma_1 A_t^+ \bm{\lambda}^n,                 \\
		 & D_t^+ \bm{\lambda}^n = -\widetilde{\varGamma}_1 A_t^+
		\bm{\lambda}^n + \varGamma_2 A_t^+ \bm{\chi}^n.
	\end{aligned}
	\right.
\end{equation}

\begin{thm}\label{thm:CN-energy}
	The Crank--Nicolson discretization,
	\eqref{eq:wave-pml-semi-discrete-component},
	inherits the dissipative property of the continuous system and satisfies
	\begin{equation}\label{eq:CN-energy-law}
		D_t^+\mathcal{E}^n = -\mathcal{D}^{n+\frac12},
	\end{equation}
	with
	\begin{equation*}
		\begin{aligned}
			\mathcal{E}^n
			 & = \frac12\Bigl(\|q^n\|^2 + \|\bm{\chi}^n\|^2
			+ \|p^n\|_b^2 + \|\bm{\lambda}^n\|_{a^{-1}\varGamma_1}^2\Bigr), \\
			\mathcal{D}^{n+\frac12}
			 & = \|A_t^+p^n\|_{ab}^2
			+ \|A_t^+\bm{\chi}^n\|_{\varGamma_1 + a^{-1}\varGamma_1^2}^2
			+ \|A_t^+(\bm{\chi}^n - \bm{\lambda}^n)\|_{a^{-1}b}^2.
		\end{aligned}
	\end{equation*}
\end{thm}

\begin{proof}
	Multiplying \eqref{eq:wave-pml-semi-discrete-component} by
	$\delta\mathcal{E}/\delta\Phi[A_t^+\Phi^n]$ gives
	$D_t^+\mathcal{E}^n =
		\bigl(\frac{\delta\mathcal{E}}{\delta\Phi}[A_t^+\Phi^n],
		D_t^+\Phi^n\bigr)$.
	The rest is the discrete analogue of the proof of
	Theorem~\ref{thm:dissipative-pml-continuous}, where one replaces
	$(p,q,\bm{\chi},\bm{\lambda})$ with their midpoint values
	$(A_t^+p^n,A_t^+q^n,A_t^+\bm{\chi}^n,A_t^+\bm{\lambda}^n)$.
\end{proof}

The midpoint formulation ,\eqref{eq:wave-pml-semi-discrete-component},
contains four coupled fields. For implementation, however, it is often
preferable to work with a reduced system involving fewer unknowns. This
leads to the following two-field leap-frog scheme associated with
\eqref{eq:wave-pml-semi-discrete-component}:
\begin{equation}\label{eq:wave-pml-semi-discrete-leap-frog}
	\left\lbrace
	\begin{aligned}
		 & D_t^2 p^n + a D_{2t} p^n + b A_t^2 p^n = c^2  \Delta  A_t^2
		p^n + c \nabla \cdot A_t^2 \bm{\lambda}^n,                     \\
		 & D_t^+ \bm{\lambda}^n + \varGamma_1 A_t^+ \bm{\lambda}^n = c
		\varGamma_2 \nabla A_t^+ p^n.
	\end{aligned}
	\right.
\end{equation}

The reduced leap-frog formulation is convenient in practice, but its
energy stability is less transparent. The next
two lemmas show that it is algebraically equivalent to the midpoint
\CN{} scheme. This equivalence clarifies the origin of the staggered
energy law and will later simplify the derivation for the object embedded formulation.

\begin{lem}
	\label{lem:constraint}
	Let $\bm{r}^n := \bm{\chi}^n - c\nabla p^n - \bm{\lambda}^n$. Then the
	midpoint scheme, \eqref{eq:wave-pml-semi-discrete-component}, yields
	\begin{equation}
		\label{eq:r-update}
		D_t^+ \bm{r}^n = -a A_t^+ \bm{r}^{n}.
	\end{equation}
	Hence, if $\bm{r}^0=0$, then $\bm{r}^n=0$ for all $n \ge 0$.
\end{lem}

\begin{lem}
	\label{lem:reduction}
	Assume $\bm{r}^0=0$, then $\bm{r}^n\equiv 0$ by
	Lemma~\ref{lem:constraint}. The midpoint scheme,
	\eqref{eq:wave-pml-semi-discrete-component}, is algebraically
	equivalent to the two-field scheme,
	\eqref{eq:wave-pml-semi-discrete-leap-frog}, for $(p^n,\bm{\lambda}^n)$.
	Moreover, once $(p^n,p^{n+1},\bm{\lambda}^n,\bm{\lambda}^{n+1})$ are known,
	the eliminated variables are recovered from
	\begin{equation}\label{eq:reconstruct_qn_chin}
		\begin{aligned}
			q^{n+1}     & = q^{n+\frac12} + \tfrac{\tau}{2}F^{n+\frac12},
			            & \qquad q^n                                      & = q^{n+\frac12} - \tfrac{\tau}{2}F^{n+\frac12}, \\
			\bm{\chi}^n & = c\nabla p^n + \bm{\lambda}^n,
			            & \qquad \bm{\chi}^{n+1}                          & = c\nabla p^{n+1} + \bm{\lambda}^{n+1},
		\end{aligned}
	\end{equation}
	where
	\begin{equation}\label{eq:den-qF}
		q^{n+\frac12}=D_t^+p^n+aA_t^+p^n,
		\qquad
		F^{n+\frac12}=-(b-c^2\Delta)A_t^+p^n+c\nabla\cdot A_t^+\bm{\lambda}^n.
	\end{equation}
\end{lem}

The reconstruction formulas above immediately yield a discrete energy
law for the leap-frog scheme.

\begin{thm}\label{thm:pml-semi-second-ene}
	The leap-frog
	scheme, \eqref{eq:wave-pml-semi-discrete-leap-frog}, satisfies
	\begin{equation}\label{eq:reduced-energy-law}
		\widehat{\mathcal{E}}^{n+\frac12}
		- \widehat{\mathcal{E}}^{n-\frac12}
		= -\widetilde{\mathcal{D}}^n,
	\end{equation}
	with
	\begin{equation*}
		\begin{aligned}
			\widehat{\mathcal{E}}^{n+\frac12}
			 & = \frac12\Bigl(
			\|D_t^+p^n + a A_t^+ p^n\|^2
			+ \|A_t^+(cp^n + \bm{\lambda}^n)\|^2
			+ \|A_t^+p^n\|_b^2
			+ \|A_t^+\bm{\lambda}^n\|_{a^{-1}\varGamma_1}^2
			\Bigr),                  \\
			\widetilde{\mathcal{D}}^n
			 & = \|A_t^2p^n\|_{ab}^2
			+ \|A_t^2(cp^n + \bm{\lambda}^n)\|_{\varGamma_1 + a^{-1}\varGamma_1^2}^2
			+ \|cA_t^2\nabla p^n\|_{a^{-1}b}^2.
		\end{aligned}
	\end{equation*}
\end{thm}

\begin{proof}
	Observe that \eqref{eq:wave-pml-semi-discrete-component} implies
	\begin{equation}\label{eq:wave-pml-semi-discrete-shifted}
		\left\lbrace
		\begin{aligned}
			 & D_t^+ A_t^- p^n = A_t^+ A_t^- q^n - a  A_t^+ A_t^- p^n,           \\
			 & D_t^+ A_t^- q^n = - b A_t^+ A_t^- p^n + c \nabla \cdot A_t^+
			A_t^-\bm{\chi}^n,                                                    \\
			 & D_t^+ A_t^- \bm{\chi}^n = c \nabla A_t^+ A_t^- q^n - 2
			\varGamma_1 A_t^+
			A_t^- \bm{\chi}^n + \varGamma_1 A_t^+ A_t^- \bm{\lambda}^n,          \\
			 & D_t^+ A_t^- \bm{\lambda}^n = -\widetilde{\varGamma}_1 A_t^+ A_t^-
			\bm{\lambda}^n + \varGamma_2 A_t^+ A_t^-\bm{\chi}^n.
		\end{aligned}
		\right.
	\end{equation}
	Consequently, a half-step-shifted version of
	Theorem~\ref{thm:CN-energy} gives
	\begin{equation*}
		\widehat{\mathcal{E}}^{n+\frac{1}{2}} -
		\widehat{\mathcal{E}}^{n-\frac{1}{2}} =
		-\widetilde{\mathcal{D}}^{n},
	\end{equation*}
	where
	\begin{equation*}
		\begin{aligned}
			 & \widehat{\mathcal{E}}^{n+\frac{1}{2}} = \frac{1}{2}
			\left(\|A_t^+ q^n\|^2 + \|A_t^+
			\bm{\chi}^n\|^2
			+ \|A_t^+ p^n\|_b^2 + \|A_t^+
			\bm{\lambda}^n\|_{a^{-1}\varGamma_1}^2\right),          \\
			 & \widetilde{\mathcal{D}}^{n} = \|A_t^2 p^n\|^2_{ab} +
			\|A_t^2\bm{\chi}^n\|^2_{\varGamma_1 + a^{-1} \varGamma_1^2} +
			\|A_t^2(\bm{\chi}^n - \bm{\lambda}^n)\|_{a^{-1} b}^2.
		\end{aligned}
	\end{equation*}
	Substituting the reconstruction formulas from
	Lemma~\ref{lem:reduction} for $q$ and $\bm{\chi}$ yields the desired result.
\end{proof}
The energy law in Theorem~\ref{thm:pml-semi-second-ene} can also be derived
directly from \eqref{eq:wave-pml-semi-discrete-leap-frog}; here we use the
reconstruction argument because it makes the underlying dissipative
structure of the system more transparent and {motivates the corresponding weighted energy estimate for}
the embedded system discussed next.

\section{The \PMLDE{} system and its structure-preserving algorithm}\label{sec:pml-de-system}

\noindent \indent We now combine the reduced PML formulation of the wave scattering problem with a domain-embedding
description of a moving sound-soft object. The computational  advantage is
clear: the moving interface problem is replaced by time-dependent coefficients of a modified PDE system
in a fixed computational domain. {In the embedded formulation below, the unknowns are \((p,\bm{\lambda})\). }
The analytical issue is how the energy dissipation rate
in the fixed-geometry  for the PML formulated problem  in
Section~\ref{sec:PML-wave} is altered after the embedding step. {When the object moves, it inflict an additional, indefinite energy rate of change dictated by the interface motion,}
provided that the object stays away from the PML collar.
The analyses in this section apply exclusively to the embedded system with a sound-soft boundary condition. For comparison, in the subsequent section we present numerical results for the embedded model with sound-hard boundary conditions and compare them with those obtained under the sound-soft boundary condition. We make no analytical claims regarding the structure-preserving properties of the numerical scheme used for the sound-hard boundary-condition model.

\subsection{The \PMLDE{} formulation}

\noindent \indent Let $\Omega_{\rm phy}\subset\R^2$ be the scattering
region and let $\Omega_{\rm PML}$ be a surrounding PML collar. The
total domain is the fixed domain
\[
	\Omega = \Omega_{\rm phy}\cup\Omega_{\rm PML}.
\]
Inside $\Omega_{\rm phy}$, let $D(t)$ denote a moving sound-soft
object with outward unit normal $\bn_\star$ pointing from the
wave region into the object. The target sharp-interface problem is the
acoustic wave equation in $\Omega\setminus\overline{D(t)}$, coupled
with the PML damping equation in $\Omega_{\rm PML}$. Figure~\ref{fig:pml}
illustrates the corresponding geometry.
\begin{figure}[H]
	\begin{center}
		\includegraphics[width=0.5\textwidth]{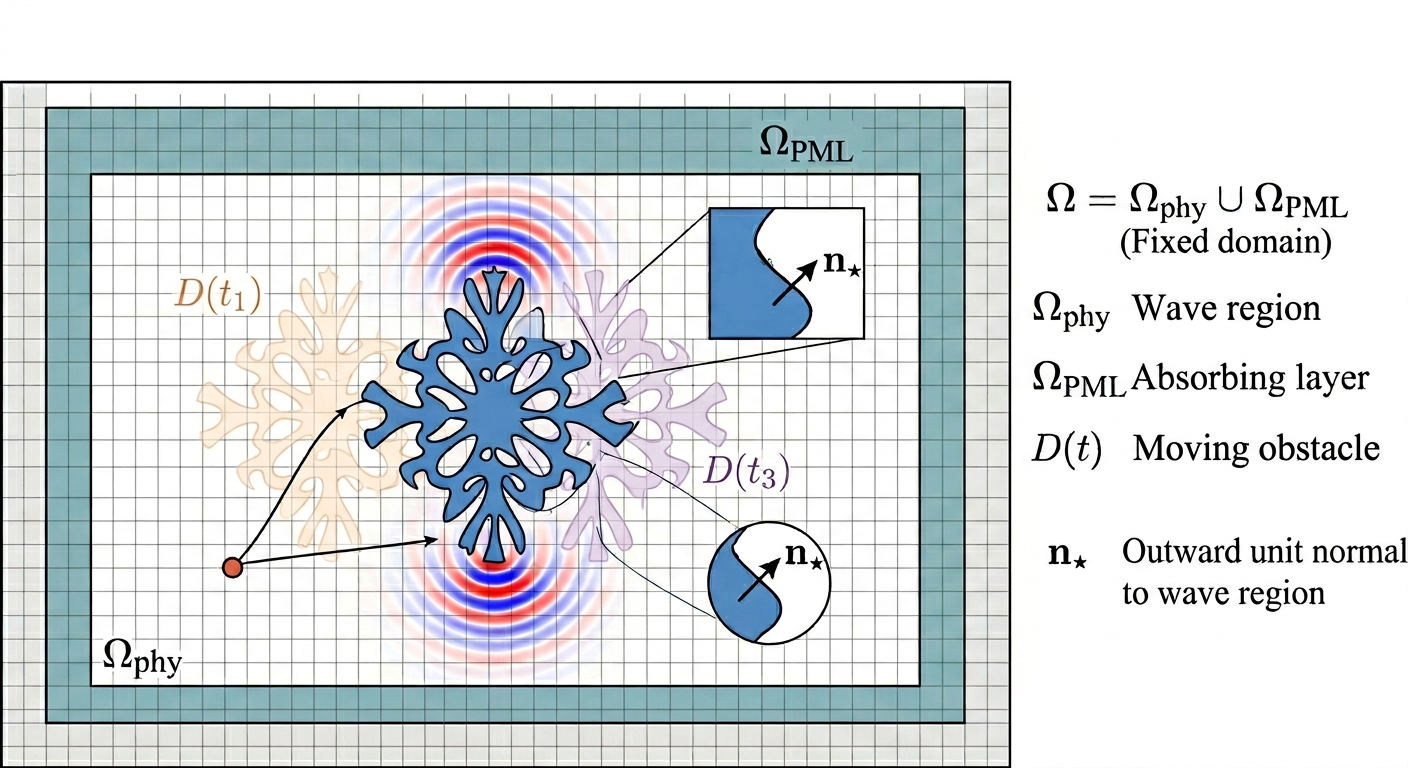}
	\end{center}
	\caption{Computational domain of the embedded
		moving-object model. The object remains inside the
		physical domain, while the outer collar is  a
		PML.}\label{fig:pml}
\end{figure}

We assume   that the moving object remains strictly inside
$\Omega_{\rm phy}$, away from the PML region . In particular, the PML
damping parameters vanish in a neighborhood of $\partial D(t)$.

To embed the moving object into the fixed domain, we introduce an
indicator $\psi(\bx,t)$ that approximates the characteristic function
of $\Omega\setminus\overline{D(t)}$
\begin{equation}
	\psi (\mathbf{x}, t) =
	\left\lbrace
	\begin{aligned}
		 & 1, &  & \mathbf{x} \in \Omega \setminus \overline{D(t)}, \\
		 & 0, &  & \mathbf{x} \in D(t).
	\end{aligned}
	\right.
\end{equation}
In the sense of distributions, the interface Dirac measure associated
with $\partial D(t)$ is given by the spatial gradient of
characteristic function $\psi$:
\begin{equation*}
	W(\mathbf{x}, t) = \delta_{\partial D(t)} =  -\nabla \psi \cdot
	\mathbf{n}_\star = |\nabla \psi|,
\end{equation*}
In practice, the presence of a sharp interface can give rise to
numerical instabilities. To improve
regularity and enhance numerical robustness, we replace the
discontinuous characteristic function $\psi$ with a smooth
approximation, $\psi_\varepsilon$, defined by
\begin{equation}\label{eq:smooth-profile}
	\psi_\vep(\bx,t)
	= \frac{1}{e^{6r(\bx,t)/\vep}+1},
	\qquad
	W_\vep(\bx,t)=|\nabla\psi_\vep(\bx,t)|,
\end{equation}
where $r(\bx,t)$ is the signed distance to $\partial D(t)$, negative
in the wave region and positive inside the object, and $\vep>0$ is
the diffuse-interface thickness.

For the sound-soft boundary
condition, the resulting PML-diffuse-embedded (\PMLDE{}) system is given by \cite{Gu2026}
\begin{equation}\label{eq:wave-pml-embedding-soft}
	\left\lbrace
	\begin{aligned}
		 & \psi_\varepsilon p_{tt} + a\psi_\varepsilon p_t +
		b\psi_\varepsilon p + \tfrac{W_\varepsilon}{\eta_d}
		p + (1-\psi_\varepsilon)
		(\alpha p_t + \beta p)
		= c^2 \nabla \cdot (\psi_\varepsilon \nabla p) + c
		\nabla \cdot (\psi_\varepsilon \bm{\lambda}),                          \\
		 & \bm{\lambda}_t + \varGamma_1 \bm{\lambda} = c \varGamma_2 \nabla p.
	\end{aligned}
	\right.
\end{equation}
Here $\eta_d$ penalizes the interfacial trace of $p$, while
$\alpha\ge 0$ and $\beta>0$ damp the fictitious field inside the
object \cite{Gu2026}. The interior damping is important in practice because the
embedded region should not sustain spurious oscillations while the
object moves across the Cartesian grid.

For the sound-hard boundary condition, the PML-DE system is given by the following
\begin{equation}\label{eq:wave-pml-embedding-hard}
	\left\lbrace
	\begin{aligned}
		 & \psi_\varepsilon p_{tt} + a \psi_\varepsilon p_t + b
		\psi_\varepsilon p - \tfrac{1 - \psi_\varepsilon}{\eta_n}\left(\nabla
		\psi_\varepsilon \cdot \nabla p - {\color{blue}a_{n_\star}(t)}\right)  \\
		 & \qquad + H(\hat{\psi} -
		\psi_\varepsilon) (\alpha p_t + \beta p) = c^2 \nabla \cdot
		(\psi_\varepsilon \nabla p) + c \nabla \cdot (\psi_\varepsilon
		\bm{\lambda}),                                                         \\
		 & \bm{\lambda}_t + \varGamma_1 \bm{\lambda} = c \varGamma_2 \nabla p,
	\end{aligned}
	\right.
\end{equation}
{Here $a_{n_\star}(t)$ denotes the prescribed normal acceleration of the moving boundary in the embedded sound-hard condition. In all sound-hard numerical tests reported in Section~\ref{sec:numerical}, the object undergoes rigid translation with constant velocity; hence $a_{n_\star}(t)=0$, and the acceleration correction in \eqref{eq:wave-pml-embedding-hard} drops out.} In general, this term must be retained.
$H(\cdot)$ denotes the Heaviside function and $\hat{\psi}$ is a
user-defined threshold.  $H(\hat{\psi}-\psi_\varepsilon)$
restricts the damping mechanism to the interior of the object,
maintaining a finite separation from the diffuse interface. In
contrast to the sound-soft formulation, the penalty is  not
applied at the interface, but only serves to regularize the
fictitious interior field. {Thus the present sound-hard, numerical validation focuses on the zero-acceleration, uniformly translating object; accelerated motion with sound-hard boundary conditions requires a separate analysis and is thus left to our future work.}

Because \(D(t)\) stays away from \(\Omega_{\mathrm{PML}}\), the
diffuse-interface
region and the PML damping paraters have disjoint support, we assume
\begin{equation}\label{eq:embedding-support-condition}
	\xi_j(1-\psi_\varepsilon)=0,\qquad \xi_j W_\varepsilon=0,
	\quad j=1,2.
\end{equation}
Equivalently, all PML coefficients derived from \(\xi_1\) and \(\xi_2\),
in particular
\(\varGamma_1,\varGamma_2,\widetilde{\varGamma}_1,a,\) and \(b\),
satisfy the same support condition.

	{For the energy proof,} we introduce the same auxiliary variables as in
Section~\ref{sec:PML-wave}, namely $q := p_t + ap$ and
$\bm{\chi} := c\nabla p + \bm{\lambda}$, and define the weighted
coefficient functions
$\Theta := \psi_\varepsilon b + \tfrac{W_\varepsilon}{\eta_d} +
	(1-\psi_\varepsilon)\beta$ and
$\Theta_t := (b-\beta)\partial_t \psi_\varepsilon +
	\tfrac{1}{\eta_d} \partial_t
	W_\varepsilon$.

\begin{thm}\label{thm:wave-pml-embedding-energy-cont}
	Under the support condition, \eqref{eq:embedding-support-condition},
	every strong solution of \eqref{eq:wave-pml-embedding-soft} satisfies
	\begin{equation}\label{eq:wave-pml-embedding-energy-con}
		\frac{d\mathcal{E}_{\rm embed}}{dt}
		= -\mathcal{D}_{\rm embed} + \mathcal{R}_{\rm embed},
	\end{equation}
	where
	\begin{equation*}
		\begin{aligned}
			\mathcal{E}_{\rm embed}(t)
			 & = \frac12\Bigl(
			\|q\|_{\psi_\varepsilon}^2 + \|\bm{\chi}\|_{\psi_\varepsilon}^2
			+ \|\bm{\lambda}\|_{\psi_\varepsilon a^{-1}\varGamma_1}^2
			\Bigr) + \frac12\|p\|_{\Theta}^2,                                                                                                                                      \\
			\mathcal{D}_{\rm embed}(t)
			 & = \|p\|_{ab\psi_\varepsilon}^2 + \|p_t\|_{\alpha(1-\psi_\varepsilon)}^2
			+ \|\bm{\chi}\|_{\psi_\varepsilon\varGamma_1(I+a^{-1}\varGamma_1)}^2
			+ c^2\|\nabla p\|_{\psi_\varepsilon a^{-1}b}^2,                                                                                                                        \\
			\mathcal{R}_{\rm embed}(t)
			 & = \frac12 \int_\Omega \partial_t \psi_\varepsilon \left(q^2 + |\bm{\chi}|^2 + \bm{\lambda}^\top a^{-1} \varGamma_1 \bm{\lambda} \right) + \Theta_t p^2 d\mathbf{x}.
		\end{aligned}
	\end{equation*}
\end{thm}
\begin{proof}
	Multiplying the first equation of \eqref{eq:wave-pml-embedding-soft} by
	$q$ and using \eqref{eq:embedding-support-condition}, we obtain
	\begin{equation}\label{eq:wave-pml-embedding-ene-p}
		\frac12\bigl(\psi_\varepsilon,\partial_t|q|^2\bigr)
		+ \frac12\bigl(\Theta,\partial_t|p|^2\bigr)
		+ \|p\|_{ab\psi_\varepsilon}^2
		+ \|p_t\|_{\alpha(1-\psi_\varepsilon)}^2
		= -c(\bm{\chi},\nabla q)_{\psi_\varepsilon}.
	\end{equation}
	Adding $c\nabla p_t + c\varGamma_1\nabla p$ to the second equation of
	\eqref{eq:wave-pml-embedding-soft} gives
	\begin{equation}\label{eq:wave-pml-embedding-chi}
		\bm{\chi}_t + \varGamma_1\bm{\chi}
		= c(\partial_t + \widetilde{\varGamma}_1)\nabla p,
		\quad
		\widetilde{\varGamma}_1 := \varGamma_1 + \varGamma_2.
	\end{equation}
	Multiplying \eqref{eq:wave-pml-embedding-chi} by
	$\mathbf g := \psi_\varepsilon[(I+a^{-1}\varGamma_1)\bm{\chi} -
			ca^{-1}\varGamma_1\nabla p]$
	yields
	\begin{equation}\label{eq:wave-pml-embedding-ene-chi}
		\begin{aligned}
			 & \frac12\Bigl(\psi_\varepsilon,
			          \partial_t\bigl[|\bm{\chi}|^2
			          + \bm{\chi}^{\top}a^{-1}\varGamma_1\bm{\chi}
			          + c^2(\nabla p)^{\top}a^{-1}\varGamma_1\nabla p\bigr]\Bigr) \\
			 & \qquad
			+ \|\bm{\chi}\|_{\varGamma_1(I+a^{-1}\varGamma_1)}^2
			+ c^2\|\nabla p\|_{ a^{-1}b}^2                                        \\
			 & =
			c\bigl(\psi_\varepsilon,\partial_t(\bm{\chi}^{\top}a^{-1}\varGamma_1\nabla
			p)\bigr)
			+ c(\nabla q,\bm{\chi})_{\psi_\varepsilon},
		\end{aligned}
	\end{equation}
	where we have used
	$\varGamma_1\widetilde{\varGamma}_1=bI$ and
	$a^{-1}\varGamma_1(\varGamma_1+\widetilde{\varGamma}_1)
		+ \widetilde{\varGamma}_1 = aI$.
	Since
	\[
		\bm{\lambda}^{\top}a^{-1}\varGamma_1\bm{\lambda}
		= \bm{\chi}^{\top}a^{-1}\varGamma_1\bm{\chi}
		+ c^2(\nabla p)^{\top}a^{-1}\varGamma_1\nabla p
		- 2c\,\bm{\chi}^{\top}a^{-1}\varGamma_1\nabla p,
	\]
	adding \eqref{eq:wave-pml-embedding-ene-p} and
	\eqref{eq:wave-pml-embedding-ene-chi} gives
	\begin{equation}\label{eq:wave-pml-embedding-balance-pre}
		\begin{aligned}
			 & \frac12\bigl(\psi_\varepsilon,
			\partial_t[|q|^2 + |\bm{\chi}|^2
				+ \bm{\lambda}^{\top}a^{-1}\varGamma_1\bm{\lambda}]\bigr)
			+ \frac12\bigl(\psi_\varepsilon b + \tfrac{W_\varepsilon}{\eta_d} +
			(1-\psi_\varepsilon)\beta,\partial_t|p|^2\bigr) =
			-\mathcal{D}_{\rm embed}(t).
		\end{aligned}
	\end{equation}
	Finally, the weighted chain rule
	$\frac12(w,\partial_t|z|^2) = \frac12\frac{d}{dt}\|z\|_w^2 -
		\frac12 \int_\Omega w_t |z|^2 d\mathbf{x}$
	with $w=\psi_\varepsilon$ or $w=\Theta$ converts
	\eqref{eq:wave-pml-embedding-balance-pre} into
	\eqref{eq:wave-pml-embedding-energy-con}.
\end{proof}
We remark that the term, $\mathcal{R}_{\rm embed}(t)$, keeps track of the energy rate of change related to the motion of the moving object. It vanishes completely when the object is static. It is indefinite in general and nonzero only in the neighborhood of the moving boundary. Hence the strict dissipative property can be established only for the stationary-object.

\subsection{Formal asymptotic analysis of the \PMLDES}

\noindent \indent We now examine the relation between {\PMLDES},
\eqref{eq:wave-pml-embedding-soft}, and the sharp-interface PML system
\eqref{eq:wave_pml} as $\vep\to 0^+$. We analyze the
sound-soft case only. Because the moving object stays a positive distance
away from the PML collar, there exists a tubular neighborhood
$\omega_\delta(t):=\{\bx:\ |r(\bx,t)|<\delta\}$ of $\partial D(t)$ in
which the PML parameters vanish, that is,
\begin{equation*}
	\varGamma_1=\varGamma_2=0, \quad a=b=0 \quad \text{in } \omega_\delta(t).
\end{equation*}
Hence the interfacial layer at the boundary of the object is governed locally by the embedded wave
part of the model.

\paragraph{Outer expansion.}
On every compact subset of $\Omega\setminus\overline{D(t)}$ away from
$\partial D(t)$, one has $\psi_\vep\to 1$ and $W_\vep\to 0$ as
$\vep\to 0^+$. We therefore seek outer expansions of the solution as follows
\begin{equation}\label{eq:outer-expansion}
	p = p_0 + \vep p_1 + \cdots,
	\quad
	\bm{\lambda} = \bm{\lambda}_0 + \vep \bm{\lambda}_1 + \cdots.
\end{equation}
Substituting \eqref{eq:outer-expansion} into
\eqref{eq:wave-pml-embedding-soft} and collecting the leading-order
terms gives precisely the sharp PML system \eqref{eq:wave_pml} for
$(p_0,\bm{\lambda}_0)$ in the exterior region. In particular, on the
scattering side of the interface where $\varGamma_1=\varGamma_2=0$, the
leading-order equations reduce locally to
\begin{equation}\label{eq:outer-leading-interface}
	p_{0,tt}=c^2\Delta p_0 + c\nabla\cdot\bm{\lambda}_0,
	\quad
	\bm{\lambda}_{0,t}=0.
\end{equation}
Under the standard initialization $\bm{\lambda}_0(\bx,0)=0$ in
$\omega_\delta(t)$, we have $\bm{\lambda}_0\equiv 0$ there; so the
outer limit near the interface is simply the wave equation at the
exterior side.

\paragraph{Inner expansion.}
To resolve the diffuse layer, let $\mathbf{X}(s,t)$ be a smooth
parametrization of $\partial D(t)$ and write
\begin{equation*}
	\bx = \mathbf{X}(s,t) + r\,\bn_\star(s,t), \quad \bx \in \omega_\delta(t),
\end{equation*}
where $r=r(\bx,t)$ is the signed distance to $\partial D(t)$, negative
in the wave region and positive inside the object. Introducing the
stretched normal coordinate $z=r/\vep$, the differential operators take
the form
\begin{equation}\label{eq:inner-operators}
	\begin{aligned}
		\nabla     & = \frac1\vep \bn_\star\partial_z +
		\frac{1}{1+\vep z\ka}\nabla_s,                                           \\
		\Delta     & = \frac1{\vep^2}\partial_{zz}
		+ \frac1\vep\frac{\ka}{1+\vep z\ka}\partial_z
		+ \frac{1}{1+\vep z\ka}\nabla_s\cdot
		\left(\frac{1}{1+\vep z\ka}\nabla_s\right),                              \\
		\partial_t & = -\frac{\bv\cdot\bn_\star}{\vep}\partial_z + \frac{d}{dt},
	\end{aligned}
\end{equation}
where $\ka=\nabla_s\cdot\bn_\star$ is the curvature of the moving
interface and $\frac{d}{dt}$ denotes the effective time derivative in
$(z,t)$ coordinates.

We write the inner variables as
\begin{equation}\label{eq:inner-expansion}
	\bar p(s,z,t) = \bar p_0(s,z,t) + \vep \bar p_1(s,z,t)+\cdots,
	\quad
	\bar{\bm{\lambda}}(s,z,t)=\bar{\bm{\lambda}}_0(s,z,t)+\vep\bar{\bm{\lambda}}_1(s,z,t)+\cdots,
\end{equation}
with
\begin{equation*}
	\Psi(s,z,t)=\Psi_0(z)=\frac{1}{e^{6z}+1},
	\qquad
	W_\vep = -\frac1\vep \Psi_0'(z)+\mathcal{O}(1).
\end{equation*}
The matching conditions on the wave scattering side read
\begin{subequations}\label{eq:matching-conditions}
	\begin{align}
		\lim_{z\to -\infty}\bar p_0(s,z,t)
		                                                & = \lim_{r\to -0} p_0(s,r,t), \\
		\lim_{z\to -\infty}\partial_z^m \bar p_0(s,z,t) & = 0,
		\qquad m\ge 1.
	\end{align}
\end{subequations}
Because $\varGamma_1=\varGamma_2=0$ in $\omega_\delta(t)$, the inner
auxiliary field satisfies
\begin{equation}\label{eq:inner-lambda-leading}
	\partial_t \bar{\bm{\lambda}}_0 = 0.
\end{equation}
With the same zero initialization as above, we obtain
$\bar{\bm{\lambda}}_0\equiv 0$. Thus the leading-order inner problem for
$\bar p_0$ is exactly the one arising in the non-PML embedded wave
model.

If $\eta_d=\widehat{\eta}_d \varepsilon$, then the leading-order terms of
\eqref{eq:wave-pml-embedding-soft} in the inner layer are of order
$\mathcal{O}(\vep^{-2})$ and satisfy
\begin{equation}\label{eq:inner-leading-L1}
	\Psi_0(\bv\cdot\bn_\star)^2\partial_{zz}\bar p_0
	- c^2\partial_z\bigl(\Psi_0\partial_z\bar p_0\bigr)
	- \frac{1}{\widehat{\eta}_d} \Psi_0'\,\bar p_0 = 0.
\end{equation}
Multiplying \eqref{eq:inner-leading-L1} by $\bar p_0$, integrating over
$(-\infty,\infty)$, and using the matching conditions
\eqref{eq:matching-conditions}, we obtain
\begin{equation}\label{eq:inner-energy-identity}
	\int_{-\infty}^{\infty}
	\left[
		\bigl(c^2-(\bv\cdot\bn_\star)^2\bigr)\Psi_0 |\partial_z \bar p_0|^2
		+ \left(\frac12 (\bv\cdot\bn_\star)^2\Psi_0''
		-\frac{1}{\widehat{\eta}_d} \Psi_0'\right)|\bar p_0|^2
		\right]dz = 0.
\end{equation}
Since $\Psi_0>0$ and $\Psi_0'<0$, the sufficient conditions
$|\bv\cdot\bn_\star|<c$ and
$\frac12 (\bv\cdot\bn_\star)^2\Psi_0'' - \widehat L\Psi_0'\ge 0$
force $\bar p_0\equiv 0$. In particular,
\begin{equation}\label{eq:dirichlet-recovery-1}
	\lim_{r\to -0} p_0(s,r,t) = 0.
\end{equation}
Thus the leading-order outer solution satisfies the sound-soft boundary
condition.

If $\eta_d =\mathcal{O}(\vep^{k})$ with $k\ge 2$, then the dominant
balance yields directly
\begin{equation}\label{eq:inner-leading-Lk}
	-\Psi_0'\,\bar p_0 = 0,
\end{equation}
so again $\bar p_0\equiv 0$ and hence
\begin{equation}\label{eq:dirichlet-recovery-2}
	\lim_{r\to -0} p_0(s,r,t) = 0.
\end{equation}
Therefore, we recover the homogeneous Dirichlet boundary of $p$ at leading order.

\subsection{Structure-preserving semi-discretization of the \PMLDES}\label{subsec:pml-de-semidiscrete}

\noindent \indent We now turn to the temporal discretization of the {\PMLDES} and
employ the following semi-discrete scheme:

\begin{equation}\label{eq:wave-pml-embedding-leap}
	\left\lbrace
	\begin{aligned}
		 & \psi_\varepsilon^{n+1} D_t^2 p^n + a\psi_\varepsilon^{n+1} D_{2t} p^n
		+ b\psi_\varepsilon^{n+1} A_t^2 p^n +
		\tfrac{W^{n+1}_\varepsilon}{\eta_d} A_t^2 p^n                            \\
		 & \qquad + (1-\psi_\varepsilon^{n+1})\bigl(\alpha D_{2t} p^n +
		\beta p^{n+1}\bigr)
		= c^2 \nabla \cdot \bigl(\psi_\varepsilon^{n+1} \nabla A_t^2 p^n\bigr)
		+ c \nabla \cdot \bigl(\psi_\varepsilon^{n+1} A_t^2
		\bm{\lambda}^n\bigr),                                                    \\
		 & D_t^+ \bm{\lambda}^n + \varGamma_1 A_t^+ \bm{\lambda}^n = c
		\varGamma_2 A_t^+\nabla p^n.
	\end{aligned}
	\right.
\end{equation}
We denote
$q^n := D_t^+p^n + aA_t^+p^n$, $\bm{\chi}^n := c\nabla p^n +
	\bm{\lambda}^n$,
and introduce shorthand notations
$\Theta^{n+1} := b\psi_\varepsilon^{n+1}+
	\tfrac{W_\varepsilon^{n+1}}{\eta_d}$ and
$\dot\Theta^{n+1} := bD_t^-\psi_\varepsilon^{n+1}+
	\tfrac{1}{\eta_d}D_t^-W_\varepsilon^{n+1}$.

	{As in the continuous case, the term \(\mathcal{R}_{\rm embed}^n\) below is generated by the time-dependent embedding profile; it disappears for a stationary object.}
\begin{thm}\label{thm:wave-pml-embedding-energy-discrete}
	Assume that $\xi_i (1-\psi_\varepsilon^n)=0$ and
	$\xi_i W_\varepsilon^n=0$ for all
	$n\ge 0$. Then, for every
	$n\ge 1$, scheme \eqref{eq:wave-pml-embedding-leap} satisfies
	\begin{equation}\label{eq:wave-pml-embedding-energy-discrete}
		D_t^-\mathcal{E}_{\rm embed}^{n+\frac12}
		= -\mathcal{D}_{\rm embed}^n + \mathcal{R}_{\rm embed}^n,
	\end{equation}
	where
	\begin{equation*}
		\begin{aligned}
			\mathcal{E}_{\rm embed}^{n+\frac12}
			 & = \frac12\Bigl(
			\|q^n\|_{\psi_\varepsilon^{n+1}}^2
			+ \|A_t^+\bm{\chi}^n\|_{\psi_\varepsilon^{n+1}}^2
			+
			\|A_t^+\bm{\lambda}^n\|_{a^{-1}\varGamma_1}^2
			\Bigr) + \frac12\|A_t^+p^n\|_{\Theta^{n+1}}^2
			+ \frac12\bigl(\beta(1-\psi_\varepsilon^{n+1}),A_t^+|p^n|^2\bigr),                                                                                                                                         \\
			\mathcal{D}_{\rm embed}^n
			 & = \|A_t^2p^n\|_{ab}^2
			+ \|D_{2t}p^n\|_{(\alpha+\tau\beta)(1-\psi_\varepsilon^{n+1})}^2
			+ \|A_t^2\bm{\chi}^n\|_{\varGamma_1(I+a^{-1}\varGamma_1)}^2
			+ c^2\|A_t^2\nabla p^n\|_{a^{-1}b}^2,                                                                                                                                                                      \\
			\mathcal{R}_{\rm embed}^n
			 & = \frac12 \int_\Omega  D_t^- \psi_\varepsilon^{n+1} \left( |q^{n-1}|^2 + |A_t^+ \bm{\chi}^{n-1}|^2 + (A_t^+ \bm{\lambda}^{n-1})^\top a^{-1} \varGamma_1 (A_t^+ \bm{\lambda}^{n-1})  \right) d\mathbf{x} \\
			 & \qquad + \frac12 \int_\Omega \dot{\Theta}^{n+1} |A_t^+ p^{n-1}|^2 + \beta D_t^- \psi_\varepsilon^{n+1} \cdot A_t^+|p^{n-1}|^2 d\mathbf{x}.
		\end{aligned}
	\end{equation*}
\end{thm}

\begin{proof}
	\enlargethispage{2\baselineskip}
	Using $D_t^-q^n = D_t^2p^n + aD_{2t}p^n$ and
	$A_t^-q^n = D_{2t}p^n + aA_t^2p^n$, we rewrite the first equation of
	\eqref{eq:wave-pml-embedding-leap}  into
	\begin{equation}\label{eq:wave-pml-discrete-p-rewrite}
		\psi_\varepsilon^{n+1}D_t^-q^n
		+ \Theta^{n+1}A_t^2p^n
		+ (1-\psi_\varepsilon^{n+1})(\alpha D_{2t}p^n + \beta p^{n+1})
		= c\nabla\!\cdot\!\bigl(\psi_\varepsilon^{n+1}A_t^2\bm{\chi}^n\bigr).
	\end{equation}
	Multiplying \eqref{eq:wave-pml-discrete-p-rewrite} with $A_t^-q^n$ and using
	$a(1-\psi_\varepsilon^n)=0$ and $aW_\varepsilon^n=0$ gives
	\begin{equation}\label{eq:wave-pml-discrete-p-energy}
		\begin{aligned}
			 & \frac12\bigl(\psi_\varepsilon^{n+1},D_t^-|q^n|^2\bigr)
			+
			\frac12\bigl(b\psi_\varepsilon^{n+1}+\tfrac{W_\varepsilon^{n+1}}{\eta_d},D_t^-|A_t^+p^n|^2\bigr)
			\\
			 & \qquad +
			\frac12\bigl(\beta(1-\psi_\varepsilon^{n+1}),D_t^-A_t^+|p^n|^2\bigr)
			+ \|A_t^2p^n\|_{ab}^2
			+ \|D_{2t}p^n\|_{(\alpha+\tau\beta)(1-\psi_\varepsilon^{n+1})}^2 \\
			 & \qquad = -c\bigl(A_t^2\bm{\chi}^n,(D_{2t}+aA_t^2)\nabla
			p^n\bigr)_{\psi_\varepsilon^{n+1}},
		\end{aligned}
	\end{equation}
	where we used
	$p^{n+1}D_{2t}p^n = \frac12 D_t^-A_t^+|p^n|^2 + \tau|D_{2t}p^n|^2$.

	Next, the second equation of \eqref{eq:wave-pml-embedding-leap} gives
	$D_t^-A_t^+\bm{\chi}^n + \varGamma_1A_t^2\bm{\chi}^n
		= c(D_{2t}+\widetilde{\varGamma}_1A_t^2)\nabla p^n$.
	Multiplying this equation with
	$\mathbf g^n :=
		\psi_\varepsilon^{n+1}[(I+a^{-1}\varGamma_1)A_t^2\bm{\chi}^n
			- ca^{-1}\varGamma_1A_t^2\nabla p^n]$
	and using
	$\varGamma_1\widetilde{\varGamma}_1=bI$ together with
	$a^{-1}\varGamma_1(\varGamma_1+\widetilde{\varGamma}_1)
		+ \widetilde{\varGamma}_1 = aI$, we find
	\begin{equation}\label{eq:wave-pml-discrete-chi-energy-pre}
		\begin{aligned}
			 & \frac12\Bigl(\psi_\varepsilon^{n+1},
			D_t^-\bigl[|A_t^+\bm{\chi}^n|^2
				     + (A_t^+\bm{\chi}^n)^\top a^{-1}\varGamma_1A_t^+\bm{\chi}^n
				     + c^2(A_t^+\nabla p^n)^\top a^{-1}\varGamma_1A_t^+\nabla
				     p^n\bigr]\Bigr) \\
			 & \qquad +
			\|A_t^2\bm{\chi}^n\|_{\varGamma_1(I+a^{-1}\varGamma_1)}^2
			+ c^2\|A_t^2\nabla p^n\|_{a^{-1}b}^2                             \\
			 & \qquad = c\Bigl(\psi_\varepsilon^{n+1},
			D_t^-\bigl[(A_t^+\bm{\chi}^n)^\top a^{-1}\varGamma_1A_t^+\nabla
				     p^n\bigr]\Bigr)
			+ c\bigl(A_t^2\bm{\chi}^n,(D_{2t}+aA_t^2)\nabla
			p^n\bigr)_{\psi_\varepsilon^{n+1}}.
		\end{aligned}
	\end{equation}
	Since $A_t^+\bm{\lambda}^n = A_t^+\bm{\chi}^n - cA_t^+\nabla p^n$,
	we also have
	\[
		(A_t^+\bm{\lambda}^n)^\top a^{-1}\varGamma_1A_t^+\bm{\lambda}^n
		= (A_t^+\bm{\chi}^n)^\top a^{-1}\varGamma_1A_t^+\bm{\chi}^n
		+ c^2(A_t^+\nabla p^n)^\top a^{-1}\varGamma_1A_t^+\nabla p^n
		- 2c(A_t^+\bm{\chi}^n)^\top a^{-1}\varGamma_1A_t^+\nabla p^n.
	\]
	Therefore \eqref{eq:wave-pml-discrete-chi-energy-pre} reduces to
	\begin{equation}\label{eq:wave-pml-discrete-chi-energy}
		\begin{aligned}
			 & \frac12\Bigl(\psi_\varepsilon^{n+1},
			D_t^-\bigl[|A_t^+\bm{\chi}^n|^2
				     + (A_t^+\bm{\lambda}^n)^\top
				     a^{-1}\varGamma_1A_t^+\bm{\lambda}^n\bigr]\Bigr)    \\
			 & \qquad +
			\|A_t^2\bm{\chi}^n\|_{\varGamma_1(I+a^{-1}\varGamma_1)}^2
			+ c^2\|A_t^2\nabla p^n\|_{a^{-1}b}^2                      \\
			 & \qquad = c\bigl(A_t^2\bm{\chi}^n,(D_{2t}+aA_t^2)\nabla
			p^n\bigr)_{\psi_\varepsilon^{n+1}}.
		\end{aligned}
	\end{equation}
	Adding \eqref{eq:wave-pml-discrete-p-energy} and
	\eqref{eq:wave-pml-discrete-chi-energy}, we notice that the coupling terms cancel
	and obtain
	\begin{equation}\label{eq:wave-pml-discrete-energy-pre}
		\begin{aligned}
			 & \frac12\Bigl(\psi_\varepsilon^{n+1},
			D_t^-\bigl[|q^n|^2 + |A_t^+\bm{\chi}^n|^2
				     + (A_t^+\bm{\lambda}^n)^\top
				     a^{-1}\varGamma_1A_t^+\bm{\lambda}^n\bigr]\Bigr)        \\
			 & \qquad + \frac12\bigl(\Theta^{n+1},D_t^-|A_t^+p^n|^2\bigr)
			+
			\frac12\bigl(\beta(1-\psi_\varepsilon^{n+1}),D_t^-A_t^+|p^n|^2\bigr)
			\\
			 & \qquad = -\mathcal{D}_{\rm embed}^n.
		\end{aligned}
	\end{equation}
	To convert \eqref{eq:wave-pml-discrete-energy-pre} into {the stated weighted energy balance}, we use
	the discrete product rule
	$\,(w^{n+1},D_t^-r^n) = D_t^-(w^{n+1},r^n) - (D_t^-w^{n+1},r^{n-1})$.
	In particular,
	\begin{align*}
		\frac12\bigl(\psi_\varepsilon^{n+1},D_t^-|q^n|^2\bigr)
		 & = D_t^-\frac12\|q^n\|_{\psi_\varepsilon^{n+1}}^2
		- \frac12 \int_\Omega D_t^- \psi_\varepsilon^{n+1} |q^{n-1}|^2 d\mathbf{x}, \\
		\frac12\bigl(\Theta^{n+1},D_t^-|A_t^+p^n|^2\bigr)
		 & = D_t^-\frac12\|A_t^+p^n\|_{\Theta^{n+1}}^2
		- \frac12 \int_\Omega  \dot{\Theta}^{n+1}  |A_t^+p^{n-1}|^2 d\mathbf{x},    \\
		\frac12\bigl(\beta(1-\psi_\varepsilon^{n+1}),D_t^-A_t^+|p^n|^2\bigr)
		 & = D_t^-\frac12\bigl(\beta(1-\psi_\varepsilon^{n+1}),A_t^+|p^n|^2\bigr)
		+ \frac12\bigl(\beta D_t^-\psi_\varepsilon^{n+1},A_t^+|p^{n-1}|^2\bigr).
	\end{align*}
	Applying the same product rule to the remaining weighted terms in
	\eqref{eq:wave-pml-discrete-energy-pre} yields
	\eqref{eq:wave-pml-embedding-energy-discrete}.
\end{proof}

\begin{rem}
	If the embedding profile is time-independent,
	$D_t^-\psi_\varepsilon^{n+1} = D_t^-W_\varepsilon^{n+1} = 0$,
	the remainder term,
	$\mathcal{R}_{\rm embed}^n$, vanishes, and the proposed scheme
	preserves an energy dissipation property in the case of a static object.
\end{rem}

\subsection{Fully discrete adaptive finite difference scheme}

\noindent \indent We now discretize the spatial operators on a Cartesian mesh in a way
that remains both compact and compatible with the energy property of
	{\S~\ref{subsec:pml-de-semidiscrete}}. The key point is that the discrete Laplace operator and
its companion gradient--divergence pair must match. For this
reason, we place the scalar variable $p_h$ at cell centers and the two
components of the auxiliary vector $\bm{\lambda}_h$ on a staggered edge
grid, respectively. This yields a compact conservative stencil for the weighted
Laplace operator and, at the same time, an exact discrete
summation-by-parts identity.

Let $\Omega=(x_L,x_R)\times (y_L,y_R)$, and let $N_x, \ N_y\in\mathbb{N}$.
Set $h_x=(x_R-x_L)/N_x$ and $h_y=(y_R-y_L)/N_y$. We define cell
centers at
\[
	x_i=x_L+\Bigl(i+\frac12\Bigr)h_x,\quad
	y_j=y_L+\Bigl(j+\frac12\Bigr)h_y,
	\quad 0\le i\le N_x-1,\quad 0\le j\le N_y-1,
\]
and introduce the horizontal and vertical edge locations at
\[
	x_{i+\frac12}=x_L+(i+1)h_x,\quad
	y_{j+\frac12}=y_L+(j+1)h_y.
\]
The pressure, $p_h=(p_{i,j})$, is stored at cell centers, whereas the
first and second components of $\bm{\lambda}_h$ are stored at the
horizontal and vertical edge grids,
\[
	\lambda_{1,h}=(\lambda_{1,i+\frac12,j}),
	\quad
	\lambda_{2,h}=(\lambda_{2,i,j+\frac12}).
\]
For a cell-centered scalar grid function $u_h$ and an edge-centered
vector grid function
$\bm{v}_h=(v_{1,h},v_{2,h})$, we define
\begin{equation}
	\begin{aligned}
		 & (D_x^+u_h)_{i+\frac12,j}=\frac{u_{i+1,j}-u_{i,j}}{h_x},
		 &                                                                       & (D_y^+u_h)_{i,j+\frac12}=\frac{u_{i,j+1}-u_{i,j}}{h_y},           \\
		 & (D_x^-v_{1,h})_{i,j}=\frac{v_{1,i+\frac12,j}-v_{1,i-\frac12,j}}{h_x},
		 &                                                                       &
		(D_y^-v_{2,h})_{i,j}=\frac{v_{2,i,j+\frac12}-v_{2,i,j-\frac12}}{h_y},                                                                        \\
		 & (A_x^+ u_h)_{i+\frac{1}{2}, j} = \frac{u_{i+1, j} + u_{i, j}}{2},
		 &                                                                       & (A_y^+ u_h)_{i, j+\frac{1}{2}} = \frac{u_{i, j+1} + u_{i, j}}{2}.
	\end{aligned}
\end{equation}
Then, we introduce the staggered gradient and divergence
\[
	\nabla_h^+u_h := (D_x^+u_h,D_y^+u_h)^\top,
	\quad
	\nabla_h^-\!\cdot\!\bm{v}_h := D_x^-v_{1,h}+D_y^-v_{2,h}.
\]

For a given function $\omega(\mathbf{x})$, we define the weighted
discrete divergence and weighted discrete Laplacian operator,
\begin{equation}\label{eq:weighted-discrete-operators}
	\mathcal
	D_h^{\omega}\bm{v}_h:=\nabla_h^-\!\cdot\!(\omega
	\bm{v}_h),
	\quad
	\mathcal L_h^{\omega}u_h:=\mathcal D_h^{\omega}(\nabla_h^+u_h).
\end{equation}
Here, we define
\begin{equation}
	(\omega  \bm{v}_h)_{1, i+\frac12, j} =
	\omega_{i+\frac12, j} v_{1, i+\frac12, j}, \quad (\omega
	\bm{v}_h)_{2, i, j+\frac12} =
	\omega_{i, j + \frac12} v_{1, i, j+\frac12}.
\end{equation}
In componentwise forms,
\[
	(\mathcal L_h^{\omega}u_h)_{i,j}
	= \frac{\omega_{i+\frac12,j}(u_{i+1,j}-u_{i,j})
		-\omega_{i-\frac12,j}(u_{i,j}-u_{i-1,j})}{h_x^2}
	+ \frac{\omega_{i,j+\frac12}(u_{i,j+1}-u_{i,j})
		-\omega_{i,j-\frac12}(u_{i,j}-u_{i,j-1})}{h_y^2}.
\]
For cell-centered quantities, we use the inner product,
\[
	(u_h,v_h)_{h}=h_xh_y\sum_{i=0}^{N_x-1}\sum_{j=0}^{N_y-1}u_{i,j}v_{i,j},
\]
whereas for edge-centered vectors, we set
\[
	(\bm{v}_h,\bm{w}_h)_{e}
	= h_xh_y\sum_{i,j}\left(v_{1,i+\frac12,j}w_{1,i+\frac12,j}
	+ v_{2,i,j+\frac12}w_{2,i,j+\frac12}\right).
\]
Weighted cell and edge inner products are defined in the obvious way.
With the homogeneous outer-boundary closure, the staggered operators
satisfy the summation-by-parts identities,
\begin{equation}\label{eq:sbp-identity}
	(\mathcal D_h^{\omega}\bm{v}_h,u_h)_{h}
	= -(\bm{v}_h,\nabla_h^+u_h)_{\omega,e},
	\quad
	(\mathcal L_h^{\omega}u_h,v_h)_{h}
	= -(\nabla_h^+u_h,\nabla_h^+v_h)_{\omega,e}.
\end{equation}

Denote the discrete fields at
time level $t^n$ by $p_h^n$, $\bm{\lambda}_h^n$. The fully discrete counterpart of
\eqref{eq:wave-pml-embedding-leap} is
\begin{equation}\label{eq:wave-pml-embedding-fully-discrete}
	\left\lbrace
	\begin{aligned}
		 & \psi_\varepsilon^{n+1} D_t^2 p_h^n + a\psi_\varepsilon^{n+1}
		D_{2t} p_h^n
		+ b\psi_\varepsilon^{n+1} A_t^2 p_h^n +
		\tfrac{W_\varepsilon^{n+1}}{\eta_d}
		A_t^2 p_h^n                                                      \\
		 & \qquad + (1-\psi_\varepsilon^{n+1})\bigl(\alpha D_{2t}p_h^n +
		\beta p_h^{n+1}\bigr)
		= c^2\mathcal L_h^{\psi_\varepsilon^{n+1}}(A_t^2 p_h^n)
		+ c\mathcal D_h^{\psi_\varepsilon^{n+1}}(A_t^2\bm{\lambda}_h^n), \\
		 & D_t^+ \bm{\lambda}_h^n + \varGamma_{1} A_t^+\bm{\lambda}_h^n
		= c\varGamma_{2} A_t^+\nabla_h^+ p_h^n.
	\end{aligned}
	\right.
\end{equation}
As in the semi-discrete analysis, we introduce auxiliary fields
\begin{equation}
	q_h^n:=D_t^+p_h^n+aA_t^+p_h^n,
	\quad
	\bm{\chi}_h^n:=c\nabla_h^+p_h^n+\bm{\lambda}_h^n,
\end{equation}
and shorthand notations
\begin{equation}
	\Theta^{n+1}_h:=b\psi_\varepsilon^{n+1}+ \tfrac{1}{\eta_d}
	W_\varepsilon^{n+1},
	\quad
	\dot\Theta^{n+1}_h:=b
	D_t^-\psi_\varepsilon^{n+1}+\tfrac{1}{\eta_d}D_t^-W_\varepsilon^{n+1}.
\end{equation}

\begin{thm}\label{thm:wave-pml-embedding-energy-fully-discrete}
	Assume that the staggered operators satisfy
	\eqref{eq:sbp-identity}, $\xi_i(1-\psi_\varepsilon^n)=0$ and
	$\xi_i W_\varepsilon^n=0$ for all $n\ge 0$.
	Then, for every $n\ge 1$, the fully discrete scheme,
	\eqref{eq:wave-pml-embedding-fully-discrete}, satisfies
	\begin{equation}\label{eq:wave-pml-embedding-energy-fully-discrete}
		D_t^-\mathcal{E}_{h,{\rm embed}}^{n+\frac12}
		= -\mathcal{D}_{h,{\rm embed}}^n + \mathcal{R}_{h,{\rm embed}}^n,
	\end{equation}
	where
	\begin{equation*}
		\begin{aligned}
			\mathcal{E}_{h,{\rm embed}}^{n+\frac12}
			 & = \frac12\Bigl(
			\|q_h^n\|_{\psi_\varepsilon^{n+1},h}^2
			+ \|A_t^+\bm{\chi}_h^n\|_{\psi_\varepsilon^{n+1},e}^2
			+
			\|A_t^+\bm{\lambda}_h^n\|_{a^{-1}\varGamma_1,e}^2
			\Bigr)                                                                                                                                                                                        \\
			 & \quad + \frac12\|A_t^+p_h^n\|_{\Theta_h^{n+1},h}^2
			+
			\frac12\bigl(\beta(1-\psi_\varepsilon^{n+1}),A_t^+|p_h^n|^2\bigr)_{h},
			\\
			\mathcal{D}_{h,{\rm embed}}^n
			 & = \|A_t^2p_h^n\|_{a b,h}^2
			+ \|D_{2t}p_h^n\|_{(\alpha+\tau\beta)(1-\psi_\varepsilon^{n+1}),h}^2                                                                                                                          \\
			 & \quad +
			\|A_t^2\bm{\chi}_h^n\|_{\varGamma_1(I+a^{-1}\varGamma_{1}),e}^2
			+ c^2\|A_t^2\nabla_h^+p_h^n\|_{a^{-1}b,e}^2,
			\\
			\mathcal{R}_{h,{\rm embed}}^n
			 & = \frac12 \Bigl( \left(q_h^{n-1}, D_t^- \psi_\varepsilon^{n+1} q_h^{n-1}\right)_h + \left(A_t^+ \bm{\chi}_h^{n-1}, D_t^- \psi_\varepsilon^{n+1} A_t^+ \bm{\chi}_h^{n-1}\right)_e           \\
			 & \qquad + \left(A_t^+ \bm{\lambda}_h^{n-1}, D_t^- \psi_\varepsilon^{n+1} a^{-1} \varGamma_1 A_t^+ \bm{\lambda}_h^{n-1}\right)_e + (A_t^+ p_h^{n-1}, \dot{\Theta}_h^{n+1} A_t^+ p_h^{n-1})_h \\
			 & \qquad + \bigl(\beta
			D_t^-\psi_\varepsilon^{n+1},A_t^+|p_h^{n-1}|^2\bigr)_h  \Bigr).
		\end{aligned}
	\end{equation*}
\end{thm}

\begin{proof}
	Rewriting the first equation of
	\eqref{eq:wave-pml-embedding-fully-discrete} by means of
	$q_h^n=D_t^+p_h^n+aA_t^+p_h^n$ and
	$\bm{\chi}_h^n=c\nabla_h^+p_h^n+\bm{\lambda}_h^n$ gives
	\begin{equation}\label{eq:wave-pml-discrete-p-rewrite-h}
		\psi_\varepsilon^{n+1}D_t^-q_h^n
		+ \Theta_h^{n+1}A_t^2p_h^n
		+ (1-\psi_\varepsilon^{n+1})(\alpha D_{2t}p_h^n + \beta p_h^{n+1})
		= c\mathcal D_h^{\psi_\varepsilon^{n+1}}(A_t^2\bm{\chi}_h^n).
	\end{equation}
	Multiplying \eqref{eq:wave-pml-discrete-p-rewrite-h} by
	$A_t^-q_h^n=D_{2t}p_h^n+aA_t^2p_h^n$ and using
	\eqref{eq:sbp-identity}, we obtain
	\begin{equation}\label{eq:wave-pml-discrete-p-energy-h}
		\begin{aligned}
			 & \frac12\bigl(\psi_\varepsilon^{n+1},D_t^-|q_h^n|^2\bigr)_{h}
			+ \frac12\bigl(\Theta_h^{n+1},D_t^-|A_t^+p_h^n|^2\bigr)_{h}        \\
			 & \qquad +
			\frac12\bigl(\beta(1-\psi_\varepsilon^{n+1}),D_t^-A_t^+|p_h^n|^2\bigr)_{h}
			+ \|A_t^2p_h^n\|_{ab,h}^2 +
			\|D_{2t}p_h^n\|_{(\alpha+\tau\beta)(1-\psi_\varepsilon^{n+1}),h}^2 \\
			 & \qquad = -c\bigl(A_t^2\bm{\chi}_h^n,
			(D_{2t}+a
			A_t^2)\nabla_h^+p_h^n\bigr)_{\psi_\varepsilon^{n+1},e},
		\end{aligned}
	\end{equation}
	where we used
	$p_h^{n+1}D_{2t}p_h^n=
		\frac12D_t^-A_t^+|p_h^n|^2+\tau|D_{2t}p_h^n|^2$.

	Next, combining the second equation of
	\eqref{eq:wave-pml-embedding-fully-discrete} with the definition of
	$\bm{\chi}_h^n$ yields
	\begin{equation}\label{eq:wave-pml-discrete-chi-rewrite-h}
		D_t^-A_t^+\bm{\chi}_h^n + \varGamma_{1} A_t^2\bm{\chi}_h^n
		= c\bigl(D_{2t}+\widetilde{\varGamma}_{1} A_t^2\bigr)
		\nabla_h^+p_h^n.
	\end{equation}
	Multiplying \eqref{eq:wave-pml-discrete-chi-rewrite-h} by
	\[
		\bm g_h^n:=\psi_\varepsilon^{n+1}\Bigl[
			\bigl(I+a^{-1}\varGamma_1\bigr)A_t^2\bm{\chi}_h^n
			- c a^{-1}\varGamma_{1} A_t^2\nabla_h^+p_h^n
			\Bigr],
	\]
	and using
	$\varGamma_1\widetilde{\varGamma}_1=b I$
	together with
	$a^{-1}\varGamma_{1}
		(\varGamma_{1}+\widetilde{\varGamma}_{1})
		+\widetilde{\varGamma}_{1}=a I$,
	we have
	\begin{equation}\label{eq:wave-pml-discrete-chi-energy-pre-h}
		\begin{aligned}
			 & \frac12\Bigl(\psi_\varepsilon^{n+1},
			D_t^-\bigl[|A_t^+\bm{\chi}_h^n|^2
				     +(A_t^+\bm{\chi}_h^n)^\top
				     a^{-1}\varGamma_1 A_t^+\bm{\chi}_h^n +
				     c^2(A_t^+\nabla_h^+p_h^n)^\top
				     a^{-1}\varGamma_1 A_t^+\nabla_h^+p_h^n
				     \bigr]\Bigr)_{e} \\
			 & \qquad +
			\|A_t^2\bm{\chi}_h^n\|_{\varGamma_1
				                      (I+a^{-1}\varGamma_1),e}^2 +
			c^2\|A_t^2\nabla_h^+p_h^n\|_{a^{-1}b,e}^2
			\\
			 & \qquad = c\Bigl(\psi_\varepsilon^{n+1},
			D_t^-\bigl[(A_t^+\bm{\chi}_h^n)^\top
				     a^{-1}\varGamma_{1}
				     A_t^+\nabla_h^+p_h^n\bigr]\Bigr)_{e} + c\bigl(A_t^2\bm{\chi}_h^n,
			(D_{2t}+aA_t^2)\nabla_h^+p_h^n\bigr)_{\psi_\varepsilon^{n+1},e}.
		\end{aligned}
	\end{equation}
	Since
	\[
		A_t^+\bm{\lambda}_h^n=A_t^+\bm{\chi}_h^n-cA_t^+\nabla_h^+p_h^n,
	\]
	the quadratic identity
	\[
		(A_t^+\bm{\lambda}_h^n)^\top
		a^{-1}\varGamma_1 A_t^+\bm{\lambda}_h^n
		= (A_t^+\bm{\chi}_h^n)^\top
		a^{-1}\varGamma_{1} A_t^+\bm{\chi}_h^n
		+ c^2(A_t^+\nabla_h^+p_h^n)^\top
		a^{-1}\varGamma_{1} A_t^+\nabla_h^+p_h^n
		- 2c(A_t^+\bm{\chi}_h^n)^\top
		a^{-1}\varGamma_{1} A_t^+\nabla_h^+p_h^n
	\]
	reduces \eqref{eq:wave-pml-discrete-chi-energy-pre-h} to
	\begin{equation}\label{eq:wave-pml-discrete-chi-energy-h}
		\begin{aligned}
			 & \frac12\Bigl(\psi_\varepsilon^{n+1},
			D_t^-\bigl[|A_t^+\bm{\chi}_h^n|^2
				     +(A_t^+\bm{\lambda}_h^n)^\top
				     a^{-1}\varGamma_{1}
				     A_t^+\bm{\lambda}_h^n\bigr]\Bigr)_e \\
			 & \qquad +
			\|A_t^2\bm{\chi}_h^n\|_{\varGamma_{1}(I+a^{-1}\varGamma_{1}),e}^2
			+ c^2\|A_t^2\nabla_h^+p_h^n\|_{a^{-1}b,e}^2
			\\
			 & \qquad = c\bigl(A_t^2\bm{\chi}_h^n,
			(D_{2t}+a
			A_t^2)\nabla_h^+p_h^n\bigr)_{\psi_\varepsilon^{n+1},e}.
		\end{aligned}
	\end{equation}
	Adding \eqref{eq:wave-pml-discrete-p-energy-h} and
	\eqref{eq:wave-pml-discrete-chi-energy-h}, we arrive at
	\begin{equation*}
		\begin{aligned}
			 & \frac12\Bigl(\psi_\varepsilon^{n+1},
			          D_t^-\bigl[|q_h^n|^2\bigr]\Bigr)_{h}
			+ \frac12\Bigl(\psi_\varepsilon^{n+1},
			D_t^-\bigl[|A_t^+\bm{\chi}_h^n|^2
				     +(A_t^+\bm{\lambda}_h^n)^\top
				     a^{-1}\varGamma_1
				     A_t^+\bm{\lambda}_h^n\bigr]\Bigr)_e                           \\
			 & \qquad + \frac12\bigl(\Theta_h^{n+1},D_t^-|A_t^+p_h^n|^2\bigr)_h
			+
			\frac12\bigl(\beta(1-\psi_\varepsilon^{n+1}),D_t^-A_t^+|p_h^n|^2\bigr)_h
			= -\mathcal D_{h,{\rm embed}}^n.
		\end{aligned}
	\end{equation*}
	Finally, applying the discrete product rule
	\[
		(w^{n+1},D_t^-r^n)_\star
		= D_t^-(w^{n+1},r^n)_\star - (D_t^-w^{n+1},r^{n-1})_\star,
	\]
	both for the cell inner product and for the edge inner product,
	we convert the preceding relation into
	\eqref{eq:wave-pml-embedding-energy-fully-discrete}. This proves the
	fully discrete energy law.
\end{proof}

The adaptive computation in space is carried out on a nested hierarchy of
Cartesian domains $\{\Omega_\ell^n\}_{\ell=0}^{L^n}$ with
$\Omega_0^n=\Omega$, where level $\ell+1$ is obtained by refining
selected cells of level $\ell$. If $r_\ell$ denotes the refinement
ratio between levels $\ell$ and $\ell+1$, then the corresponding mesh
sizes are
\[
	h_{x,\ell}=\frac{h_x}{\prod_{m=0}^{\ell-1}r_m},
	\qquad
	h_{y,\ell}=\frac{h_y}{\prod_{m=0}^{\ell-1}r_m}.
\]
The composite approximation always takes the finest value at each
spatial location. On every active patch of level $\ell$, the same local
staggered scheme \eqref{eq:wave-pml-embedding-fully-discrete} is used,
now with the level-dependent operators
$\nabla_{h,\ell}^+$, $\mathcal D_{h,\ell}^{\omega}$, and
$\mathcal L_{h,\ell}^{\omega}$.

The adaptive computation is carried out on a standard block-structured AMR hierarchy \cite{BergerColella1989} with the mesh hierarchy managed by AMReX \cite{Zhang2019}. In the present implementation, all refinement levels use the same time step.

Refinement is driven by three indicators associated with the moving embedding field $\psi_\varepsilon$, the fixed PML mask, and the current pressure solution. After mild smoothing, a cell is tagged whenever at least one of these indicators exceeds its threshold. The tagged cells are then buffered and clustered into a properly nested hierarchy. Coarse-to-fine interpolation is used to fill ghost cells at coarse-fine interfaces; when regridding creates new fine patches, it is also used to initialize the data on those patches. The new time-level solution on the resulting hierarchy is then obtained by solving the discrete equations.

We summarize the adaptive algorithm in the following pseudocode.

\begin{center}
	\fbox{
		\begin{minipage}{0.95\textwidth}
			\footnotesize
			\noindent\textbf{Algorithm 1. Adaptive update
				for the fully-discrete PML-DE system.}

			\medskip
			\noindent\textbf{Input:} composite hierarchy
			$\{\Omega_\ell^n\}_{\ell=0}^{L^n}$, cell-centered states
			$\{p_{h,\ell}^{n-1},p_{h,\ell}^{n}\}_{\ell=0}^{L^n}$, edge-centered
			states
			$\{\lambda_{1,h,\ell}^{n},\lambda_{2,h,\ell}^{n}\}_{\ell=0}^{L^n}$,
			refinement ratios $\{r_\ell\}$, and a prescribed regridding set
			$\mathcal N_{\rm re}\subset\mathbb N$.

			\medskip
			\noindent\textbf{for} $n=1,2,\ldots$ \textbf{do}
			\begin{enumerate}
				\item If $n\in\mathcal N_{\rm re}$, then for each active level $\ell$
				      compute the three sensors
				      $\eta_{{\rm emb},\ell}^n$, $\eta_{{\rm pml},\ell}$, and
				      $\eta_{{\rm sol},\ell}^n$ from $\psi_{\varepsilon,\ell}^n$,
				      the PML mask, and
				      $p_{h,\ell}^n$, respectively; smooth them by a small number of
				      nearest-neighbor sweeps; and tag a cell $K\subset\Omega_\ell^n$
				      whenever
				      \[
					      \max\left\{
					      \frac{\eta_{{\rm emb},\ell}^n(K)}{\tau_{\rm emb}},
					      \frac{\eta_{{\rm pml},\ell}(K)}{\tau_{\rm pml}},
					      \frac{\eta_{{\rm sol},\ell}^n(K)}{\tau_{\rm sol}}
					      \right\}>1.
				      \]
				      Enlarge the tagged set by a fixed buffer and cluster it into a
				      properly nested patch hierarchy.

				\item Persistent fine
				      cells keep their old values. For every newly created fine cell,
				      prolong $p_{h,\ell}^n$ and $p_{h,\ell}^{n-1}$ bilinearly from the
				      parent coarse cell, and prolong the two components
				      $\lambda_{1,h,\ell}^n$ and $\lambda_{2,h,\ell}^n$ componentwise onto
				      the corresponding horizontal and vertical edge grids of level
				      $\ell+1$.

				\item For each active
				      level $\ell$, fill same-level ghost cells by copying from adjacent
				      patches; fill coarse-fine ghost cells for the cell-centered pressure
				      and the edge-centered auxiliary components by the same prolongation
				      rules; and then apply the homogeneous outer-boundary closure.

				\item For
				      $\ell=0,1,\ldots,L^n$ and every active patch $\mathcal
					      P_{\ell,m}^n$,
				      solve on its valid cells the level-dependent Crank--Nicolson system
				      \[
					      \left\lbrace
					      \begin{aligned}
						       & \psi_{\varepsilon}^{n+1} D_t^2 p_{h,\ell}^n
						      + a \psi_{\varepsilon}^{n+1} D_{2t} p_{h,\ell}^n
						      + b \psi_{\varepsilon}^{n+1} A_t^2 p_{h,\ell}^n
						      + \tfrac{W_{\varepsilon}^{n+1}}{\eta_d}A_t^2 p_{h,\ell}^n \\
						       & \qquad + (1-\psi_{\varepsilon}^{n+1})
						      \bigl(\alpha D_{2t}p_{h,\ell}^n+\beta p_{h,\ell}^{n+1}\bigr)
						      = c^2\mathcal
						      L_{h,\ell}^{\psi_{\varepsilon}^{n+1}}(A_t^2p_{h,\ell}^n)
						      + c\mathcal
						      D_{h,\ell}^{\psi_{\varepsilon}^{n+1}}(A_t^2\bm{\lambda}_{h,\ell}^n),
						      \\
						       & D_t^+\bm{\lambda}_{h,\ell}^n
						      + \varGamma_{1} A_t^+\bm{\lambda}_{h,\ell}^n
						      = c\varGamma_{2}
						      A_t^+\nabla_{h,\ell}^+p_{h,\ell}^n.
					      \end{aligned}
					      \right.
				      \]
				      This yields $p_{h,\ell}^{n+1}$ on cell centers and
				      $(\lambda_{1,h,\ell}^{n+1},\lambda_{2,h,\ell}^{n+1})$ on the two
				      edge grids of the patch.

				\item For
				      $\ell=L^n-1,L^n-2,\ldots,0$, replace every covered coarse cell value
				      of $p_{h,\ell}^{n+1}$ by the cell average of the fine solution, and
				      replace the covered coarse-edge values of
				      $\lambda_{1,h,\ell}^{n+1}$ and $\lambda_{2,h,\ell}^{n+1}$ by the
				      corresponding edge averages of the fine solution.

				\item For every active level $\ell$, set
				      \[
					      p_{h,\ell}^{n-1}\leftarrow p_{h,\ell}^{n},
					      \quad
					      p_{h,\ell}^{n}\leftarrow p_{h,\ell}^{n+1},
					      \quad
					      \bm{\lambda}_{h,\ell}^{n}\leftarrow \bm{\lambda}_{h,\ell}^{n+1}.
				      \]
			\end{enumerate}
			\noindent\textbf{end for}
		\end{minipage}}
\end{center}

Algorithm~1 follows the same order as the implementation. Step~1 changes the
hierarchy only when regridding is scheduled. Steps~2 and~3 initialize newly
created fine cells and fill ghost values. Step~4 advances the discrete PML-DE
system on the resulting hierarchy. Steps~5 and~6 average fine data back to
covered coarse regions and rotate the stored time levels. Thus, away from the
regridding steps $n\in\mathcal N_{\rm re}$, the algorithm reduces to ghost
filling, the level solve, averaging-down, and the time update on a fixed grid
layout.

\section{Numerical results}\label{sec:numerical}

\noindent \indent In this section, we conduct a sequence of numerical experiments for
both fixed and moving objects. The first example serves as a static
benchmark and illustrates the basic wave propagation, scattering, and PML
absorption properties of the sound-soft {\PMLDES}. The
second example compares solutions obtained from the present embedded finite-difference solver
with those calculated using a sharp-interface finite-element method for a circular
scatterer under both sound-soft and sound-hard boundary conditions.
The remaining numerical experiments concern moving circular, star-shaped, and ship-shaped objects
and are carried out for two incident-wave parameter sets representing
moderate- and high-frequency regimes, respectively. These experiments are used to
assess the quality of the computed, scattered field, the behavior of the weighted
energy, and the effect of the moving embedding on the fully discrete
energy law. {All moving-object experiments below use prescribed uniform rectilinear motion. Consequently, for the sound-hard cases the normal acceleration term $a_{n_\star}(t)$ in the embedded model is identically zero.}
\begin{exa}
	We consider wave scattering generated by a nonzero initial pressure
	disturbance in the presence of a fixed object. In this example the
	source term vanishes. We set $\Omega_{\rm phy} = (-10,10)^2$ and
	$\Omega = (-14,14)^2$, take wave speed $c=1$, and use time
	step size $\tau = 10^{-2}$. The initial data are
	\begin{equation*}
		p(\mathbf{x}, 0) = e^{-5[(x-5)^2 + y^2]}, \quad \partial_t
		p(\mathbf{x}, 0) = 0.
	\end{equation*}
	The object is a fixed circle centered at the origin with radius
	$R=2$.
\end{exa}
\begin{figure}
	\begin{center}
		\includegraphics[width=0.24\textwidth]{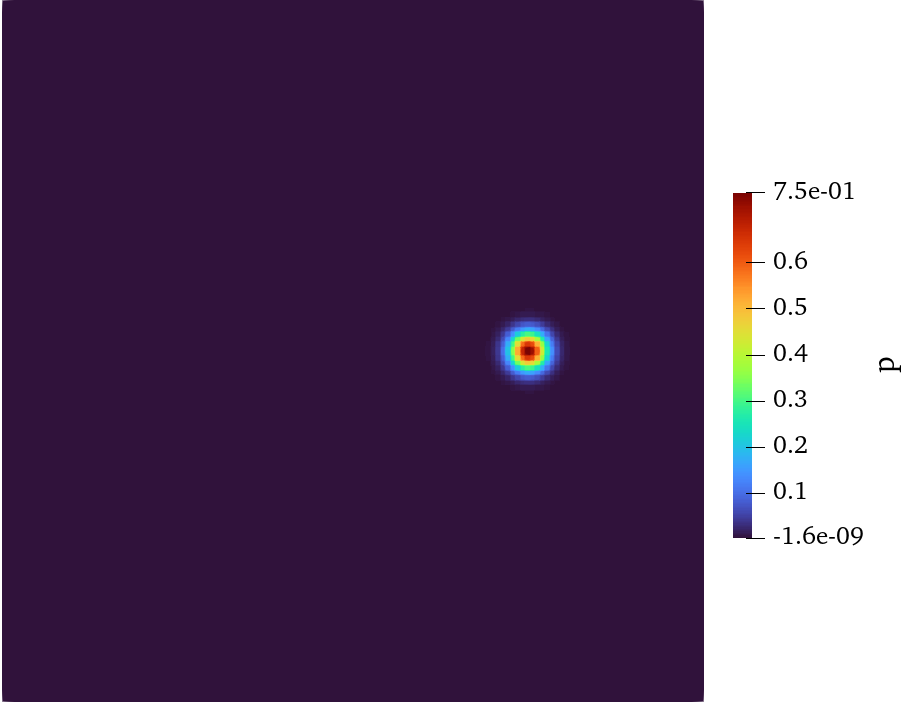}
		\includegraphics[width=0.24\textwidth]{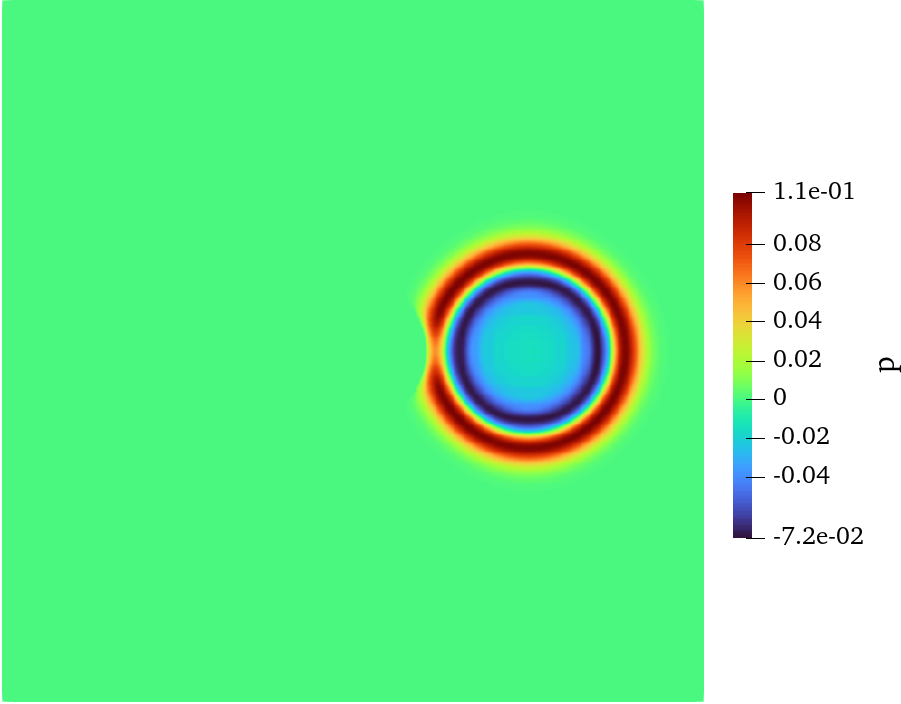}
		\includegraphics[width=0.24\textwidth]{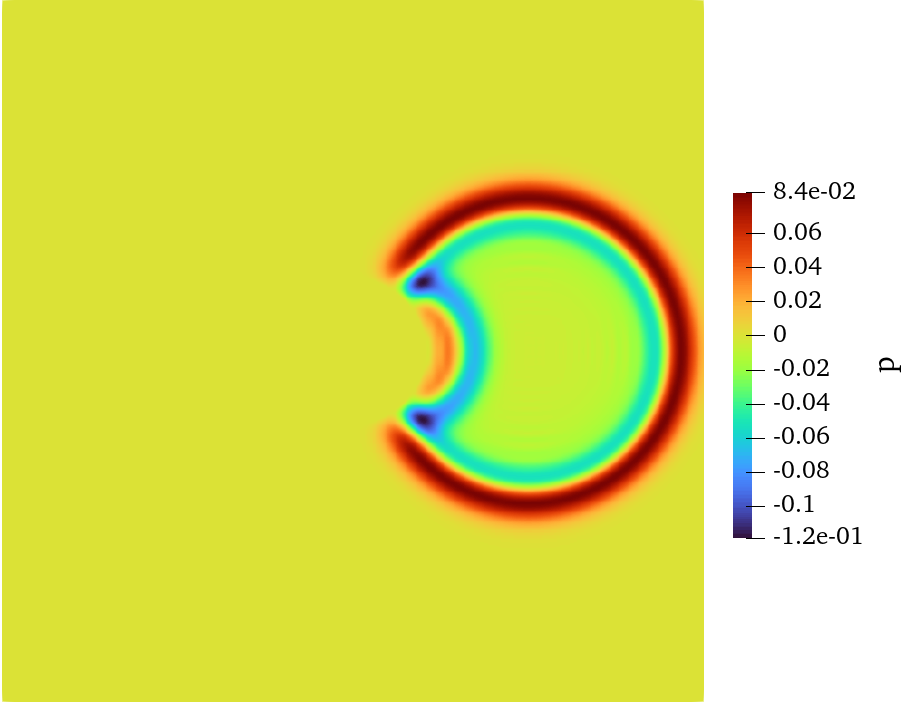}
		\includegraphics[width=0.24\textwidth]{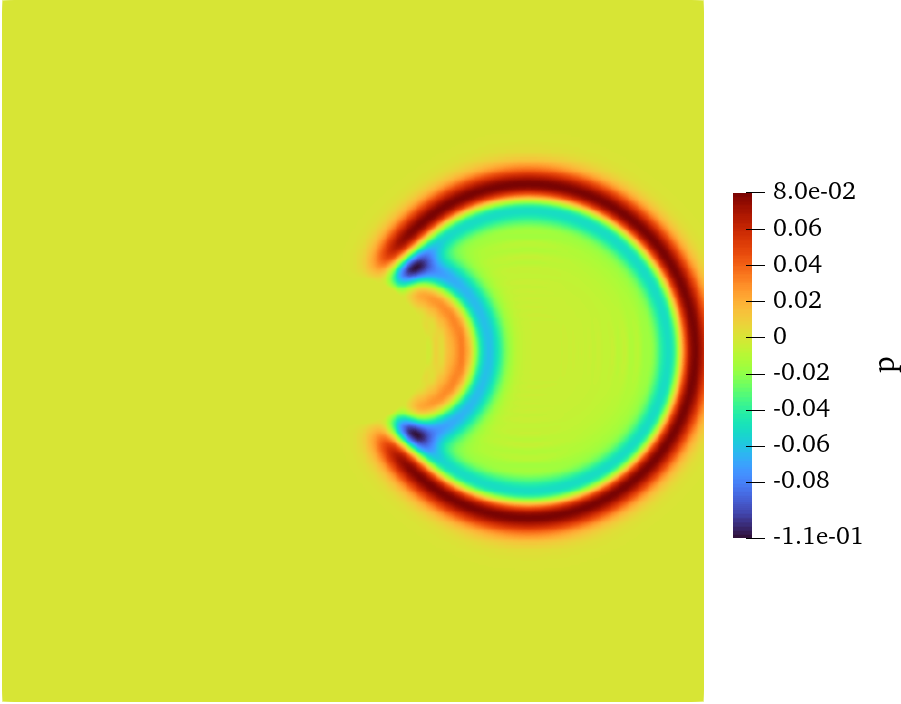}
		\includegraphics[width=0.24\textwidth]{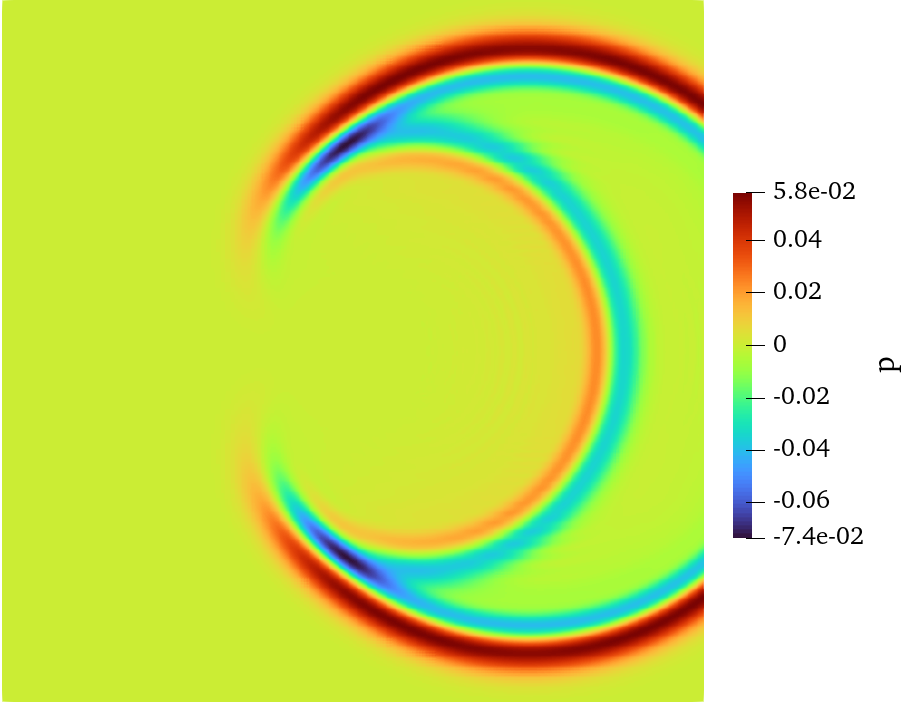}
		\includegraphics[width=0.24\textwidth]{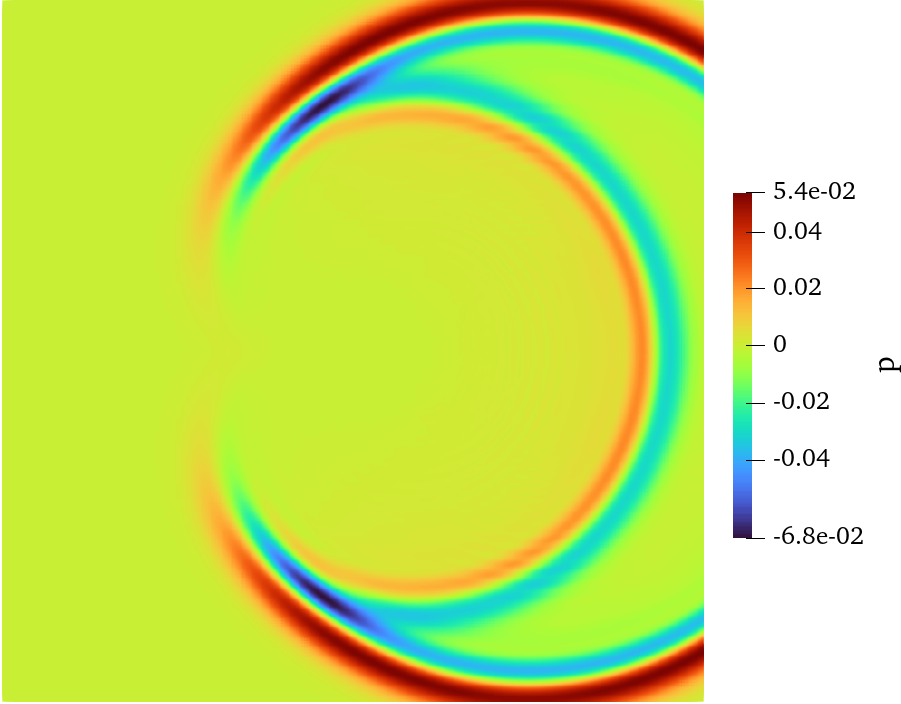}
		\includegraphics[width=0.24\textwidth]{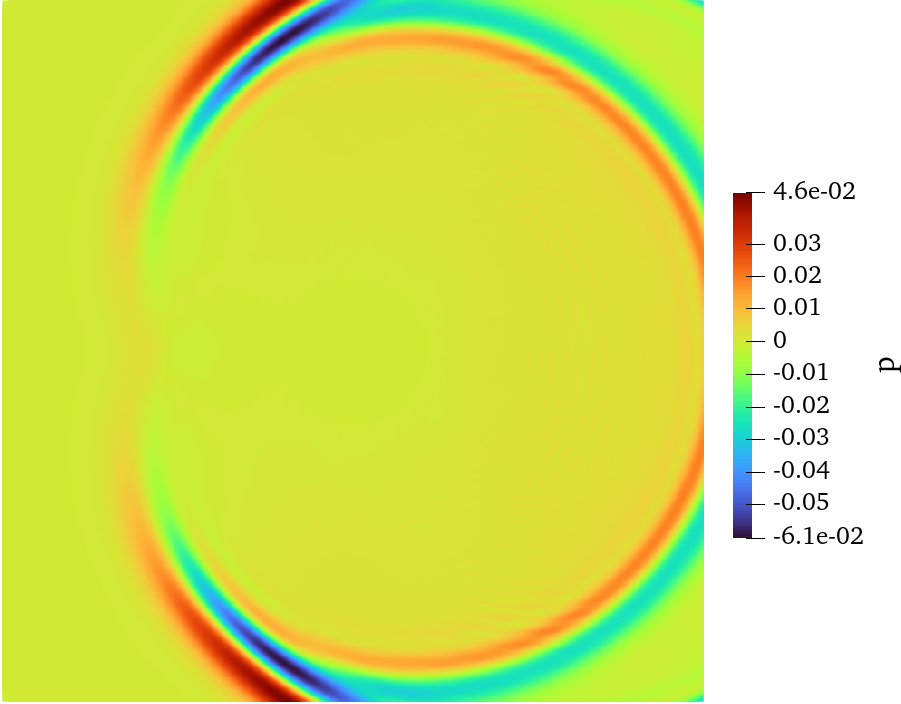}
		\includegraphics[width=0.24\textwidth]{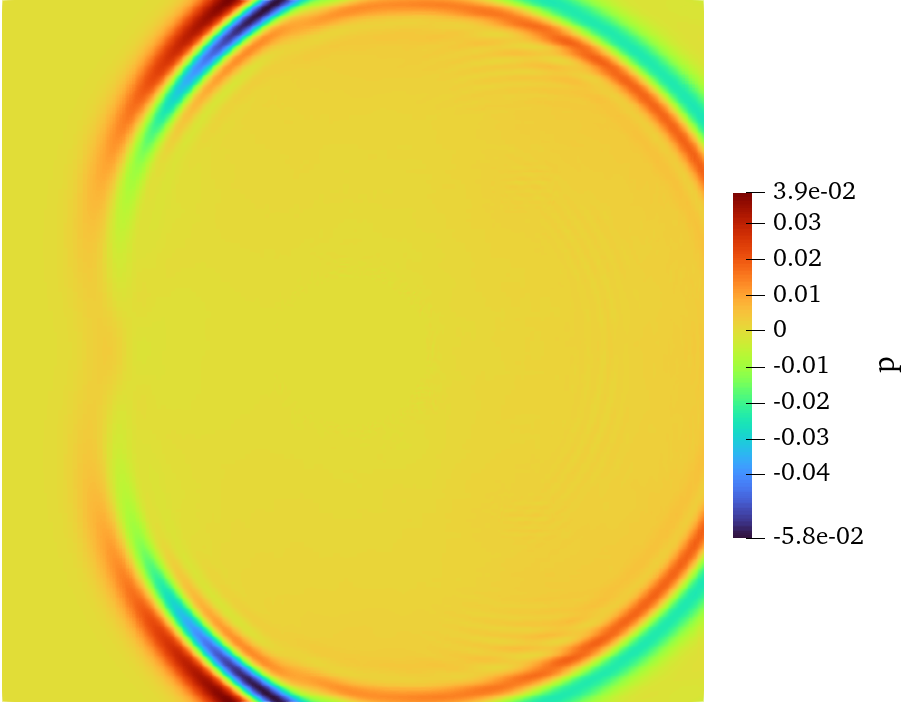}
	\end{center}
	\caption{Snapshots of the pressure field at times $t =
			0.1, 2.6,4.2, 4.6, 8.5, 9.8, 12, 13$ (from left to right, top to
		bottom), respectively,  in $\Omega_{phy}$ for the
		fixed-object test. The outgoing wave is absorbed by the outer PML very well
		while the embedded object remains sharply resolved on the fixed computational
		domain.}\label{fig:nosource}
\end{figure}

\begin{figure}
	\begin{center}
		\includegraphics[width=0.245\textwidth]{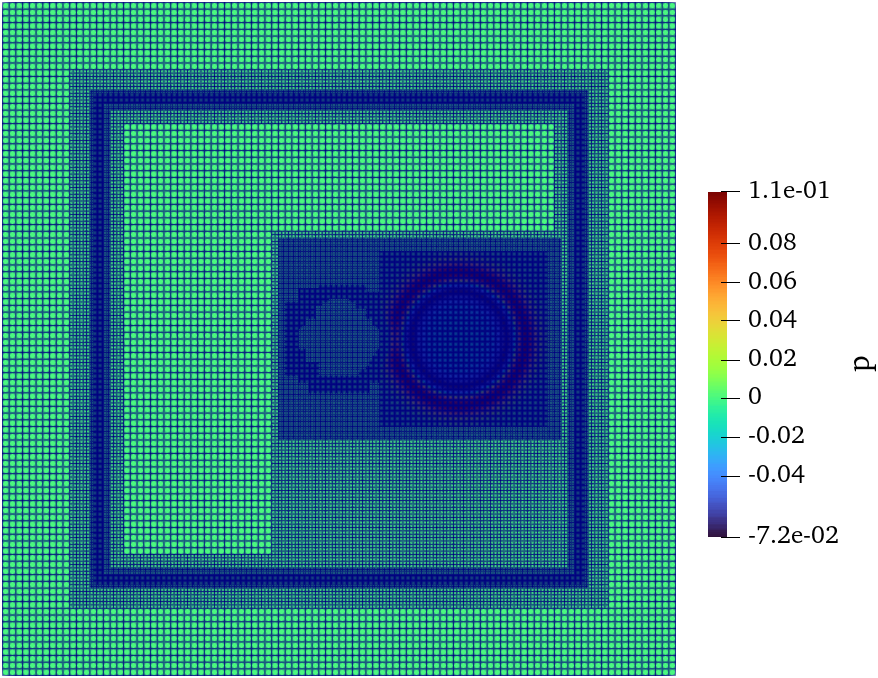}
		\includegraphics[width=0.245\textwidth]{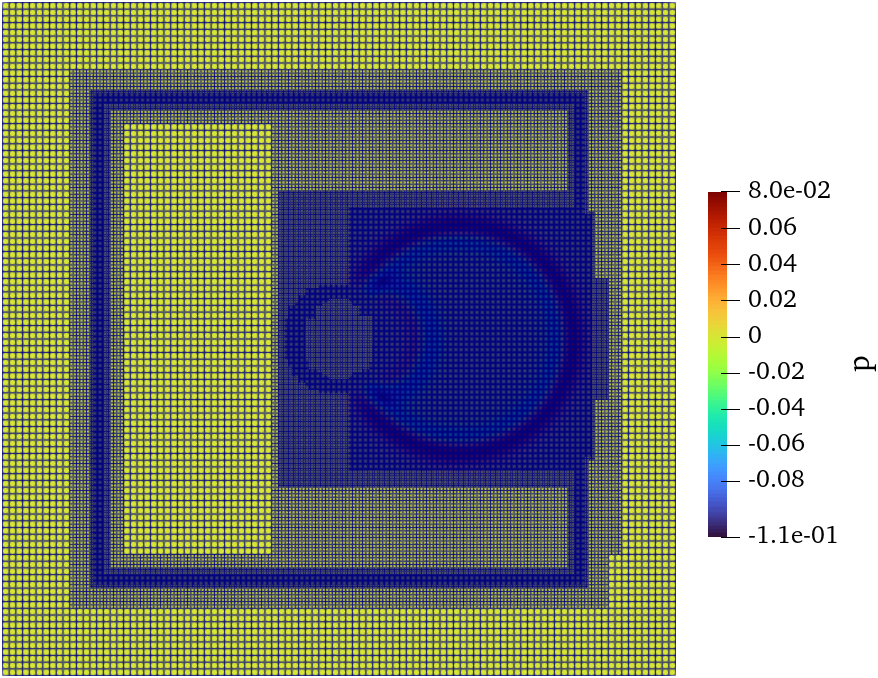}
		\includegraphics[width=0.245\textwidth]{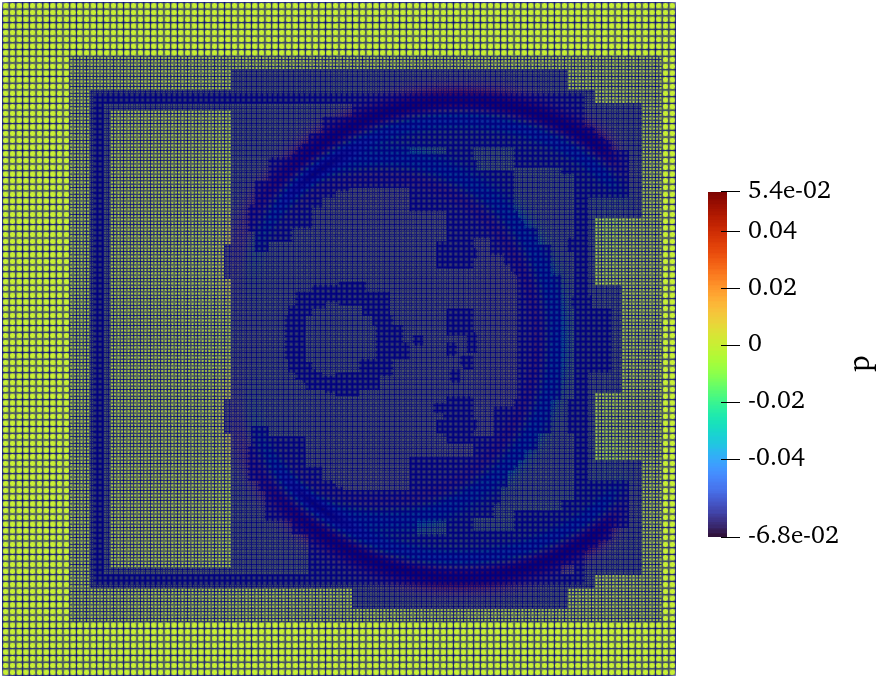}
		\includegraphics[width=0.245\textwidth]{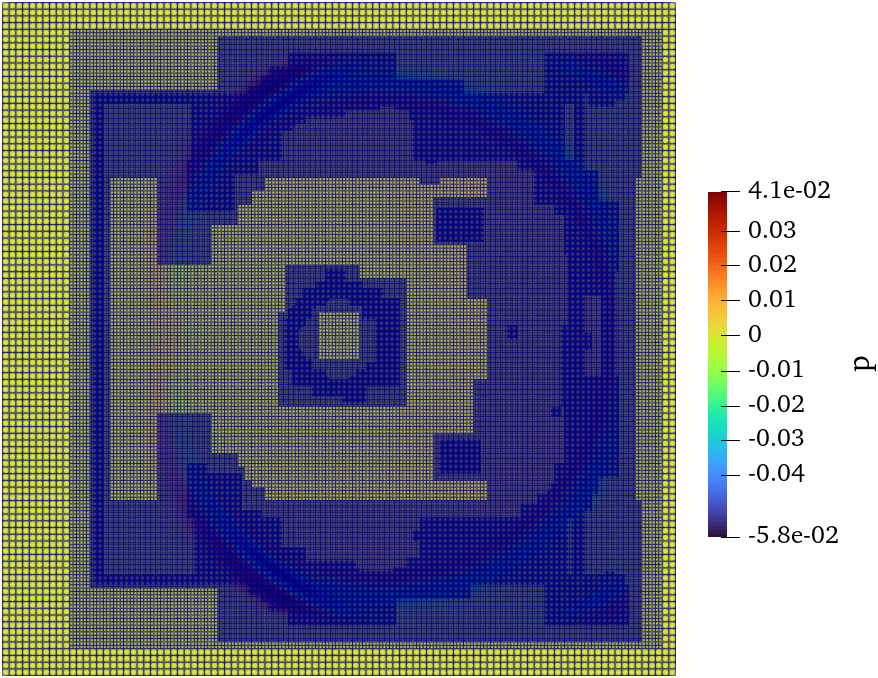}
	\end{center}
	\caption{Adaptive grid distribution at various time during the computation. Pressure profiles on the adaptive grid at selected times in
		$\Omega$. The solution remains smooth across the fixed-grid
		embedding region and decays in the absorbing layer.}\label{fig:adaptive}
\end{figure}

\begin{figure}
	\begin{center}
		\includegraphics[width=0.6\textwidth]{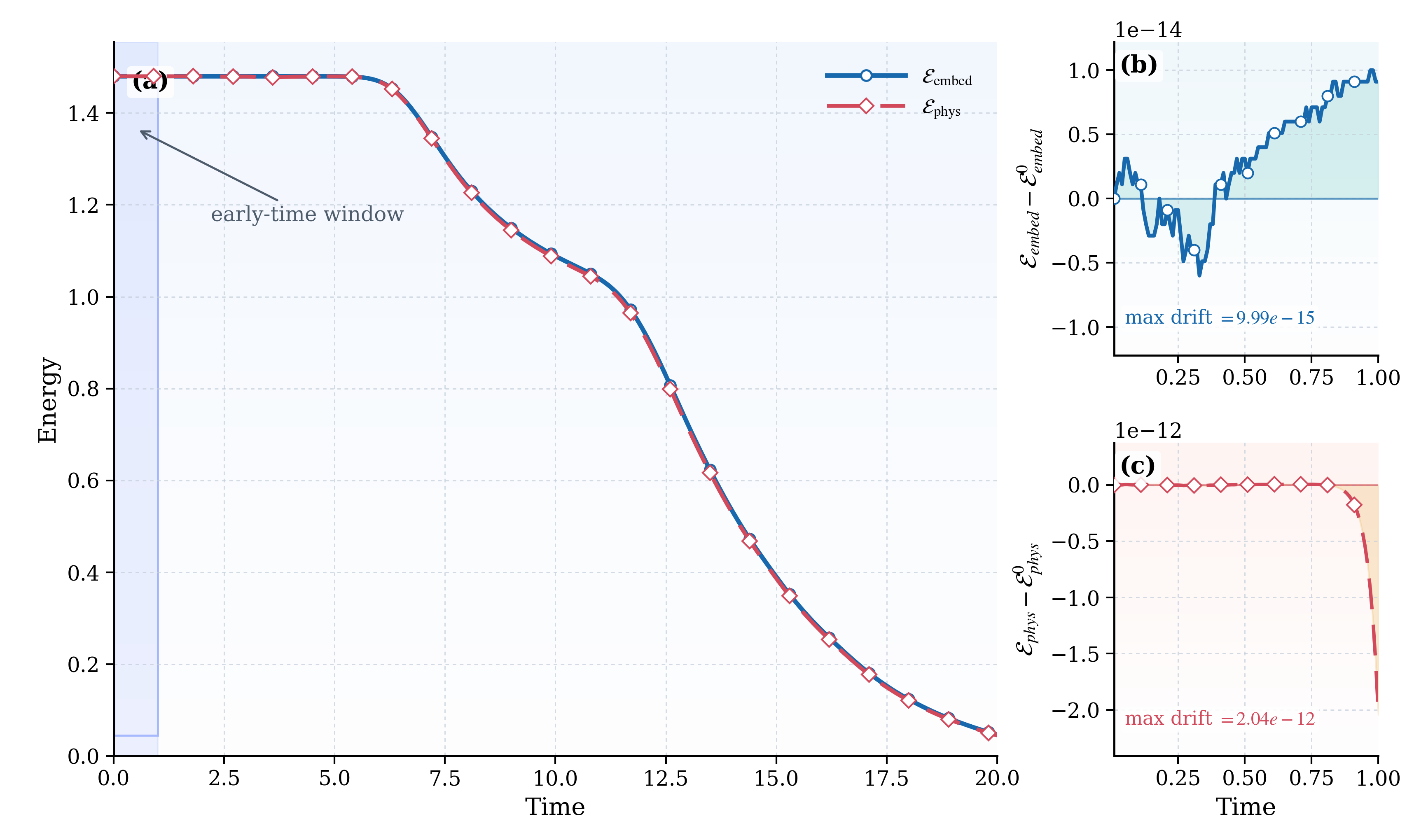}
	\end{center}
	\caption{Energy evolution in Example~4.1. Panel (a) compares the
		weighted energy $\mathcal{E}_{\rm embed}$ of the embedded method and the "ground truth" physical
		energy $\mathcal{E}_{\rm phys}$ over the full simulation interval;
		the shaded region marks the early-time window before the wave
		reaches the object or the PML. Panels (b) and (c) display the
		corresponding early-time drifts, confirming near machine-precision
		conservation prior to the onset of dissipation.}\label{fig:energy-static}
\end{figure}

Figure~\ref{fig:nosource} records eight representative snapshots of the
pressure field. Initially, the Gaussian pulse expands
nearly radially from $\mathbf{x}=(5,0)$. Once the incident front reaches
the circular object, the wave splits into reflected and diffracted
components, and a clear shadow region becomes visible behind the
object. At later times the scattered field propagates toward the
outer collar, where it is absorbed smoothly by the PML. Even in the
last two snapshots, no visible spurious wave returns from the boundary
of $\Omega_{\rm PML}$, which confirms that the absorbing layer remains
effective after coupling with the embedded-object formulation.

These results are consistent with the qualitative behavior reported in
earlier studies of the same benchmark configuration \cite{Gu2026}. In particular,
the fixed-grid embedding captures the object boundary sharply enough
to reproduce the expected reflection pattern, while the PML eliminates
the artificial reflections that were visible in our previous work with
a truncated radiation boundary condition \cite{Gu2026}.

Figure~\ref{fig:adaptive} shows the same computation on the adaptive
Cartesian grid. The refinement follows the diffuse object interface,
the dominant wave fronts, and the PML region, whereas the remainder of
the computational domain stays relatively coarse. This distribution is consistent with
the local regularity of the solution and substantially improves the
computational efficiency compared with a uniform discretization.

Figure~\ref{fig:energy-static} complements the field plots by showing the
corresponding energy evolution for the fixed-object test. Over the full
time interval, the weighted energy computed using the embedded  model and the "ground truth" physical energy are
nearly indistinguishable. In the shaded early-time window, before the
expanding pulse reaches either the object or the PML collar, both
quantities remain essentially constant; the drifts shown in panels (b)
and (c) are only of the order $10^{-14}$--$10^{-12}$. Once scattering
and absorption become active, the two energies decay altogether, which is
consistent with the dissipative structure of the static \PMLDE{}
formulation.

\begin{exa}
	We compare the numerical errors in wave scattering by a circular object
	under both sound-soft and sound-hard boundary conditions. The
	problem is solved using two approaches: a finite difference method
	in a domain-embedding framework, and a finite element method based
	on a classical sharp-interface formulation. In this example, we set
	$\Omega_{\rm phy} = (-5, 5)^2$, $\Omega = (-7, 7)^2$, $c = 10$,
	$\eta = 0.25$, $\mathbf{x}_0 = (-3, 0)$, $\sigma = \frac{2}{25}$,
	$t_0 = 0$, and $w = 10\pi$. For the sound-soft case,
	$\varepsilon = 0.01$, while for the sound-hard case, $\varepsilon = 0.05$.
\end{exa}

\begin{figure}
	\begin{center}
		\includegraphics[width=0.24\textwidth]{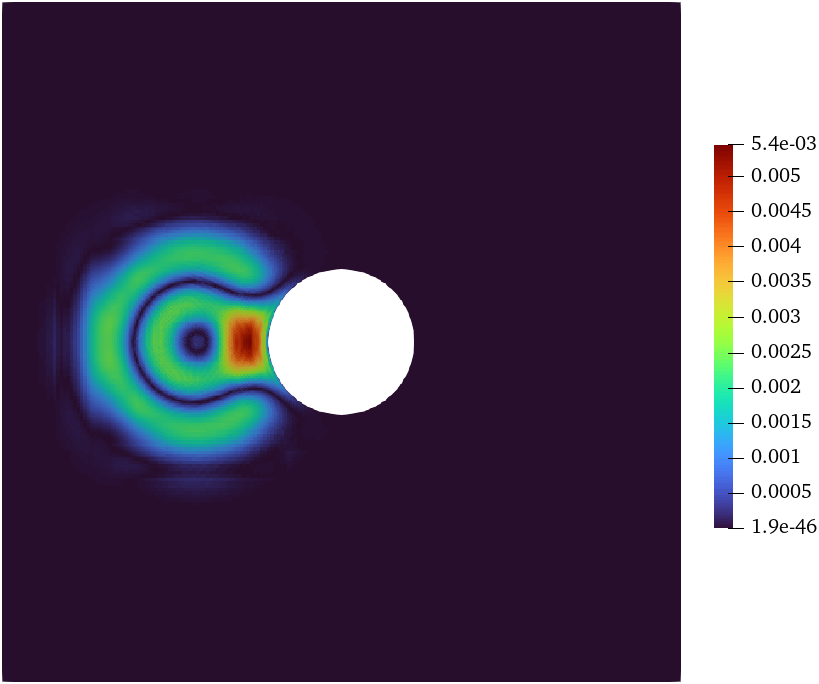}
		\includegraphics[width=0.24\textwidth]{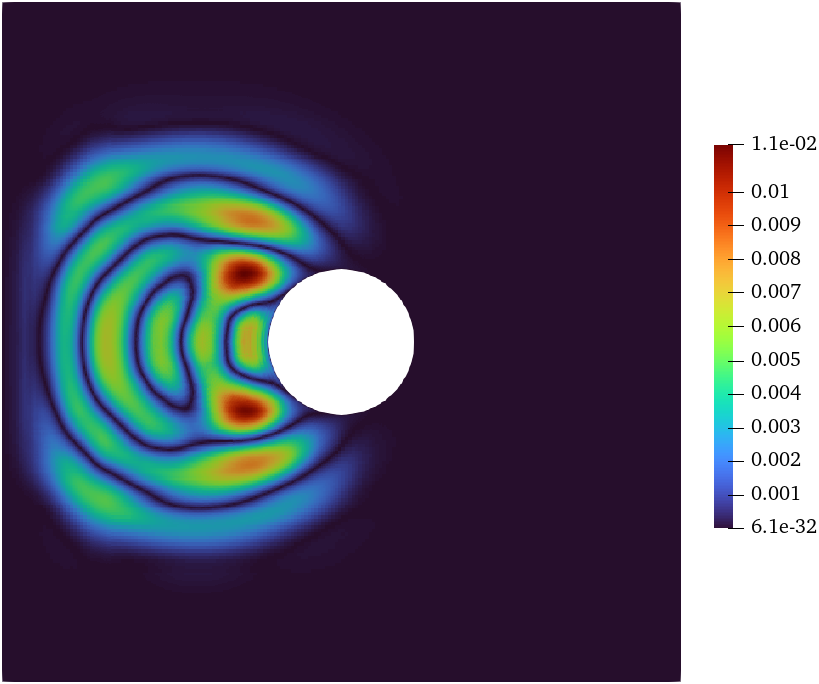}
		\includegraphics[width=0.24\textwidth]{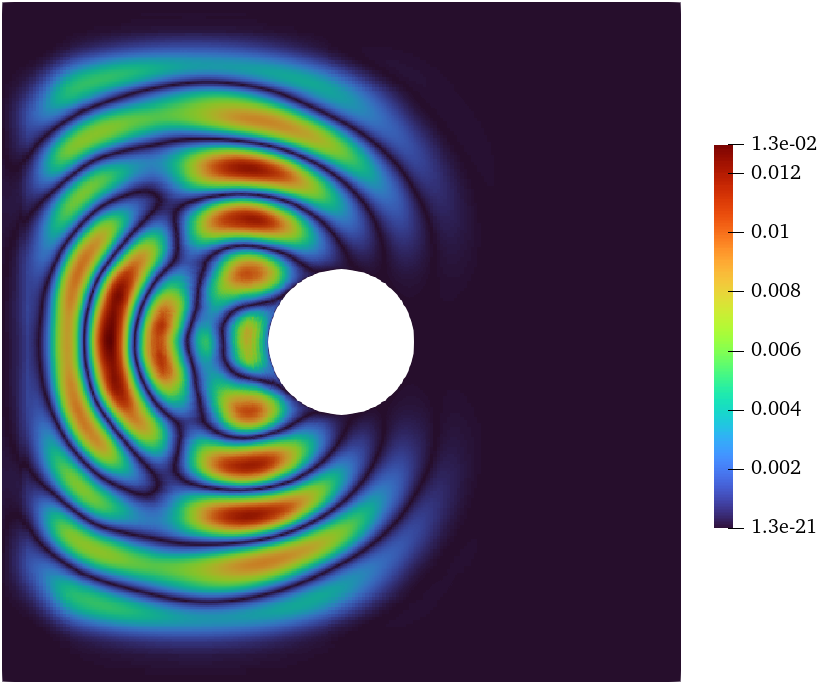}
		\includegraphics[width=0.24\textwidth]{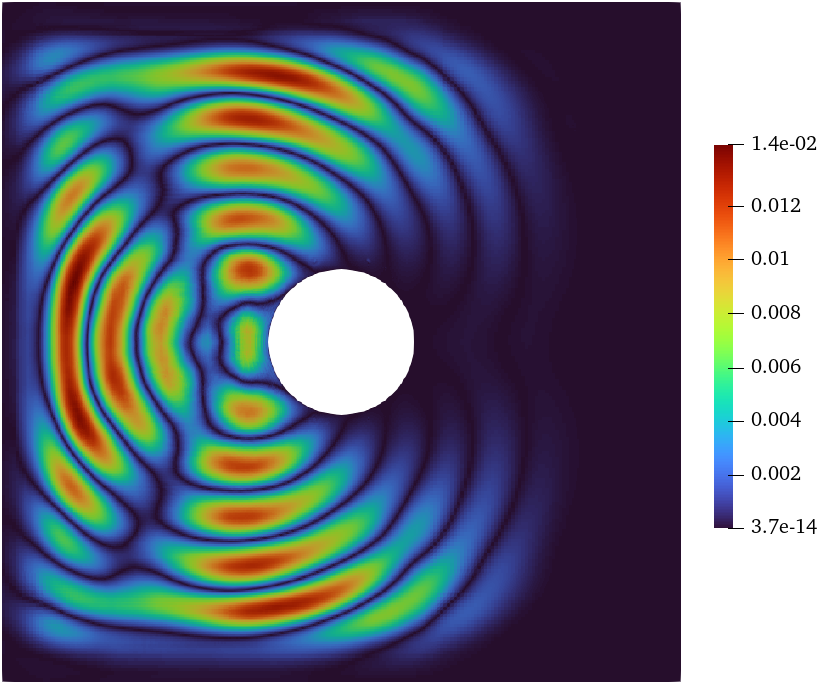}
		\includegraphics[width=0.24\textwidth]{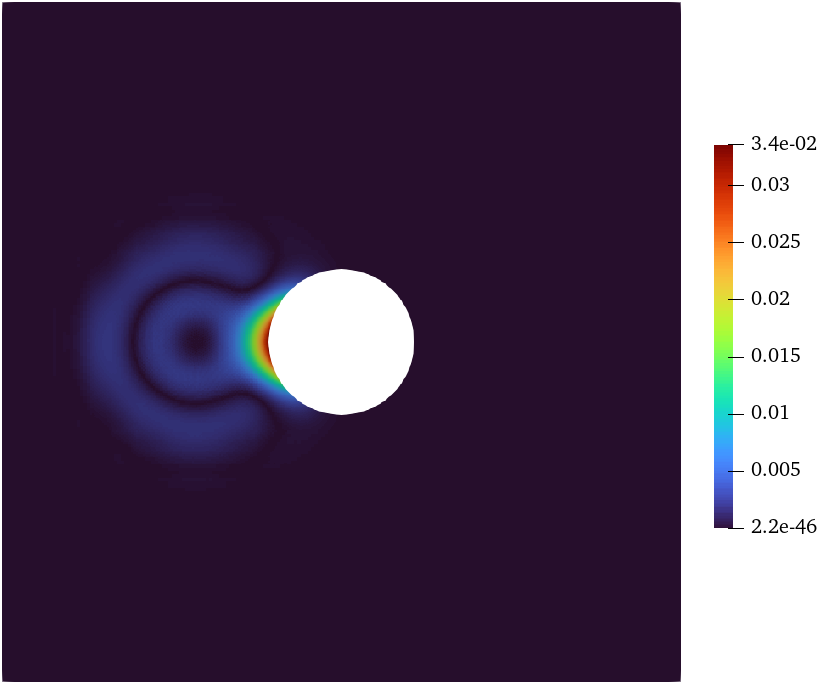}
		\includegraphics[width=0.24\textwidth]{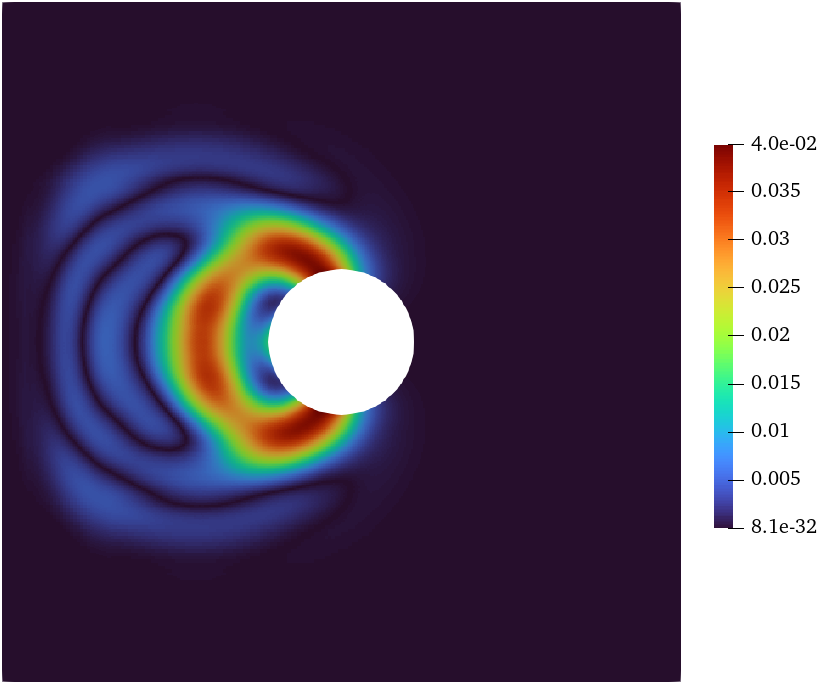}
		\includegraphics[width=0.24\textwidth]{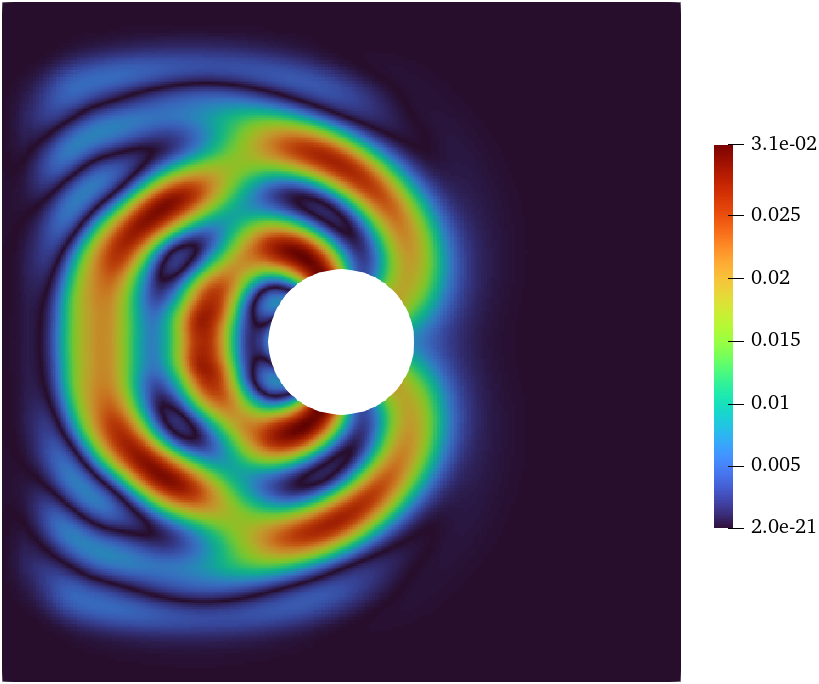}
		\includegraphics[width=0.24\textwidth]{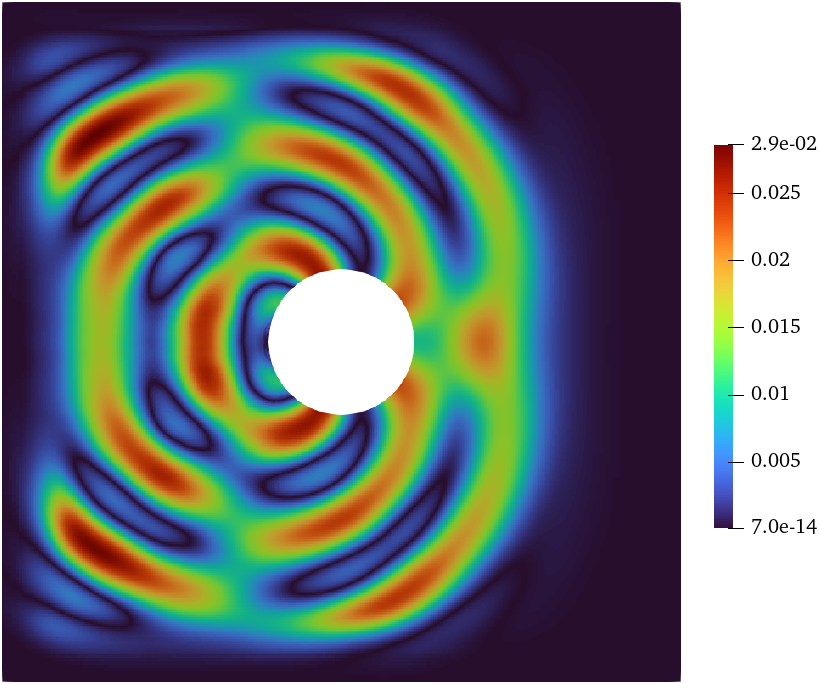}
	\end{center}
	\caption{The error density between the present embedded
		finite-difference solution and a sharp-interface finite-element ground truth
		reference for the circular-object benchmark at $t=0.2, 0.4, 0.6,$
		and $0.8$ (from left to right), respectively. Top row: sound-soft case. Bottom
		row: sound-hard case.}\label{fig:error-benchmark}
\end{figure}

Figure~\ref{fig:error-benchmark} shows the deviation or error between the
embedded finite-difference solution and a sharp-interface finite-element
reference solution. The reference solution is computed on a finite-element mesh with mesh size $h=0.01$. To evaluate the error, the
finite-element reference solution is interpolated onto the Cartesian grid
used by the embedded scheme, and the difference between the two solutions
is then plotted at the same output time.

The error remains strongly localized near the object boundary and along
the dominant outgoing and scattered wave fronts. In particular, no visible
large-scale pollution appears in the far field, indicating that the diffuse interface
embedding and the PML truncation do not introduce spurious global distortion
into the solution. The sound-hard case exhibits a somewhat broader error
region and slightly larger error than the sound-soft case, which is consistent with the larger
interface thickness used in this experiment and with the additional
sensitivity of the Neumann-type treatment near the object boundary.
Nevertheless, in both cases the error remains confined to narrow zones
associated with the geometric interface and the strongest wave activity,
confirming that the finite-difference solution obtained from the embedded  model is in good agreement
with the sharp-interface finite-element reference.

\begin{exa}
	We next consider wave scattering by moving objects under both
	sound-soft and sound-hard boundary conditions. Two incident-wave
	parameter sets are used to illustrate moderate- and high-frequency
	regimes:
	\begin{itemize}
		\item $\bx_0=(-3,0)$, $\eta=0.25$, $w=10\pi$,
		      $\sigma=\frac{2}{25}$, $c=10$, $t_0=0$;
		\item $\bx_0=(-3,0)$, $\eta=0.01$, $w=100\pi$,
		      $\sigma=\frac{1}{1000}$, $c=10$, $t_0=0$.
	\end{itemize}
	In this example, $\Omega_{\rm phy} = (-7, 7)^2$
	$\Omega = (-9, 9)^2$, and $\varepsilon = 0.05$. {The circular and star-shaped objects are transported along straight lines with constant velocities, so their accelerations vanish throughout the simulations.}
\end{exa}

\begin{figure}[H]
	\begin{center}
		\includegraphics[width=0.24\textwidth]{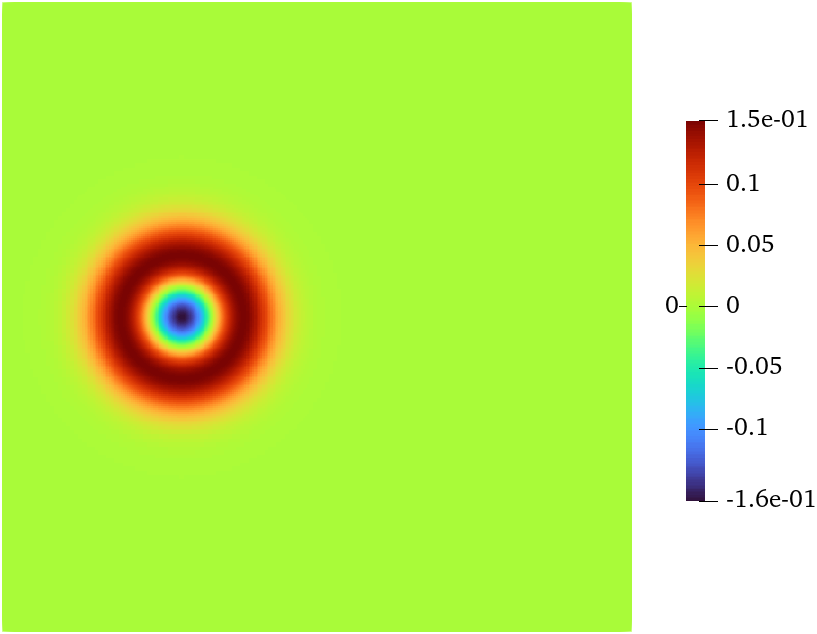}
		\includegraphics[width=0.24\textwidth]{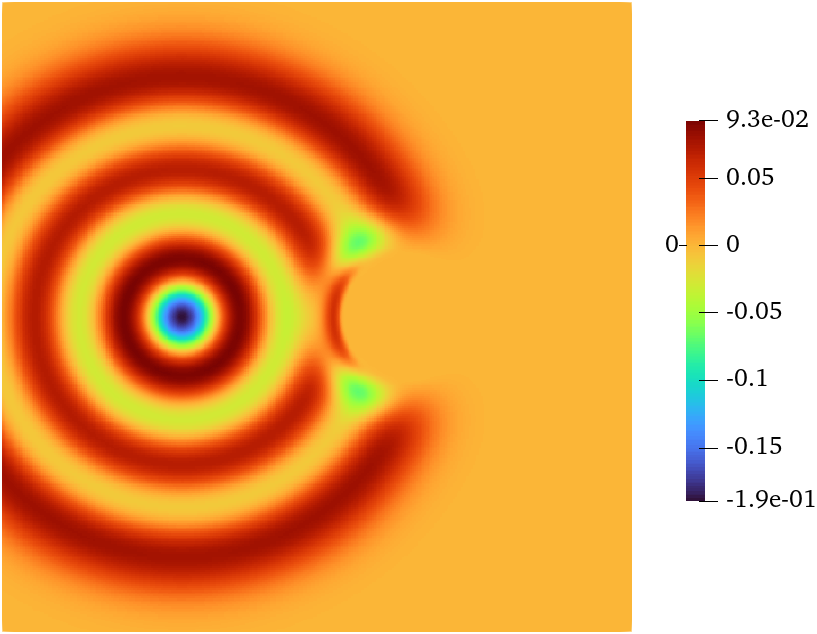}
		\includegraphics[width=0.24\textwidth]{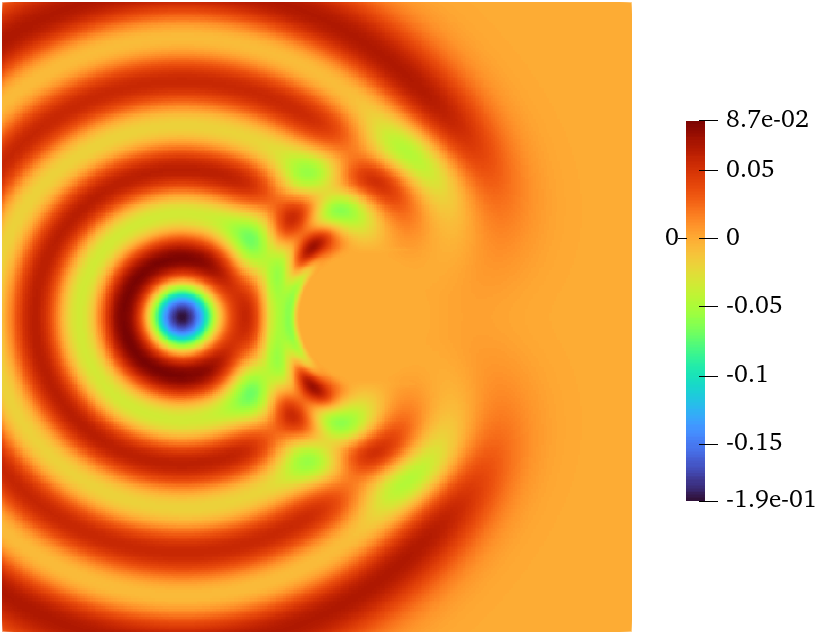}
		\includegraphics[width=0.24\textwidth]{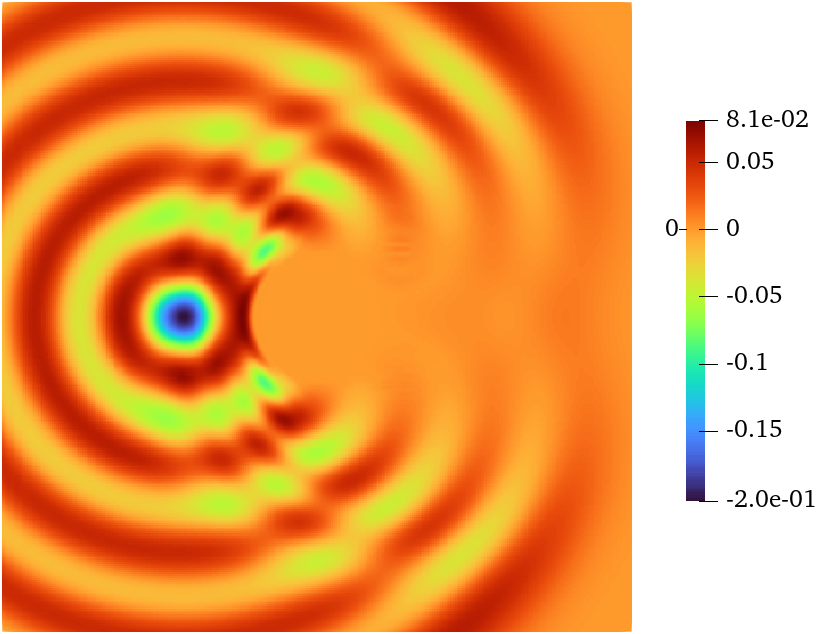}
		\includegraphics[width=0.24\textwidth]{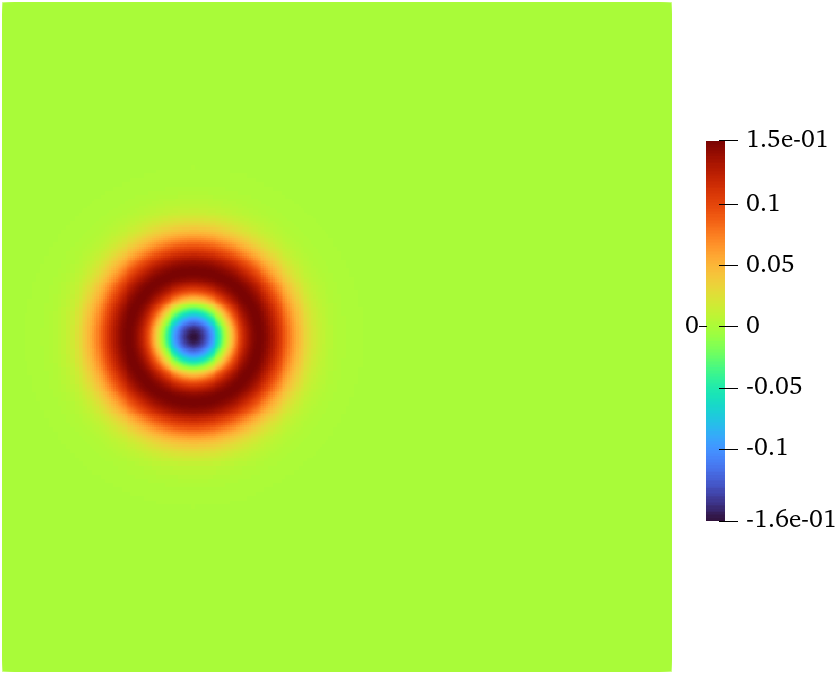}
		\includegraphics[width=0.24\textwidth]{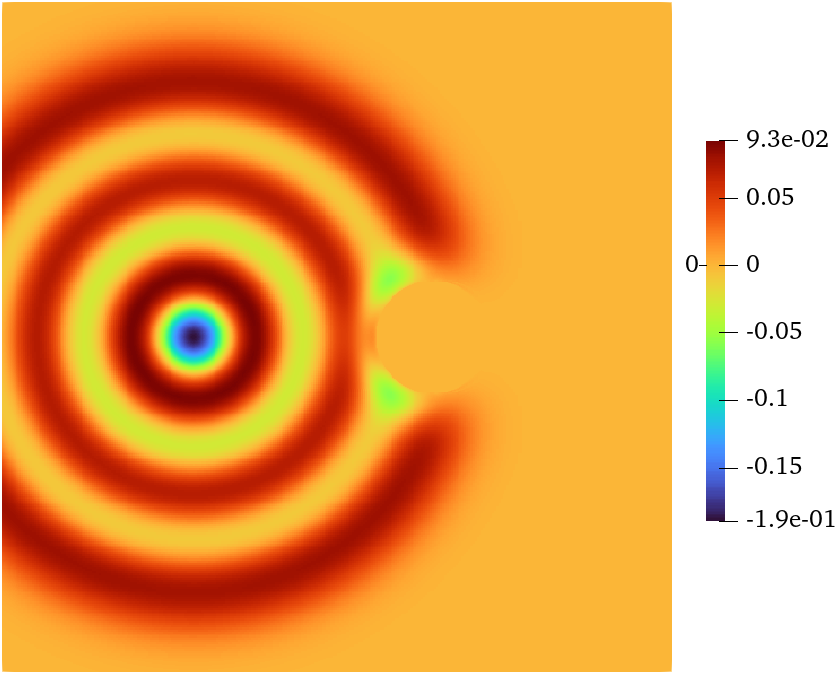}
		\includegraphics[width=0.24\textwidth]{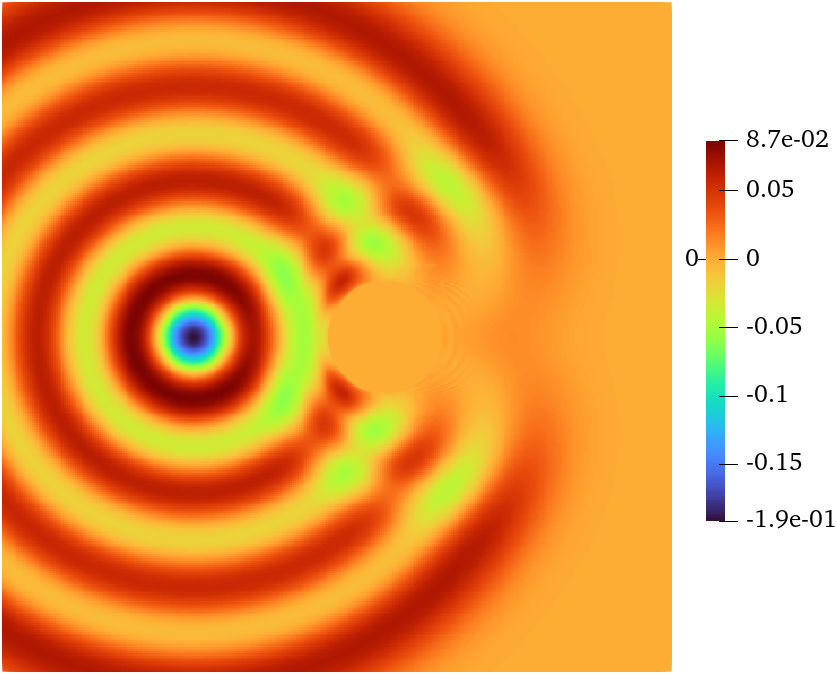}
		\includegraphics[width=0.24\textwidth]{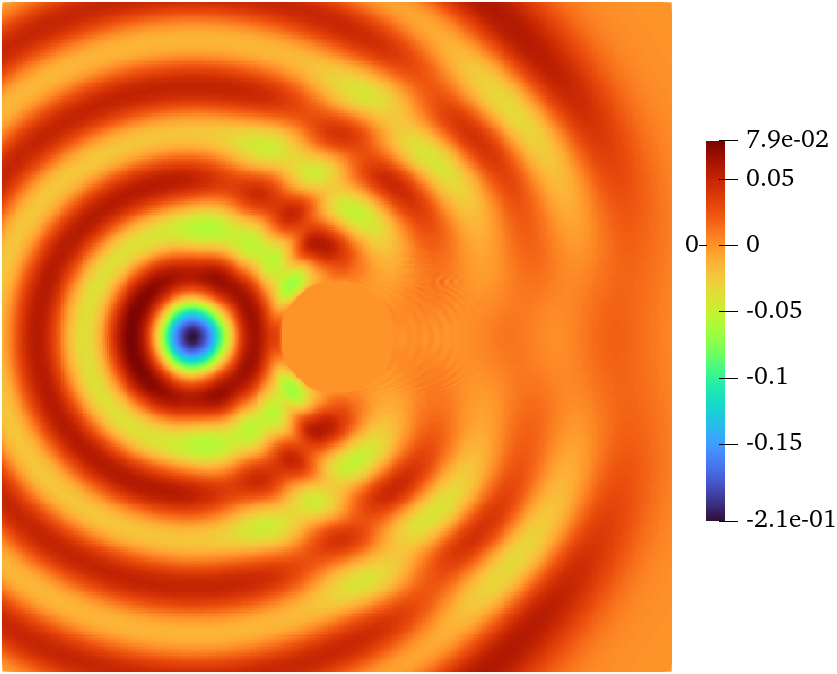}
		\begin{minipage}{0.235\textwidth}
			\hspace*{-0.25cm}
			\includegraphics[width=\linewidth]{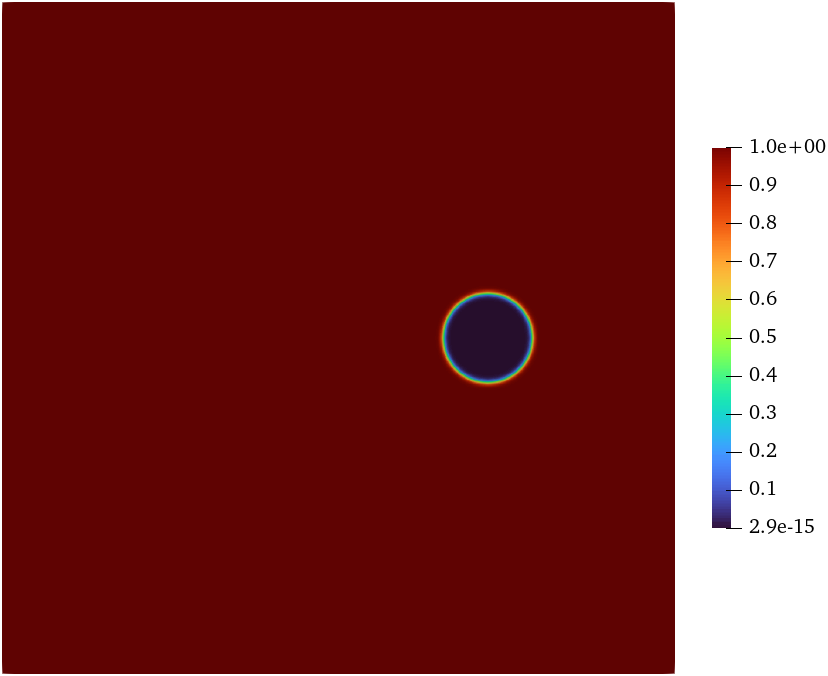}
		\end{minipage}
		\begin{minipage}{0.235\textwidth}
			\hspace*{-0.15cm}
			\includegraphics[width=\linewidth]{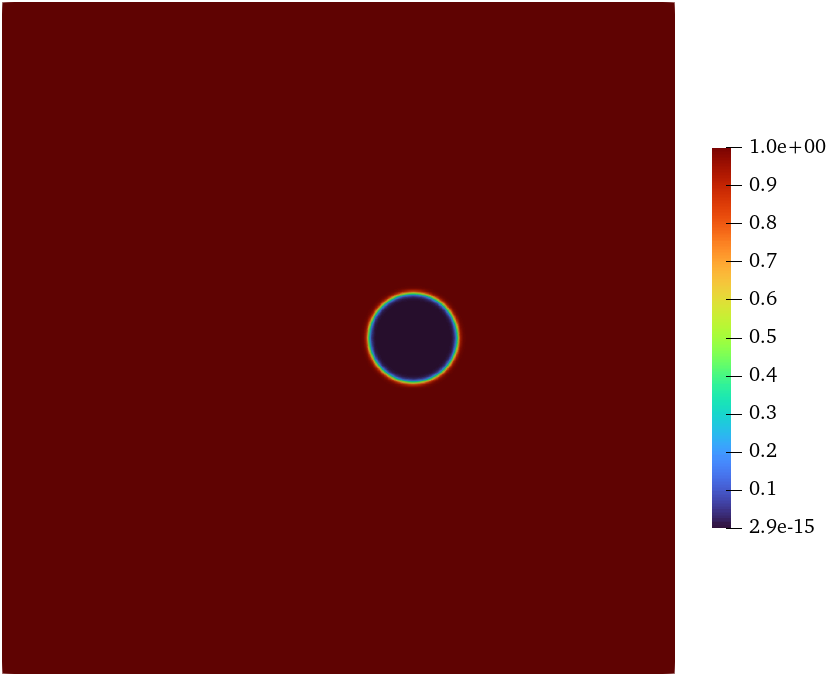}
		\end{minipage}
		\begin{minipage}{0.235\textwidth}
			\hspace*{-0.1cm}
			\includegraphics[width=\linewidth]{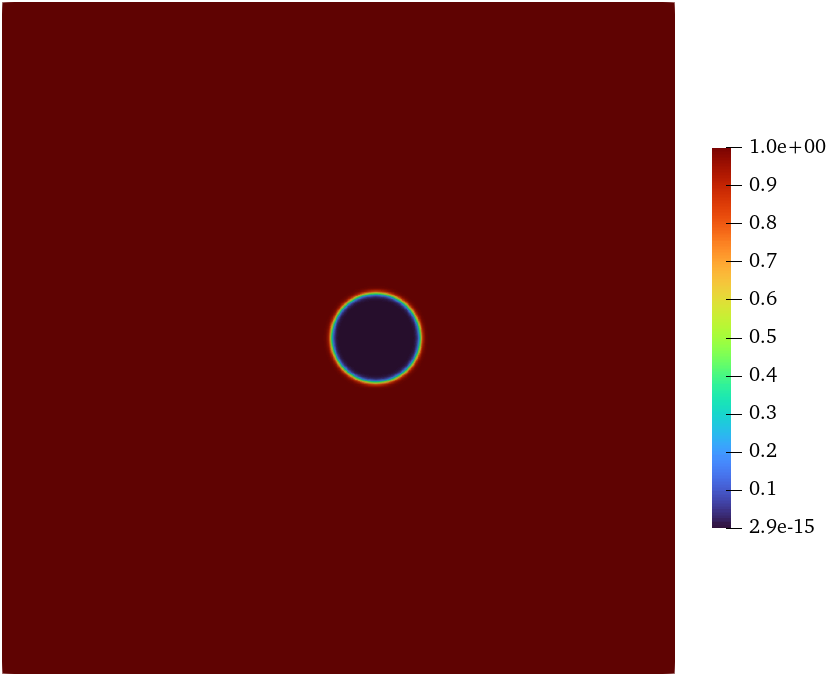}
		\end{minipage}
		\begin{minipage}{0.235\textwidth}
			\hspace*{0.05cm}
			\includegraphics[width=\linewidth]{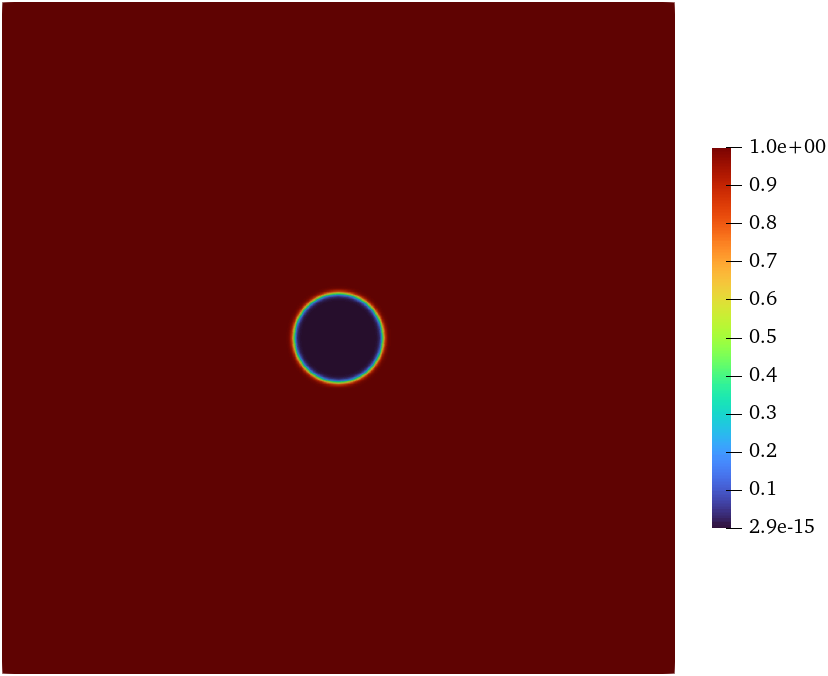}
		\end{minipage}
	\end{center}
	\caption{Pressure snapshots of a moving circular object in the
		moderate-frequency regime under sound-soft (top row) and sound-hard
		(middle row) boundary conditions at $t = 0.2, 0.6, 0.8,$ and $1.0$
		in $\Omega_{\rm phy}$, respectively; the bottom row shows the corresponding
		snapshots of $\psi_\varepsilon$.}\label{fig:moving-simple-circle.}
\end{figure}
\begin{figure}[H]
	\begin{center}
		\includegraphics[width=0.45\textwidth]{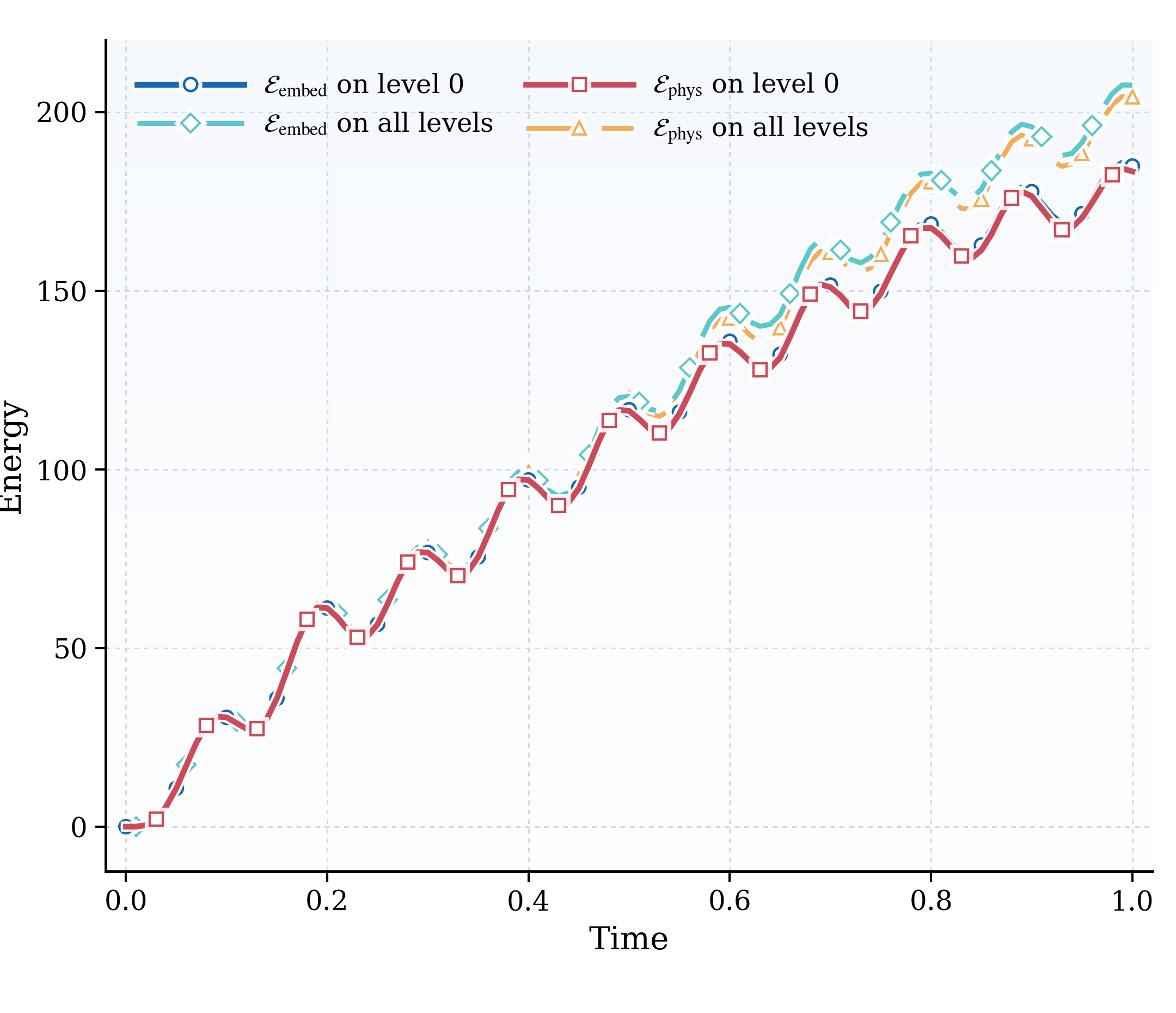}
		\includegraphics[width=0.45\textwidth]{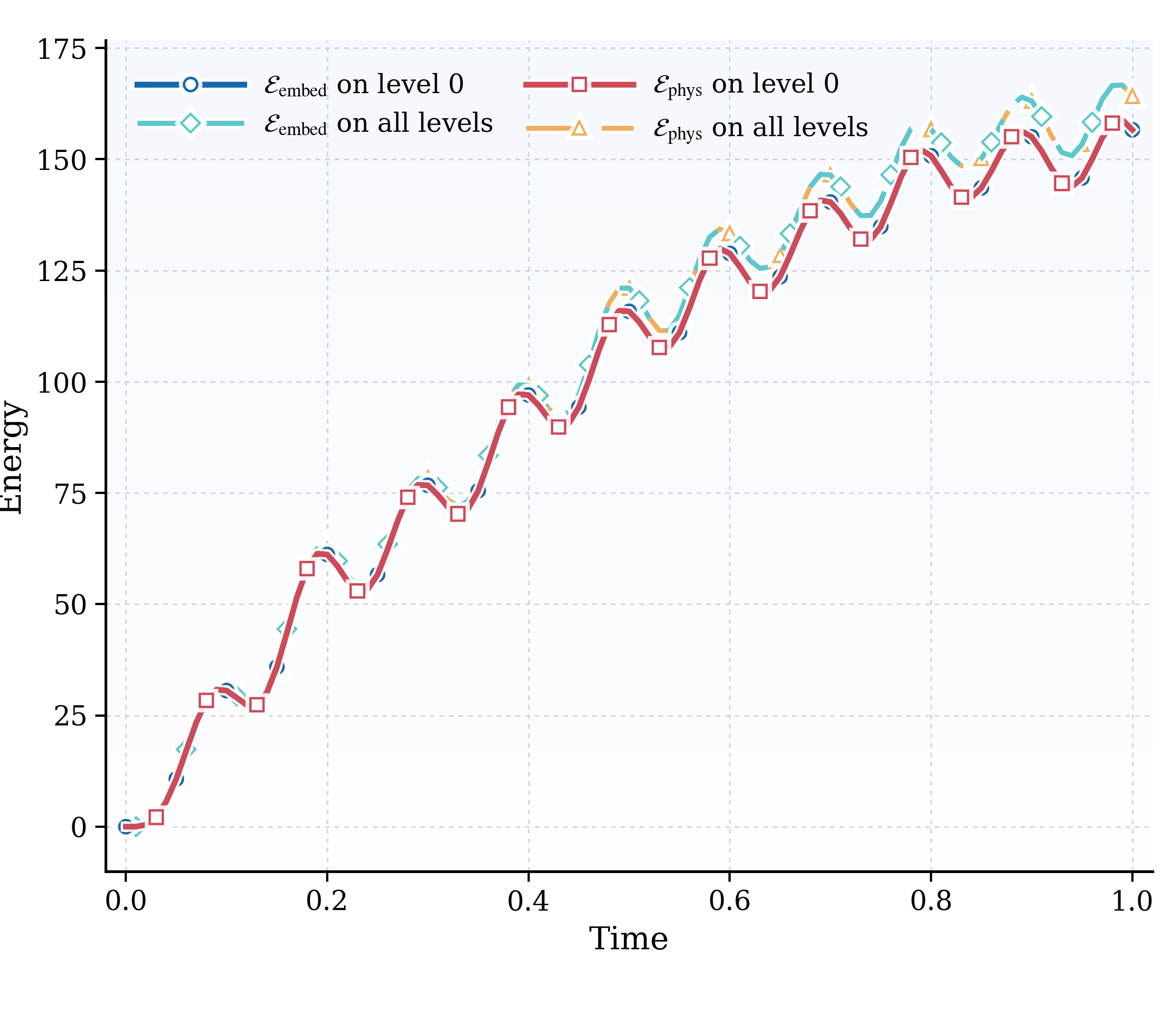}
	\end{center}
	\caption{Energy histories for the moving circular object in the
		moderate-frequency regime. Left: sound-soft case. Right: sound-hard
		case. In each panel, the weighted embedded energy $\mathcal{E}_{\rm
				embed}$ and the physical energy $\mathcal{E}_{\rm phys}$ are
		reported both on the coarsest level and on the composite adaptive
		hierarchy. }\label{fig:energy-circle-low}
\end{figure}

\begin{figure}[H]
	\begin{center}
		\includegraphics[width=0.24\textwidth]{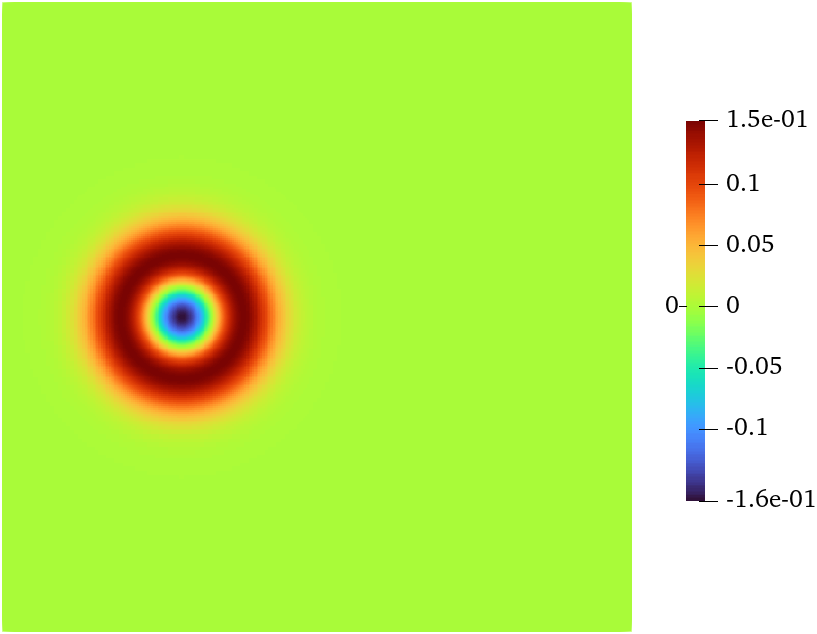}
		\includegraphics[width=0.24\textwidth]{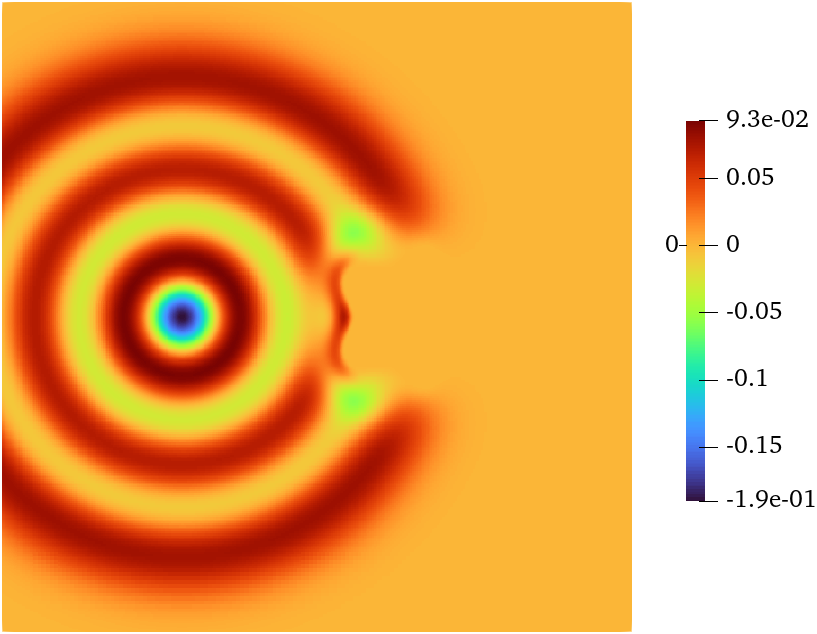}
		\includegraphics[width=0.24\textwidth]{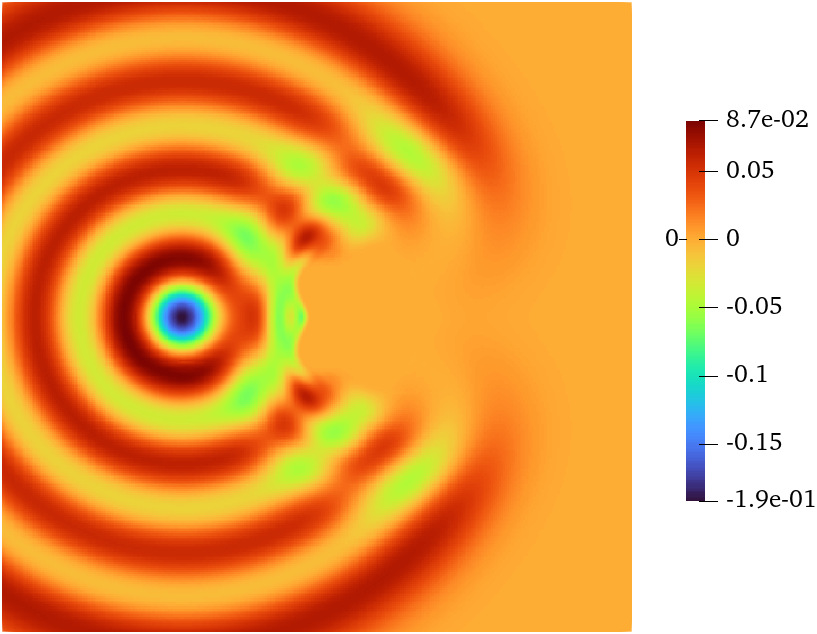}
		\includegraphics[width=0.24\textwidth]{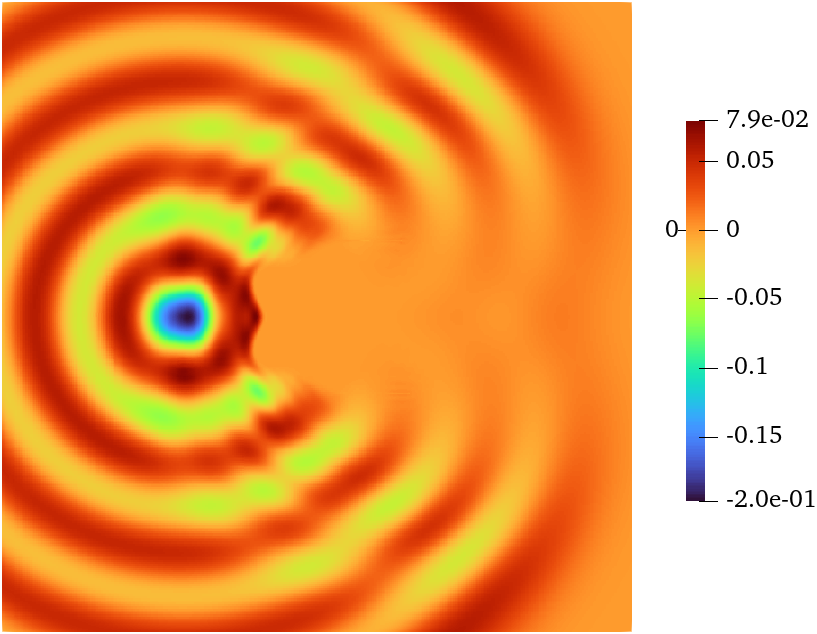}
		\includegraphics[width=0.24\textwidth]{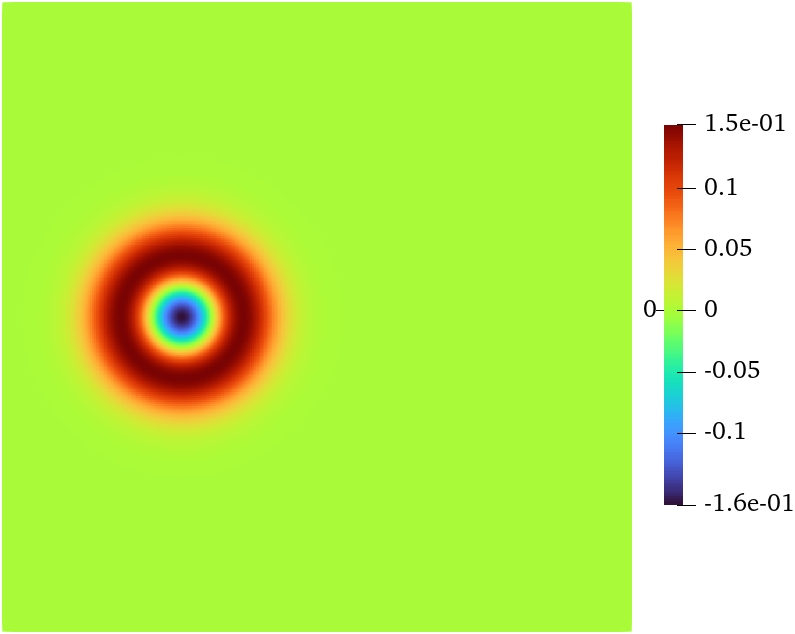}
		\includegraphics[width=0.24\textwidth]{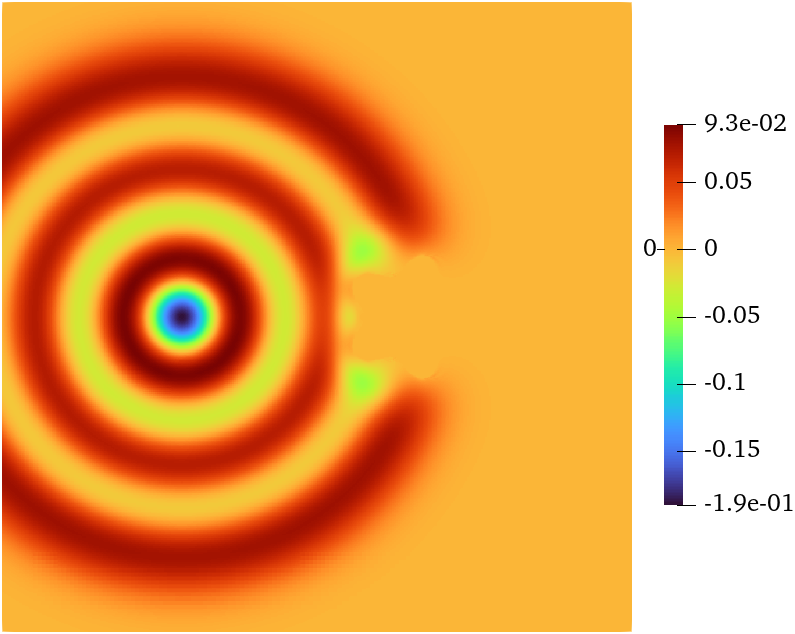}
		\includegraphics[width=0.24\textwidth]{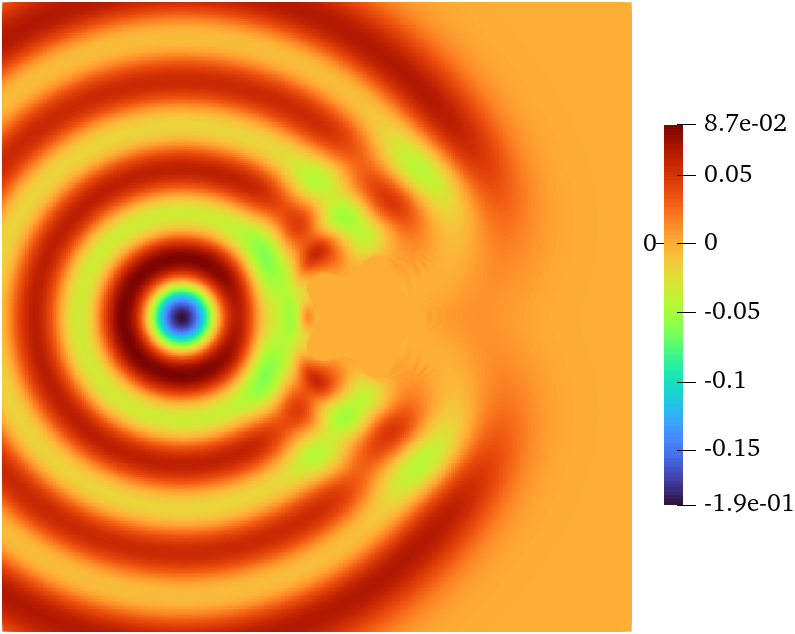}
		\includegraphics[width=0.24\textwidth]{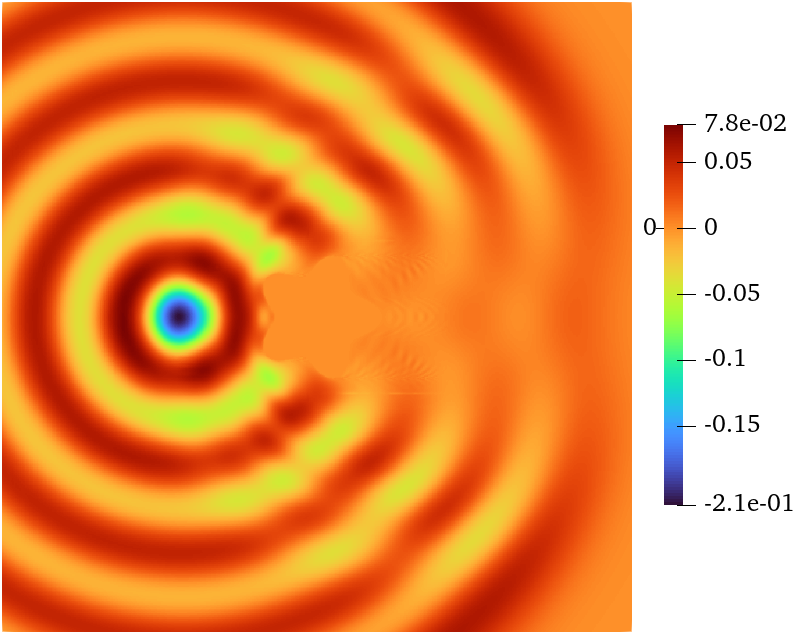}
		\begin{minipage}{0.235\textwidth}
			\hspace*{-0.25cm}
			\includegraphics[width=\linewidth]{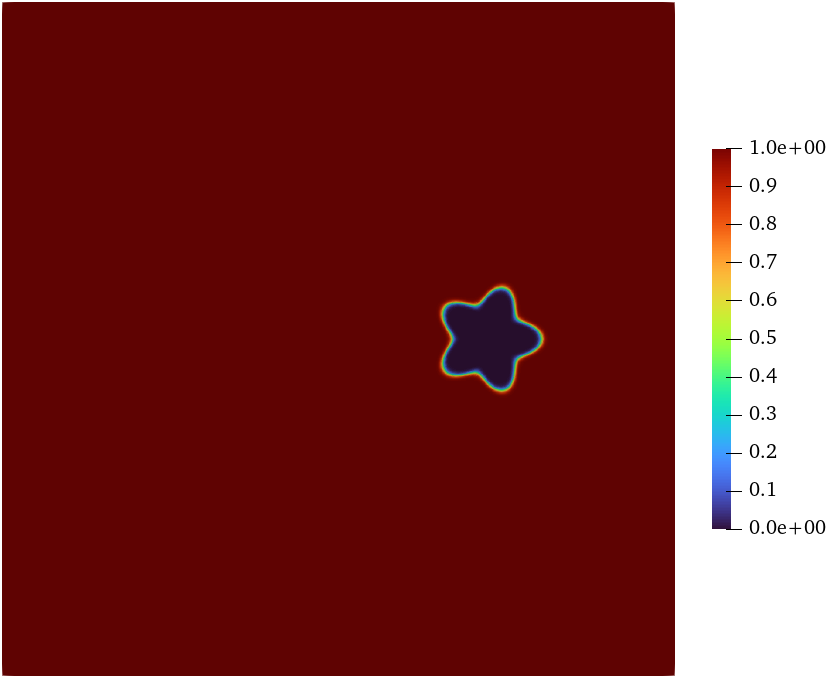}
		\end{minipage}
		\begin{minipage}{0.235\textwidth}
			\hspace*{-0.15cm}
			\includegraphics[width=\linewidth]{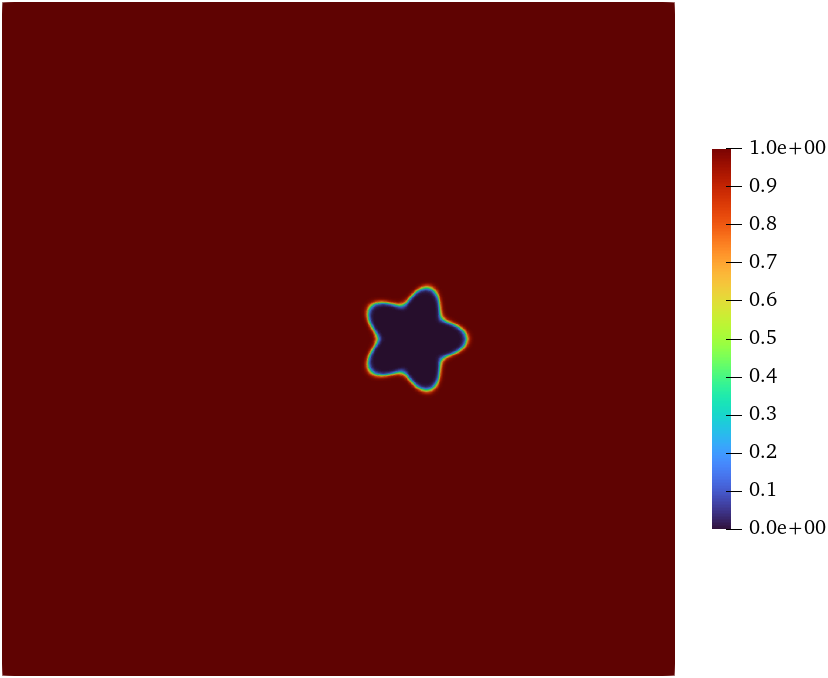}
		\end{minipage}
		\begin{minipage}{0.235\textwidth}
			\hspace*{-0.1cm}
			\includegraphics[width=\linewidth]{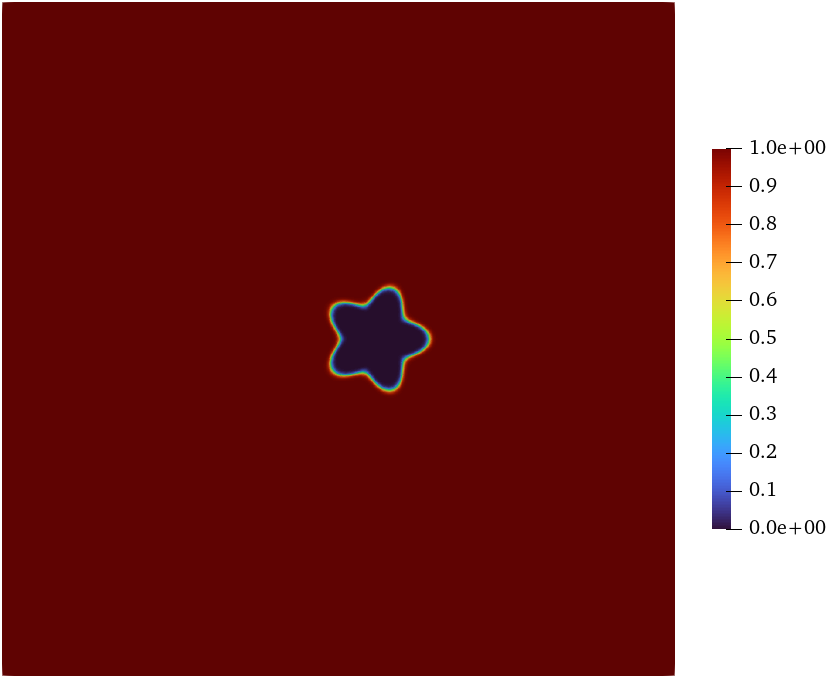}
		\end{minipage}
		\begin{minipage}{0.235\textwidth}
			\hspace*{0.05cm}
			\includegraphics[width=\linewidth]{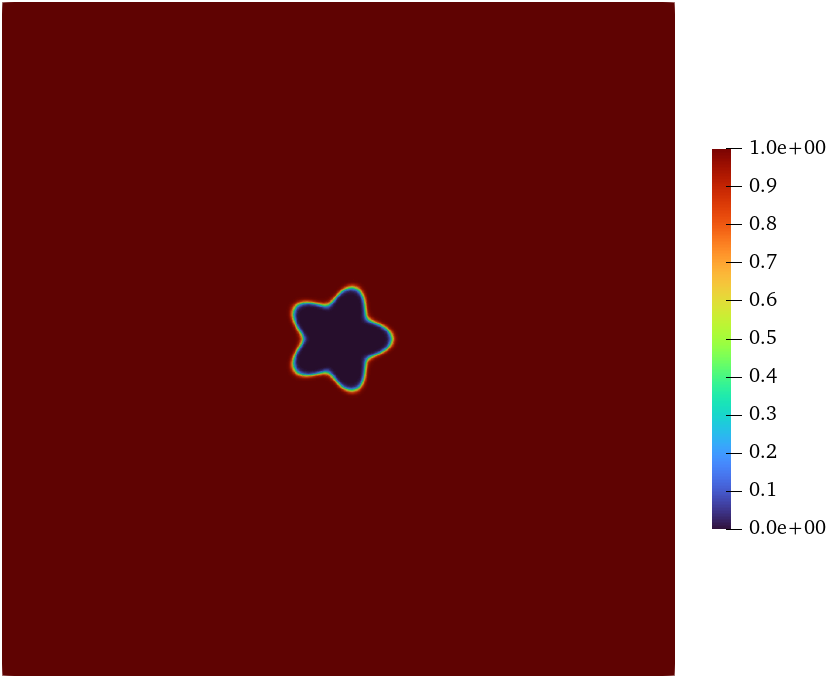}
		\end{minipage}
	\end{center}
	\caption{Pressure snapshots of a moving star-shaped object in the
		moderate-frequency regime under sound-soft (top row) and sound-hard
		(middle row) boundary conditions at $t = 0.2, 0.6, 0.8,$ and $1.0$
		in $\Omega_{\rm phy}$; the bottom row shows the corresponding
		snapshots of $\psi_\varepsilon$.}\label{fig:moving-simple-star.}
\end{figure}
\begin{figure}[H]
	\begin{center}
		\includegraphics[width=0.45\textwidth]{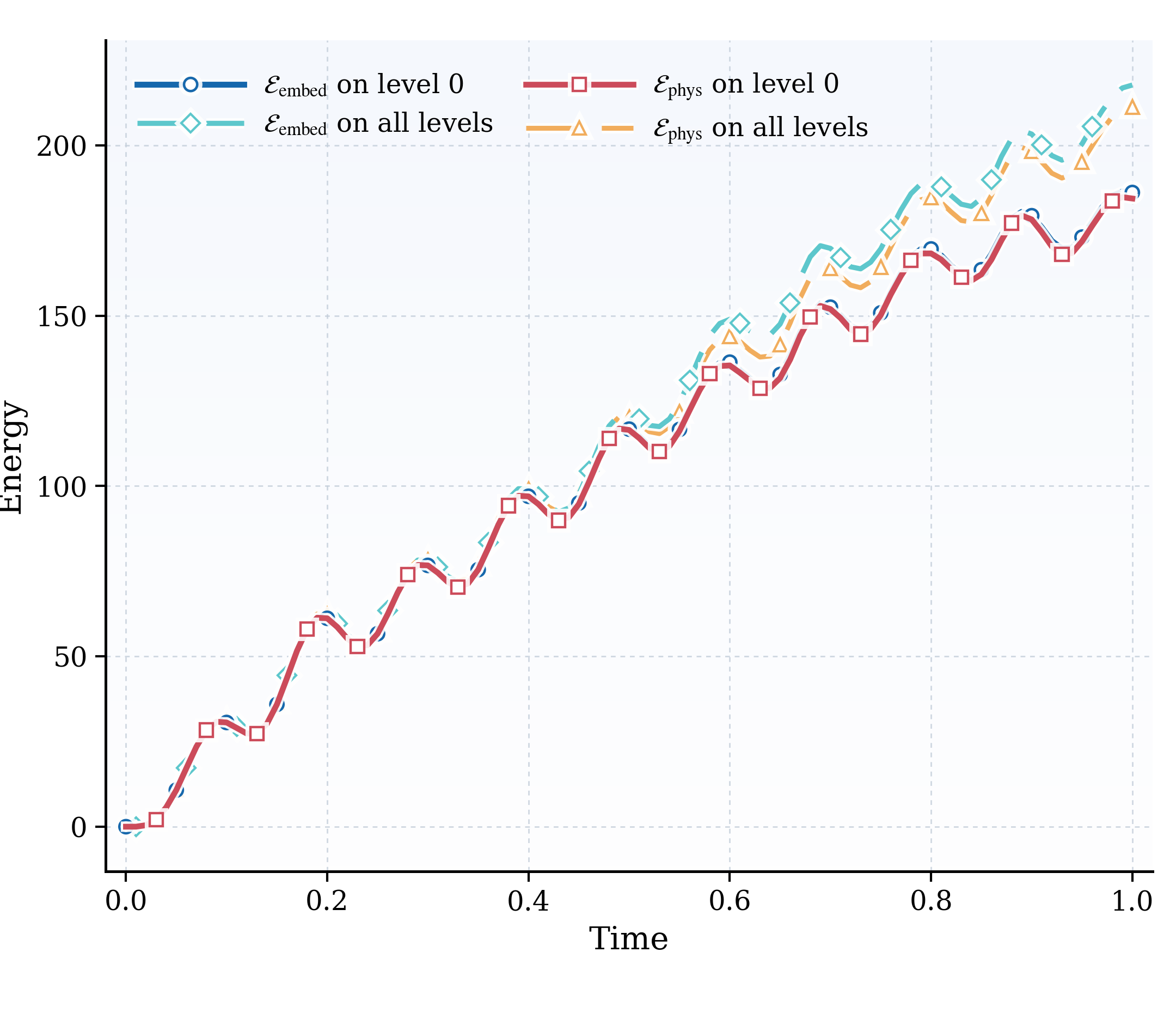}
		\includegraphics[width=0.45\textwidth]{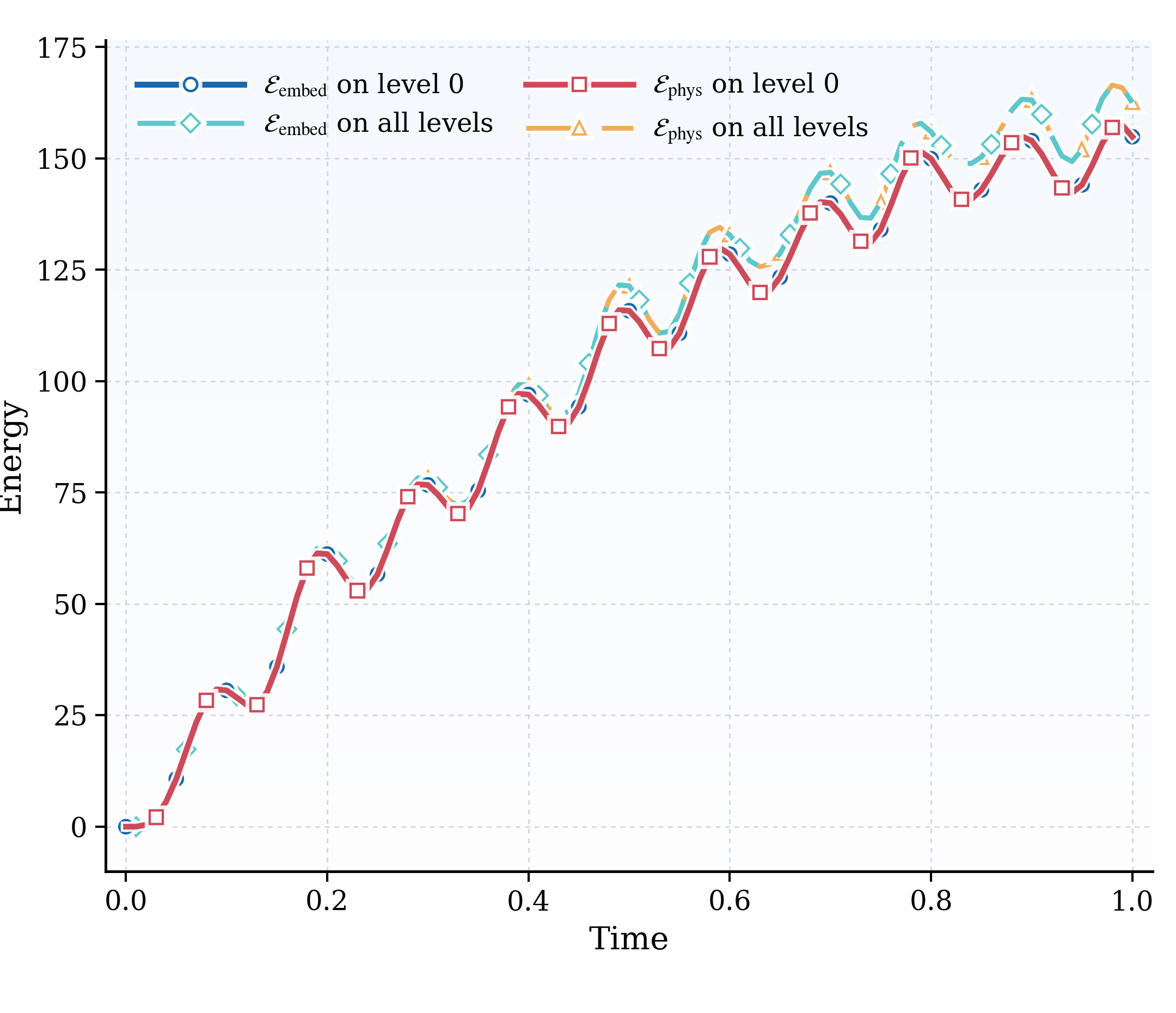}
	\end{center}
	\caption{Energy histories for the moving star-shaped object in
		the moderate-frequency regime. Left: sound-soft case. Right:
		sound-hard case. In each panel, the weighted embedded energy
		$\mathcal{E}_{\rm embed}$ and the physical energy $\mathcal{E}_{\rm
				phys}$ are reported both on the coarsest level and on the composite
		adaptive hierarchy.}\label{fig:energy-star-low}
\end{figure}

\begin{figure}[H]
	\begin{center}
		\includegraphics[width=0.24\textwidth]{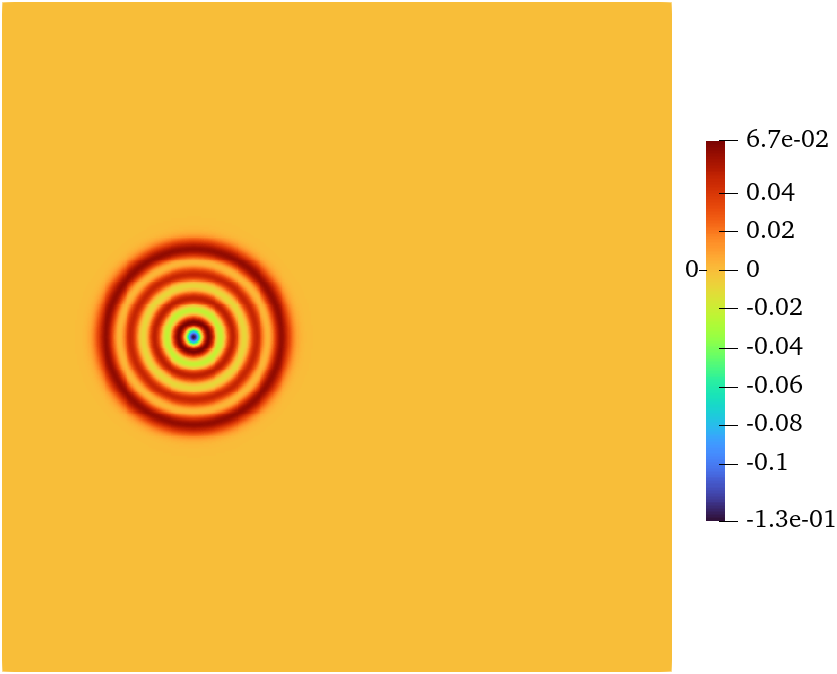}
		\includegraphics[width=0.24\textwidth]{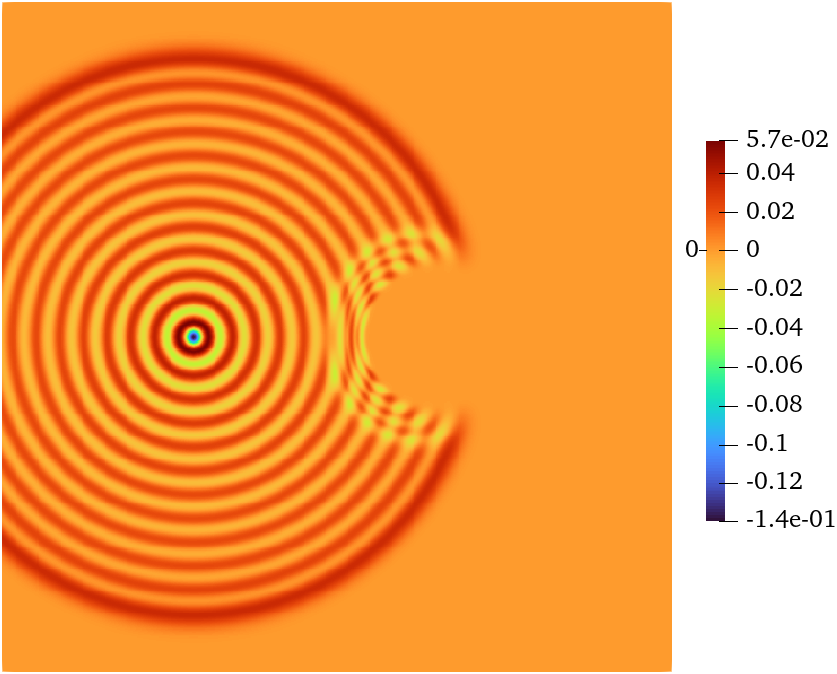}
		\includegraphics[width=0.24\textwidth]{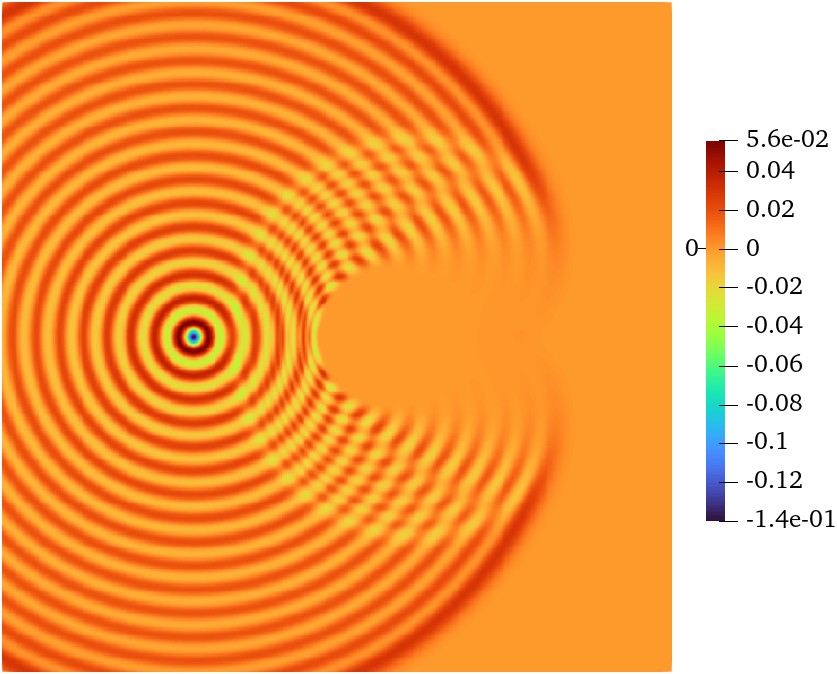}
		\includegraphics[width=0.24\textwidth]{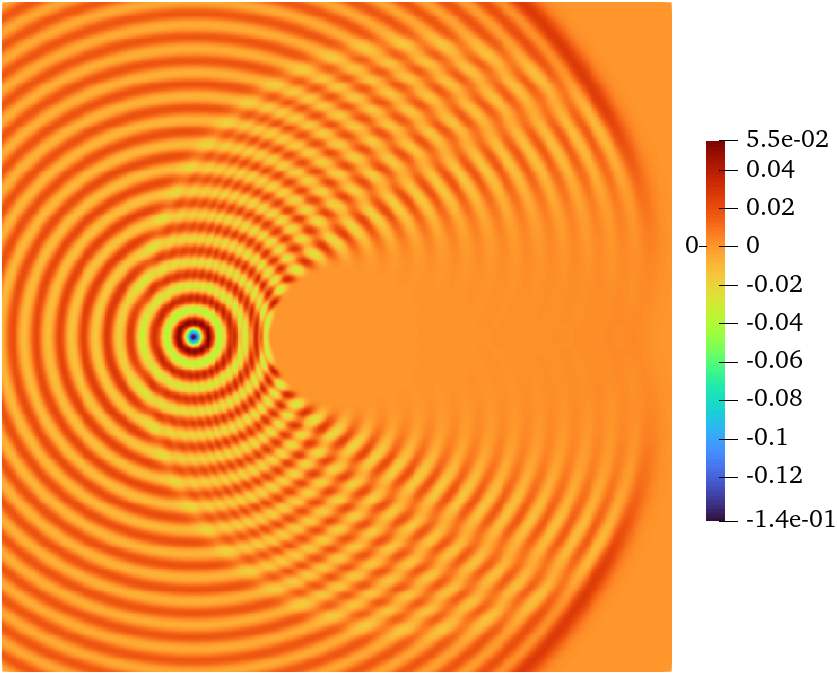}
		\includegraphics[width=0.24\textwidth]{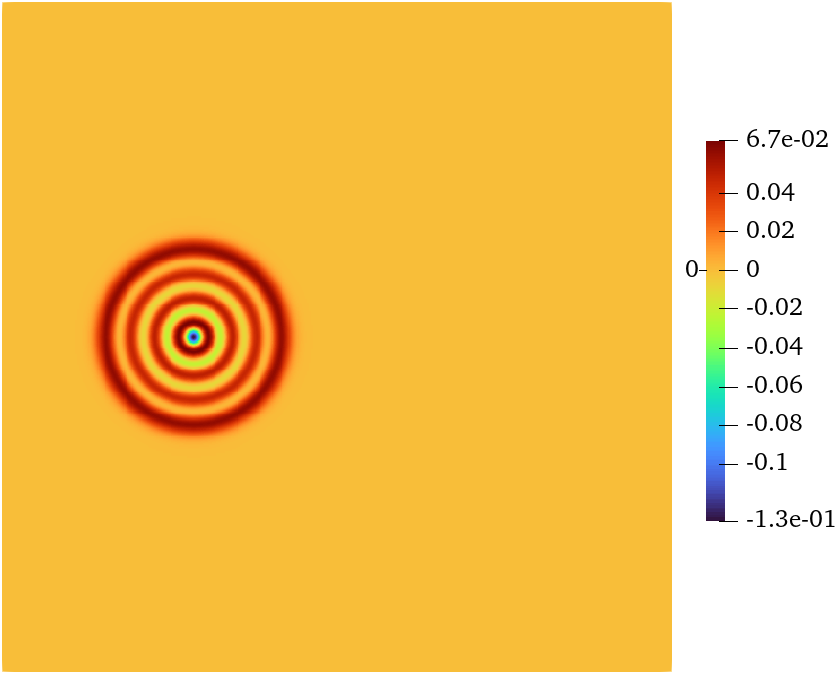}
		\includegraphics[width=0.24\textwidth]{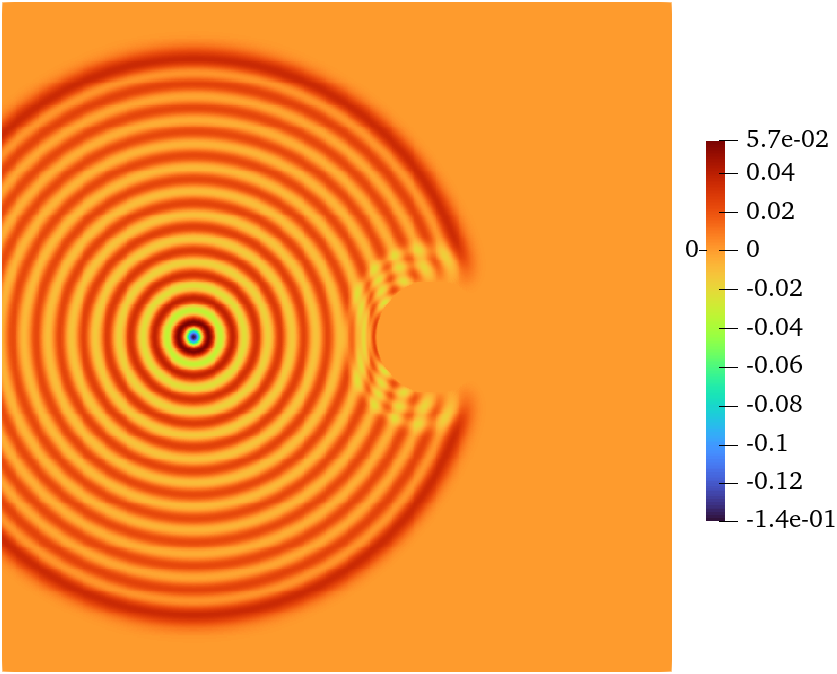}
		\includegraphics[width=0.24\textwidth]{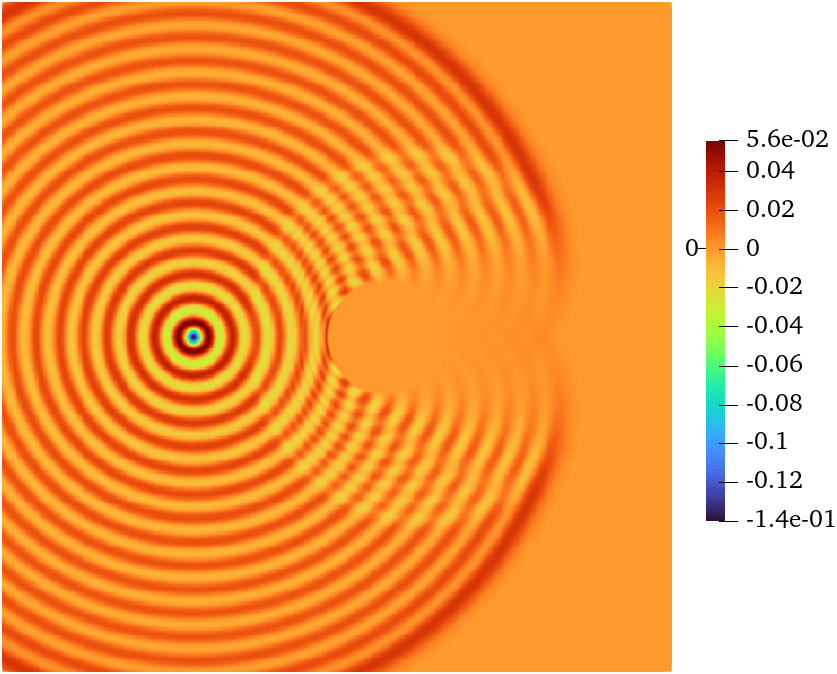}
		\includegraphics[width=0.24\textwidth]{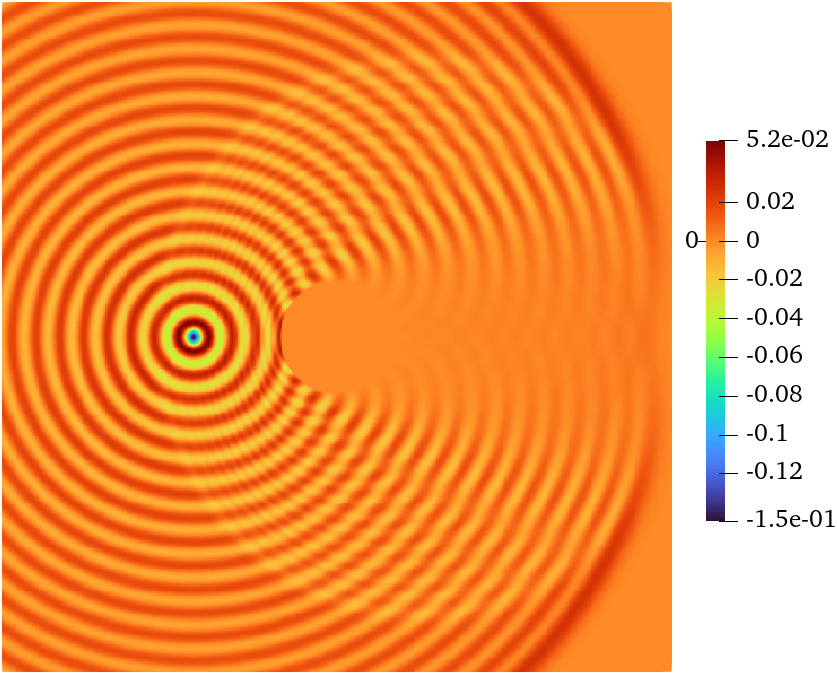}
		\begin{minipage}{0.235\textwidth}
			\hspace*{-0.25cm}
			\includegraphics[width=\linewidth]{./figs/psi_circle.0002.png}
		\end{minipage}
		\begin{minipage}{0.235\textwidth}
			\hspace*{-0.15cm}
			\includegraphics[width=\linewidth]{./figs/psi_circle.0006.png}
		\end{minipage}
		\begin{minipage}{0.235\textwidth}
			\hspace*{-0.1cm}
			\includegraphics[width=\linewidth]{./figs/psi_circle.0008.png}
		\end{minipage}
		\begin{minipage}{0.235\textwidth}
			\hspace*{0.05cm}
			\includegraphics[width=\linewidth]{./figs/psi_circle.0010.png}
		\end{minipage}
	\end{center}
	\caption{Pressure snapshots of a moving circular object in the
		high-frequency regime under sound-soft (top row) and sound-hard
		(middle row) boundary conditions at $t = 0.2, 0.6, 0.8,$ and $1.0$
		in $\Omega_{\rm phy}$; the bottom row shows the corresponding
		snapshots of $\psi_\varepsilon$.}\label{fig:moving-simple-circle_high.}
\end{figure}
\begin{figure}[H]
	\begin{center}
		\includegraphics[width=0.45\textwidth]{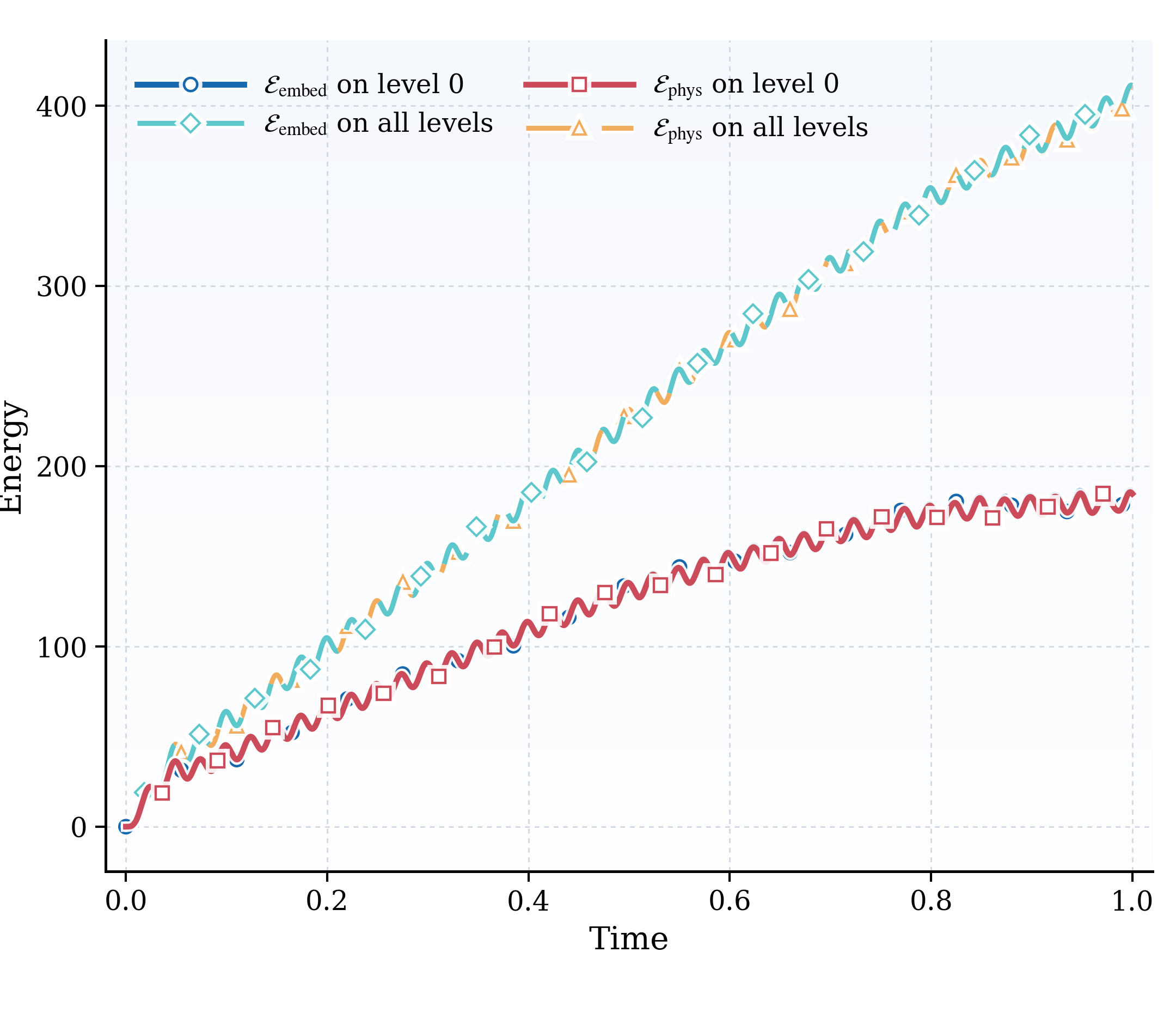}
		\includegraphics[width=0.45\textwidth]{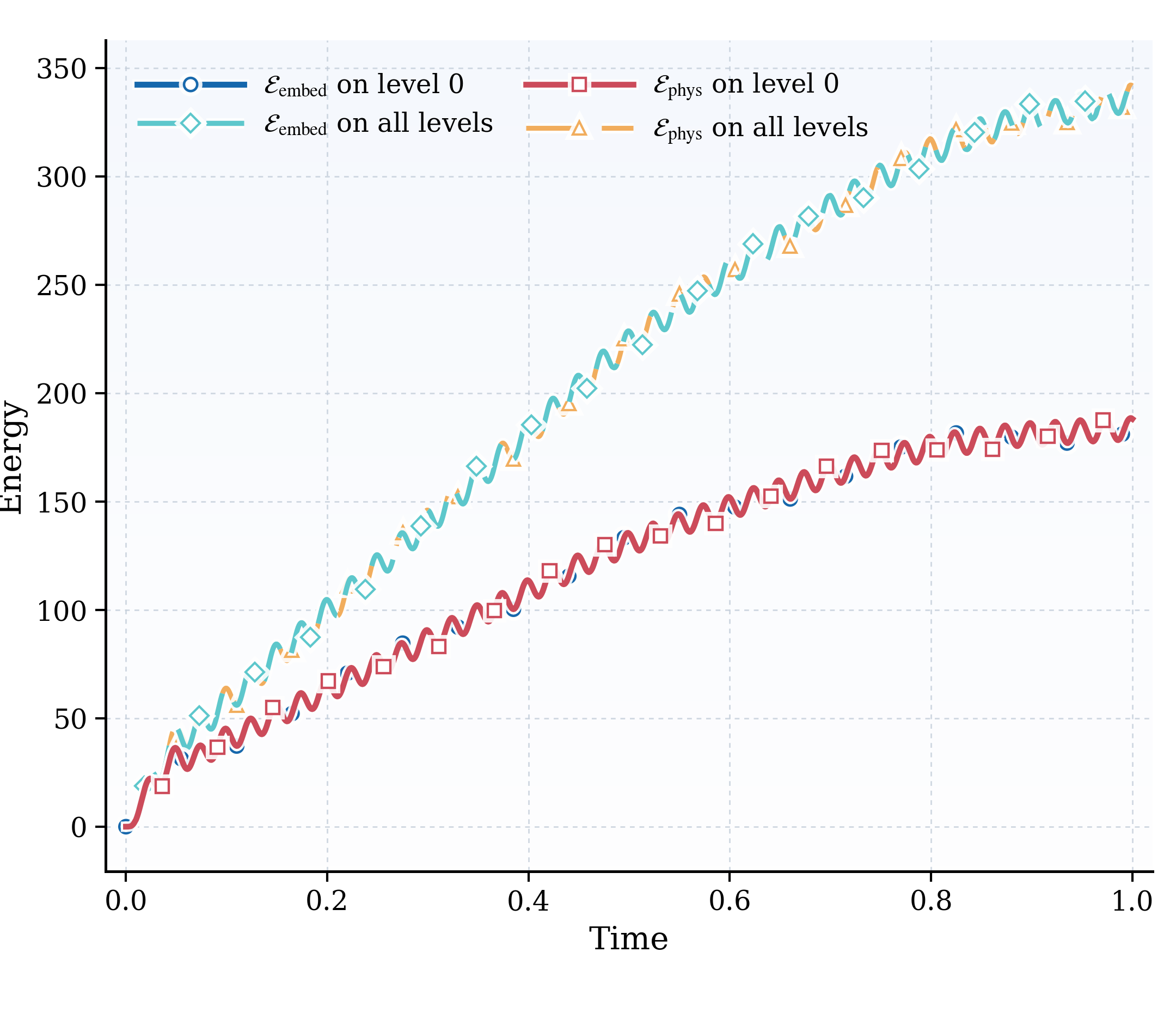}
	\end{center}
	\caption{Energy histories for the moving circular object in the
		high-frequency regime. Left: sound-soft case. Right: sound-hard
		case. In each panel, the weighted embedded energy $\mathcal{E}_{\rm
				embed}$ and the physical energy $\mathcal{E}_{\rm phys}$ are
		reported both on the coarsest level and on the composite adaptive
		hierarchy.}\label{fig:energy-circle-high}
\end{figure}

\begin{figure}[H]
	\begin{center}
		\includegraphics[width=0.24\textwidth]{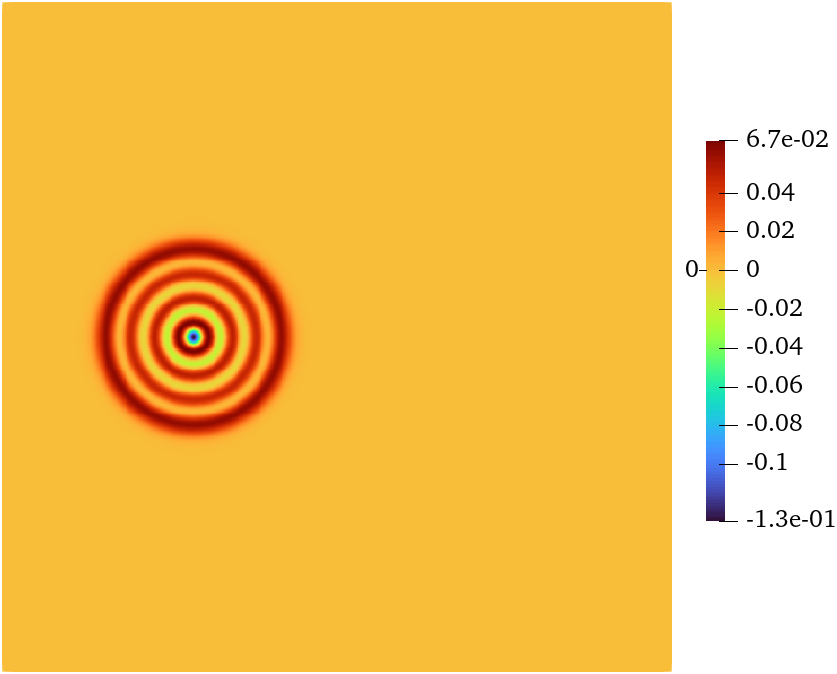}
		\includegraphics[width=0.24\textwidth]{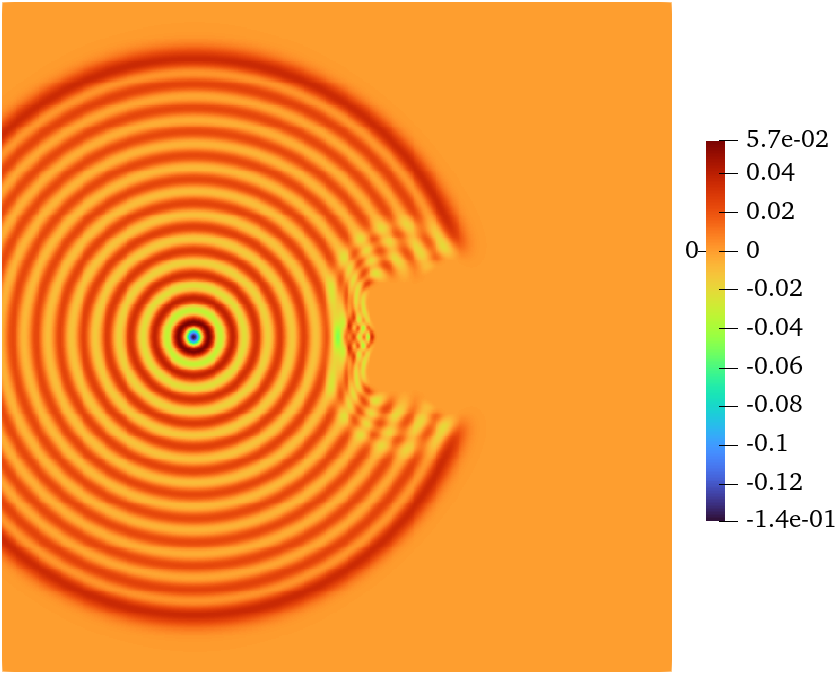}
		\includegraphics[width=0.24\textwidth]{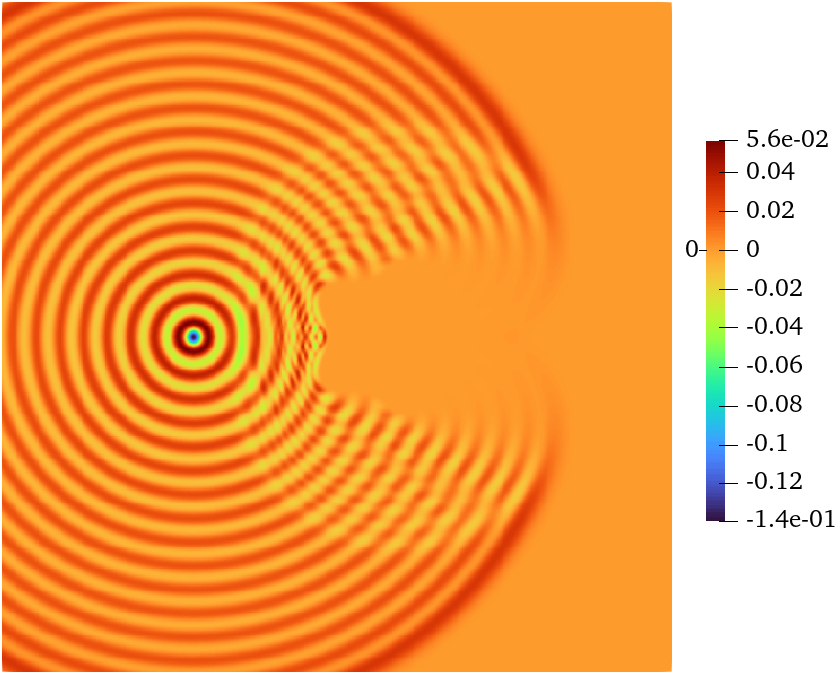}
		\includegraphics[width=0.24\textwidth]{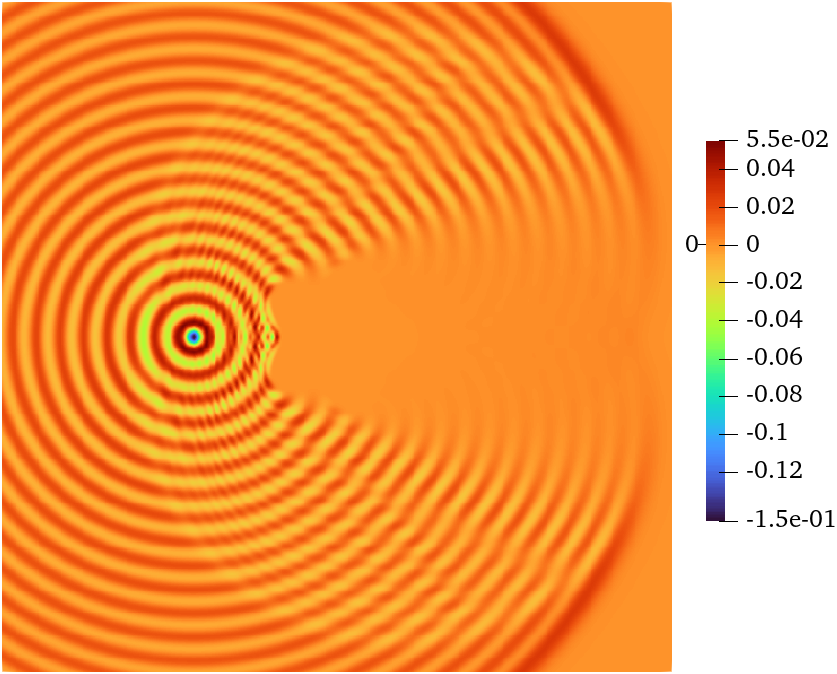}
		\includegraphics[width=0.24\textwidth]{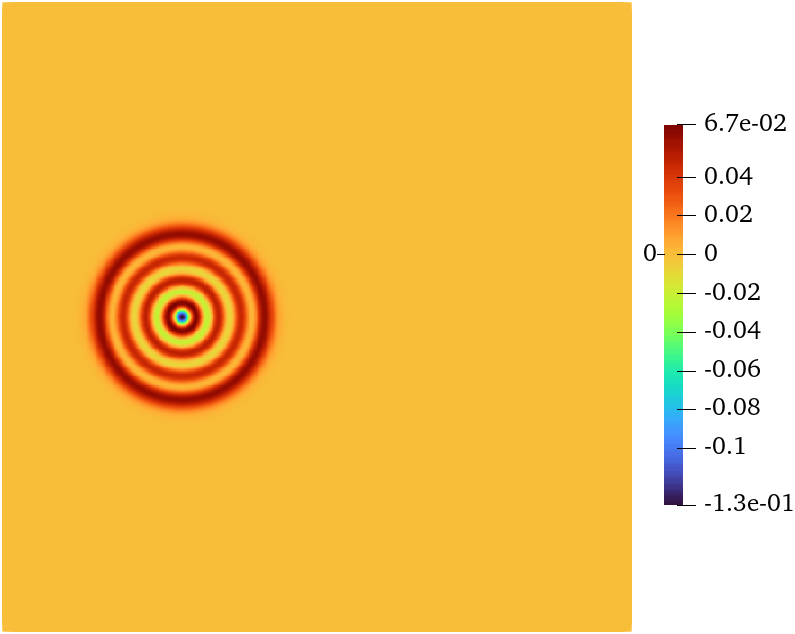}
		\includegraphics[width=0.24\textwidth]{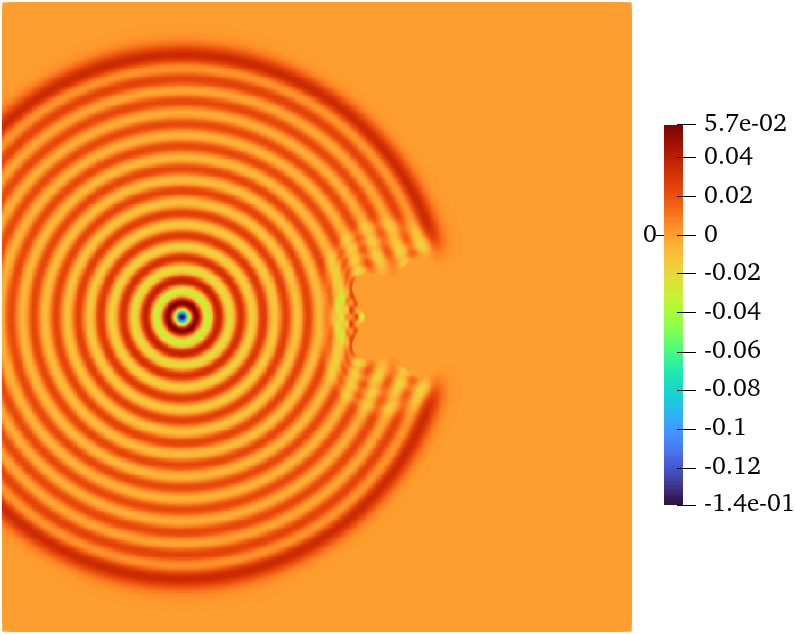}
		\includegraphics[width=0.24\textwidth]{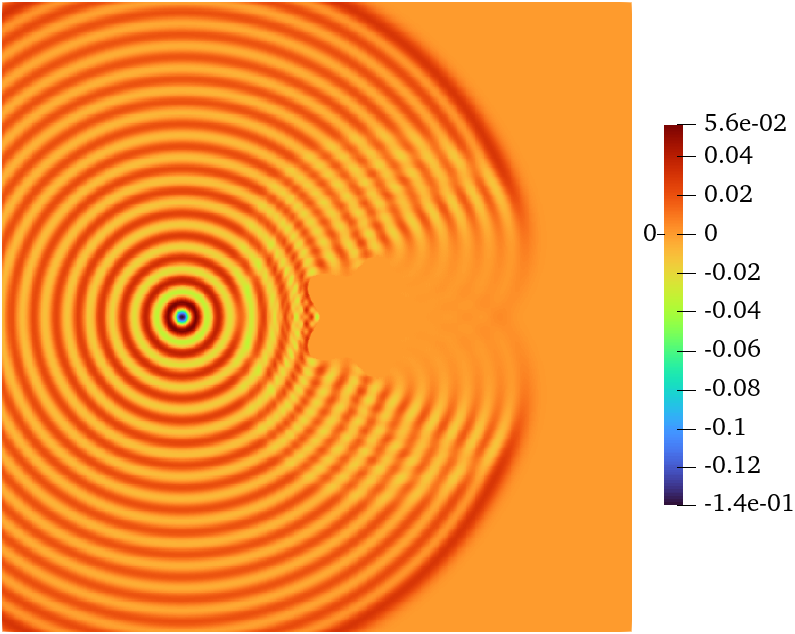}
		\includegraphics[width=0.24\textwidth]{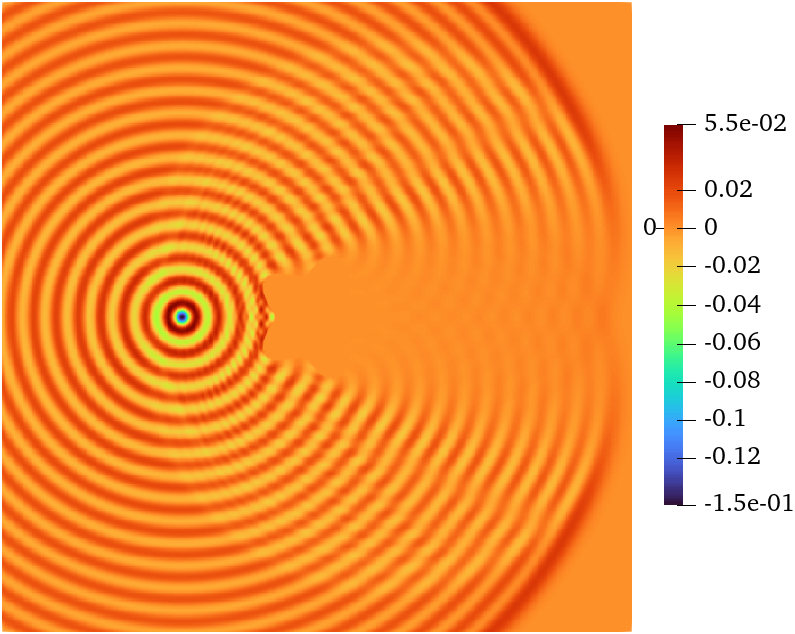}
		\begin{minipage}{0.235\textwidth}
			\hspace*{-0.25cm}
			\includegraphics[width=\linewidth]{./figs/psi_star.0002.png}
		\end{minipage}
		\begin{minipage}{0.235\textwidth}
			\hspace*{-0.15cm}
			\includegraphics[width=\linewidth]{./figs/psi_star.0006.png}
		\end{minipage}
		\begin{minipage}{0.235\textwidth}
			\hspace*{-0.1cm}
			\includegraphics[width=\linewidth]{./figs/psi_star.0008.png}
		\end{minipage}
		\begin{minipage}{0.235\textwidth}
			\hspace*{0.05cm}
			\includegraphics[width=\linewidth]{./figs/psi_star.0010.png}
		\end{minipage}
	\end{center}
	\caption{Pressure snapshots of a moving star-shaped object in the
		high-frequency regime under sound-soft (top row) and sound-hard
		(middle row) boundary conditions at $t = 0.2, 0.6, 0.8,$ and $1.0$
		in $\Omega_{\rm phy}$; the bottom row shows the corresponding
		snapshots of $\psi_\varepsilon$.}\label{fig:moving-simple-star_high.}
\end{figure}
\begin{figure}[H]
	\begin{center}
		\includegraphics[width=0.45\textwidth]{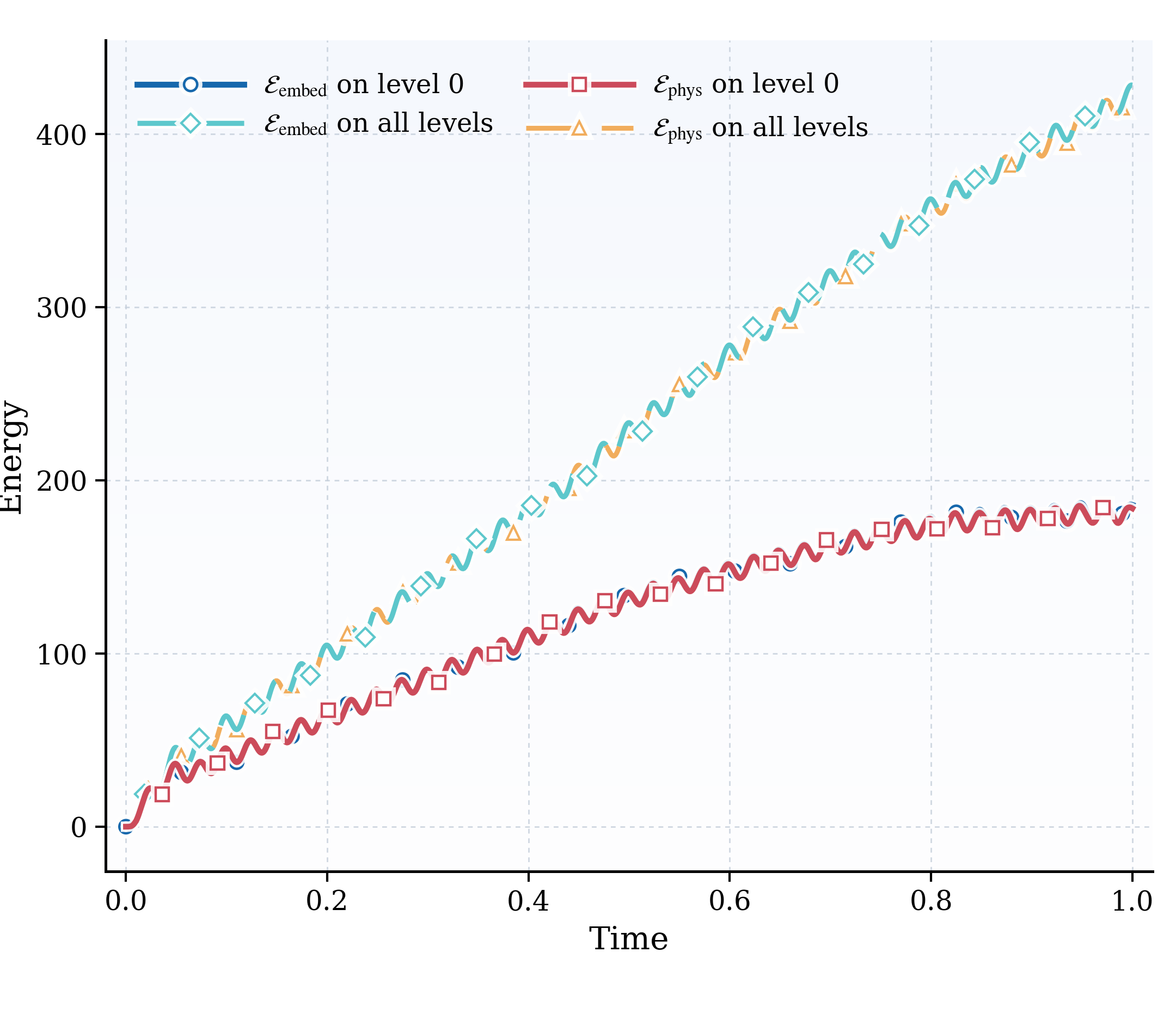}
		\includegraphics[width=0.45\textwidth]{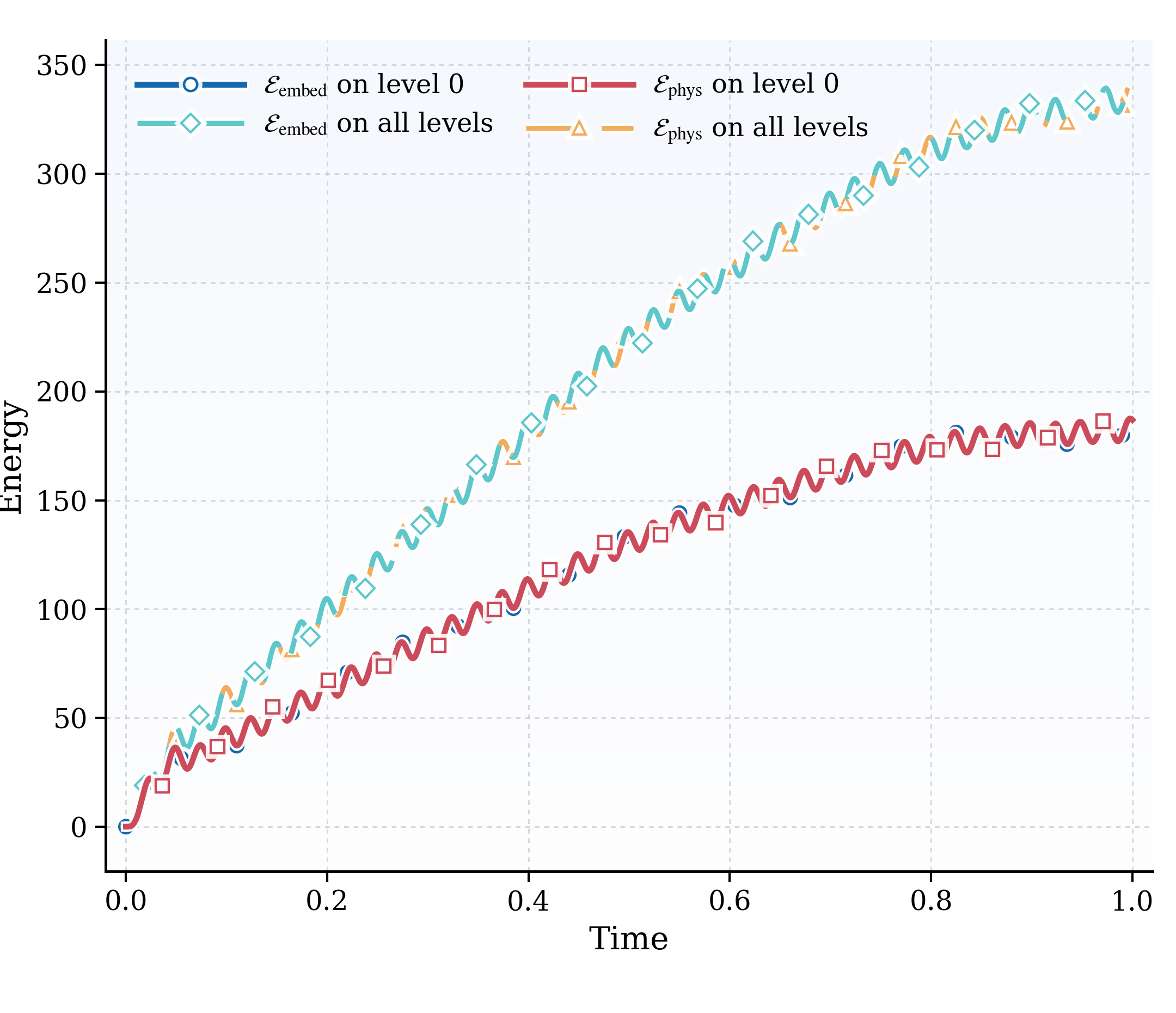}
	\end{center}
	\caption{Energy histories for the moving star-shaped object in
		the high-frequency regime. Left: sound-soft case. Right: sound-hard
		case. In each panel, the weighted embedded energy $\mathcal{E}_{\rm
				embed}$ and the physical energy $\mathcal{E}_{\rm phys}$ are
		reported both on the coarsest level and on the composite adaptive
		hierarchy.}\label{fig:energy-star-high}
\end{figure}


For the first parameter set, Figures~\ref{fig:moving-simple-circle.} and
\ref{fig:moving-simple-star.} display the scattered field generated by an incident wave packet of
moderate-frequency  for moving circular and
star-shaped objects, respectively. The bottom rows of these figures show the
corresponding snapshots of $\psi_\varepsilon$ and confirm that the
moving geometry is transported smoothly in the fixed computational domain.
After the incoming wave packet reaches the object, reflected fronts,
diffracted waves, and a distinct downstream shadow region develop in
both geometries. The comparison between the sound-soft and sound-hard
rows shows that the present framework accommodates both boundary
conditions without altering the grid representation; the most visible
difference lies in the near field, where the phase and amplitude of the
reflected pattern depend on the boundary type. In addition, the
star-shaped object induces a richer local distortion than the circle,
which is consistent with the stronger scattering effect of its corners
and narrow features.

The corresponding energy trajectories are shown in
Figures~\ref{fig:energy-circle-low} and \ref{fig:energy-star-low}. In
contrast to the static test, the energies are not expected to decay
monotonically because the wave source remains active and continuously
injects energy into the domain. The relevant observation here is that
$\mathcal{E}_{\rm embed}$ and $\mathcal{E}_{\rm phys}$ remain in close
agreement for each boundary condition and each geometry, indicating that
the weighted energy computed from the embedded model continues to track the ground truth physical energy even
in the presence of object motion. We calculate the energy in the coarse-grained grid and the adaptive gird, respectively, and refer it as the energy at level 0 and all levels. The all-level energy  curves lie above their
level-0 counterparts because they include the contribution of the
refined patches that resolve the moving interface and the dominant wave
fronts. In the moderate-frequency regime this gap remains relatively
small, consistent with the broader and smoother character of the
incident packet.

The second parameter set corresponds to a more oscillatory incident wave
that generates a significantly finer wavefront structure.
Figures~\ref{fig:moving-simple-circle_high.} and
\ref{fig:moving-simple-star_high.} show that, in this regime, the
influence of object geometry becomes even more pronounced. The moving
circular object still produces comparatively smooth reflected rings,
whereas the star-shaped object gives rise to richer interference
patterns and more visible local perturbations near the body. The
high-frequency tests also make the role of adaptive refinement more
important and indispensable, since the finer oscillations activate a larger portion of the
mesh hierarchy. This behavior is reflected in the energy trajectories in
Figures~\ref{fig:energy-circle-high} and \ref{fig:energy-star-high},
where the separation between the level-0 and all-level energy curves is more
pronounced than in the moderate-frequency regime. Nevertheless, the
energy computed from the embedded model and the physical energy remain qualitatively consistent across
all cases, and no visible spurious reflection is generated at the outer
boundary. These results indicate that the PML continues to function
effectively even when coupled with a time-dependent moving object using the embedding approach.

\begin{figure}
	\begin{center}
		\includegraphics[width=0.24\textwidth]{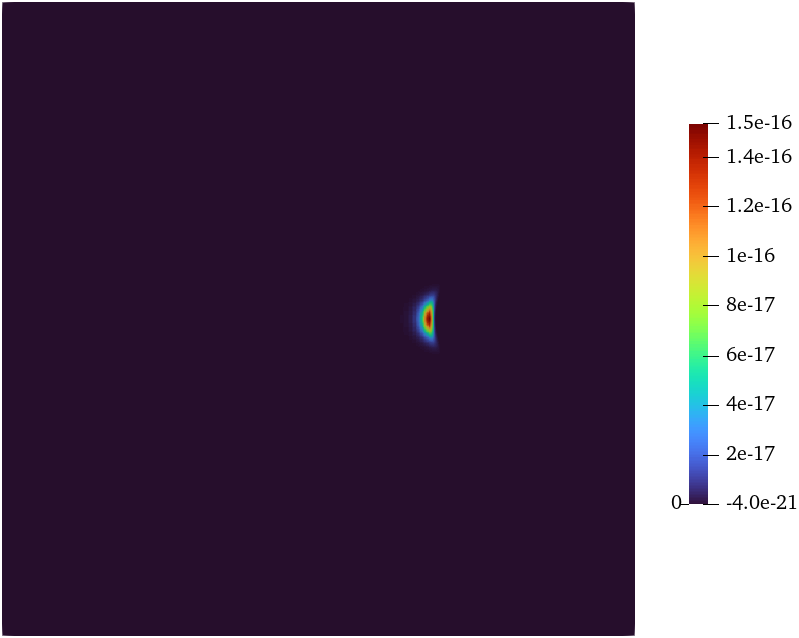}
		\includegraphics[width=0.24\textwidth]{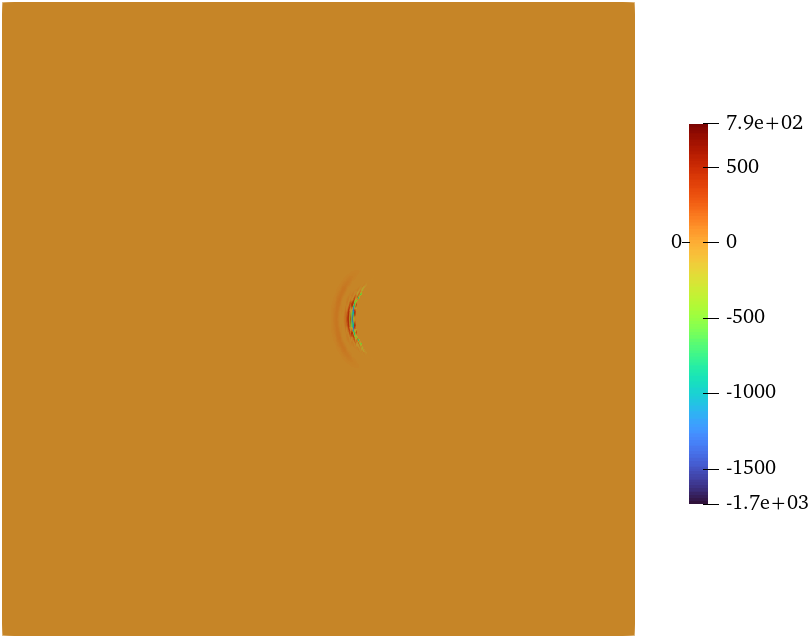}
		\includegraphics[width=0.24\textwidth]{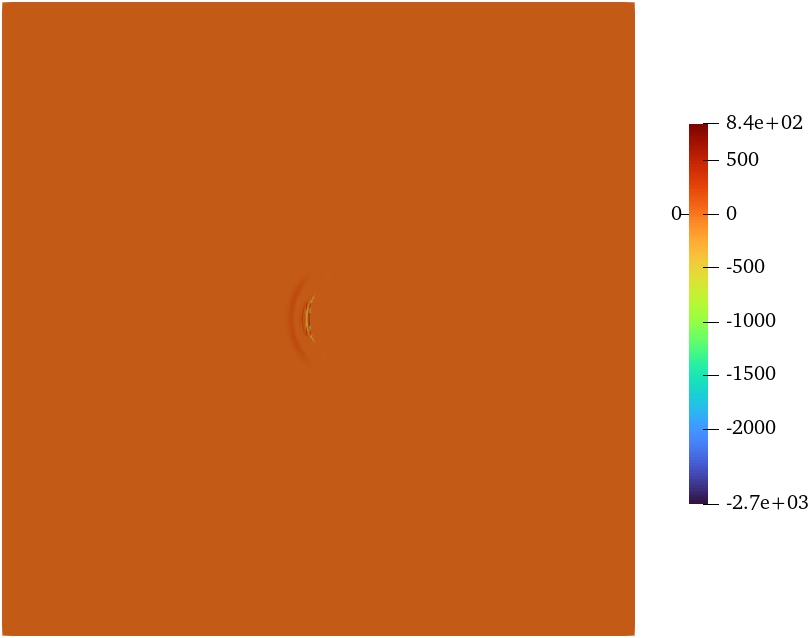}
		\includegraphics[width=0.24\textwidth]{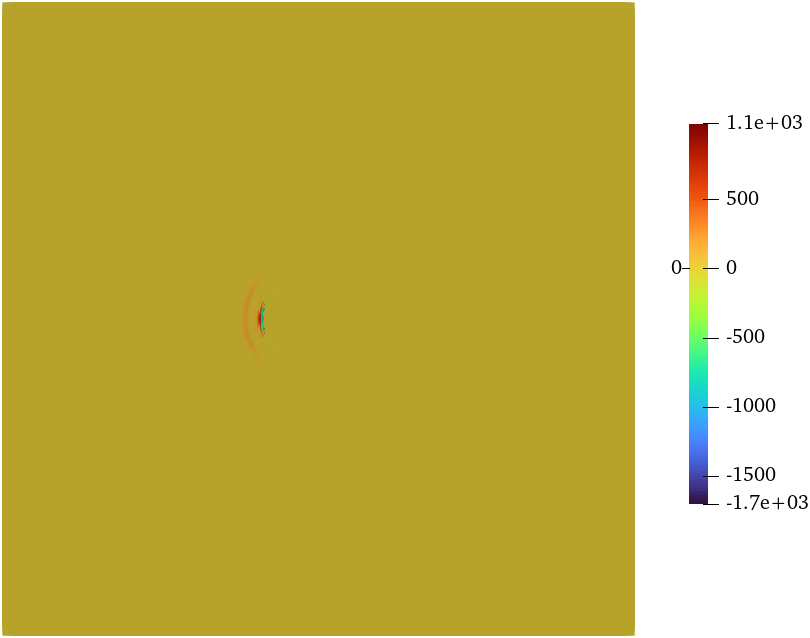}
		\begin{minipage}{0.23\textwidth}
			\hspace*{-0.45cm}
			\includegraphics[width=\linewidth]{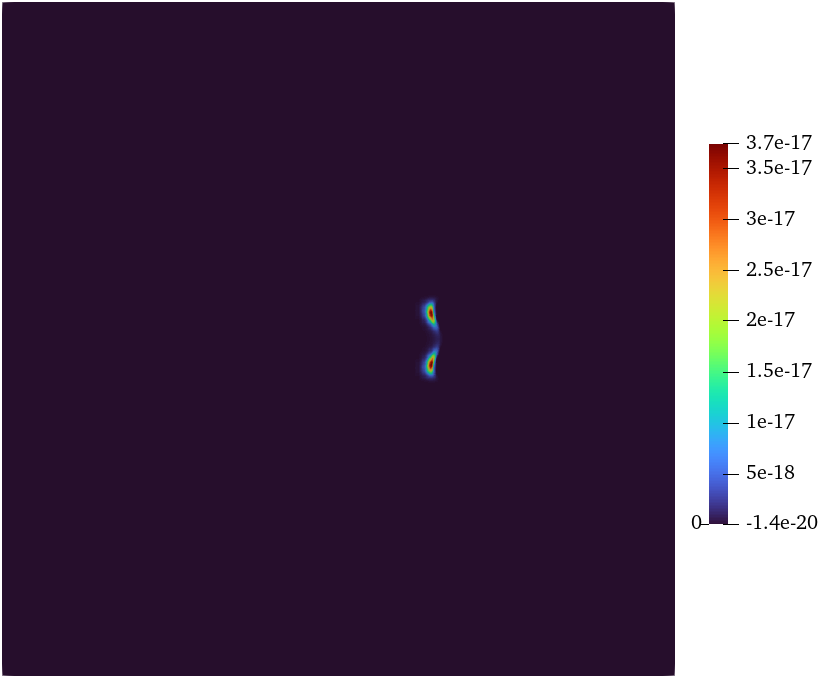}
		\end{minipage}
		\begin{minipage}{0.23\textwidth}
			\hspace*{-0.25cm}
			\includegraphics[width=\linewidth]{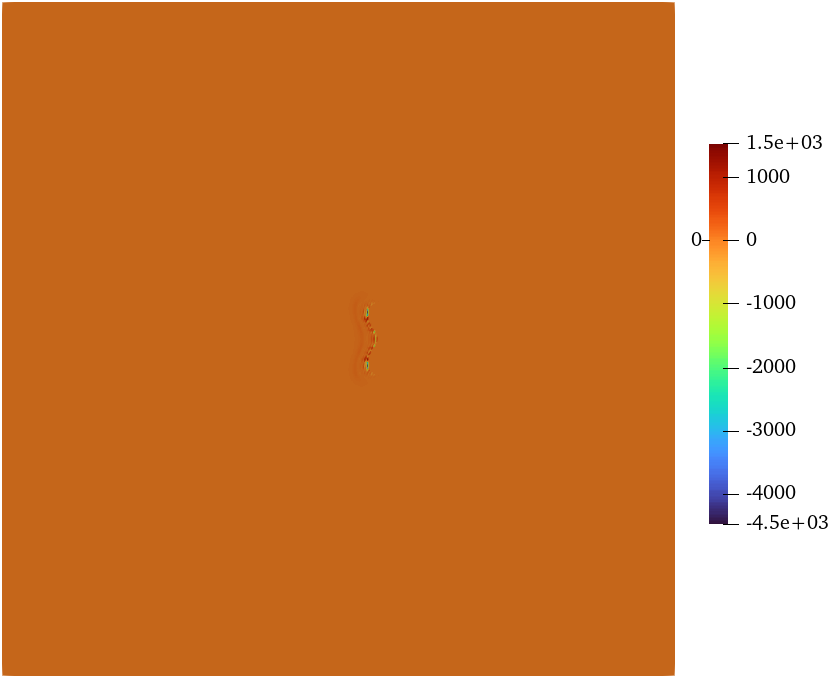}
		\end{minipage}
		\begin{minipage}{0.23\textwidth}
			\hspace*{-0.1cm}
			\includegraphics[width=\linewidth]{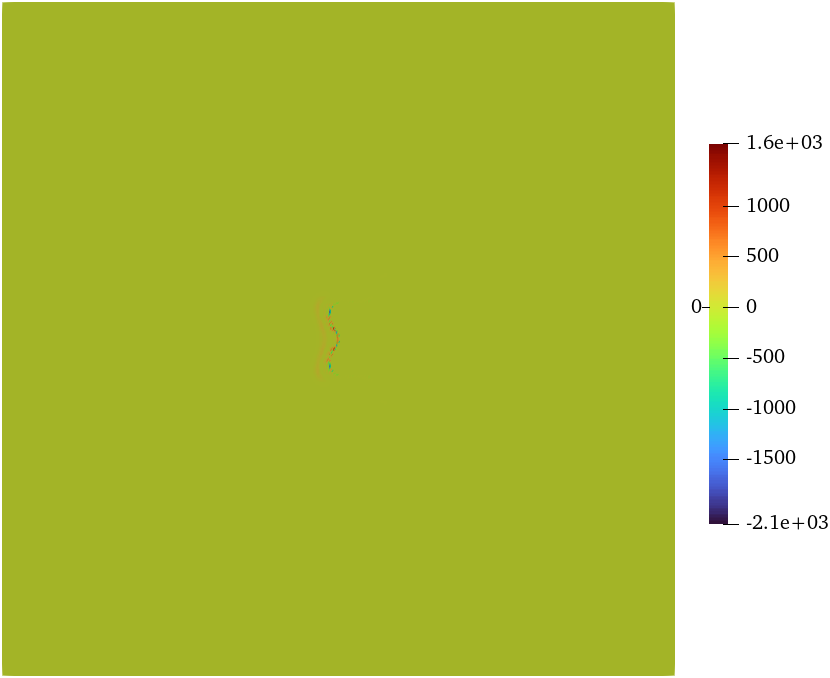}
		\end{minipage}
		\begin{minipage}{0.23\textwidth}
			\hspace*{0.15cm}
			\includegraphics[width=\linewidth]{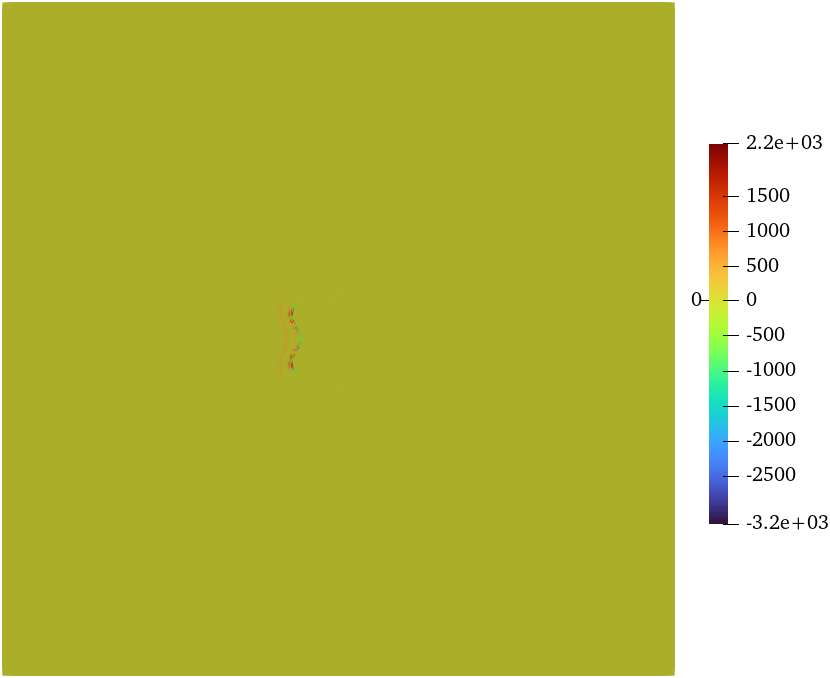}
		\end{minipage}
		\includegraphics[width=0.24\textwidth]{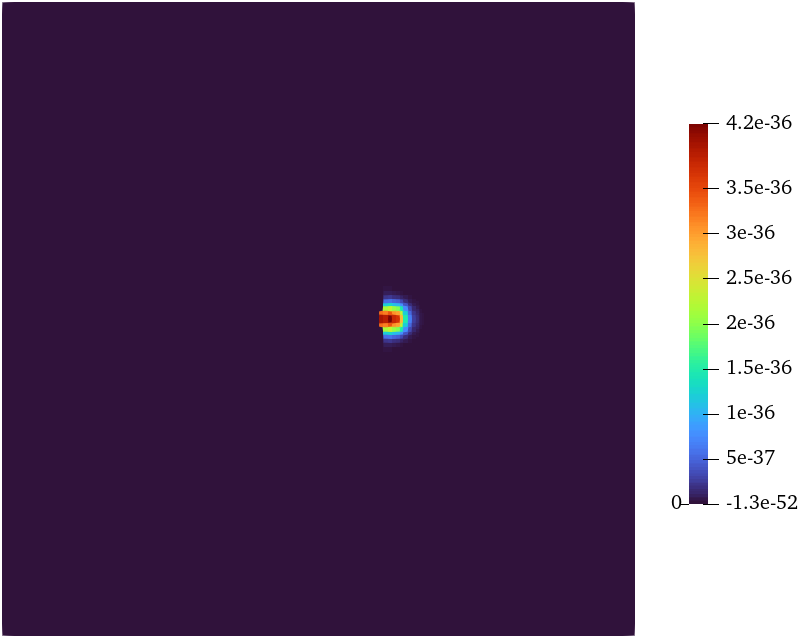}
		\includegraphics[width=0.24\textwidth]{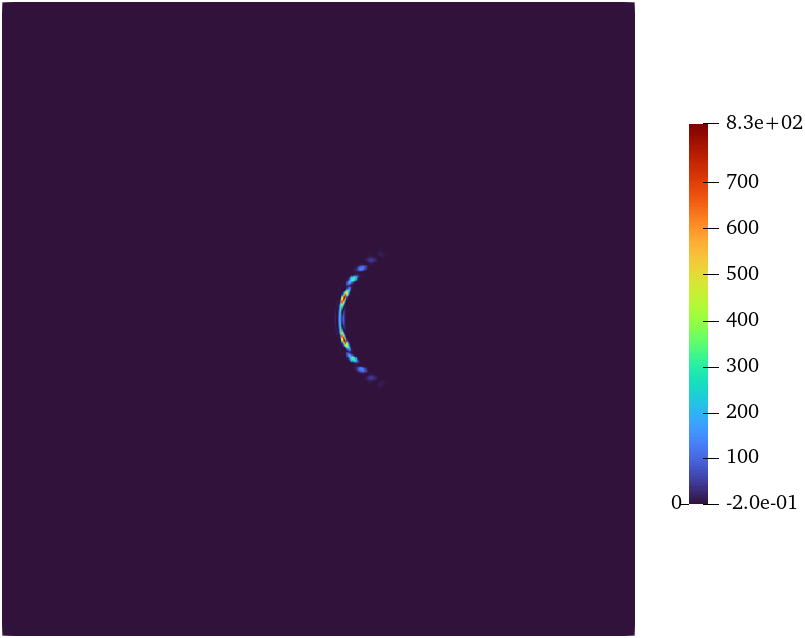}
		\includegraphics[width=0.24\textwidth]{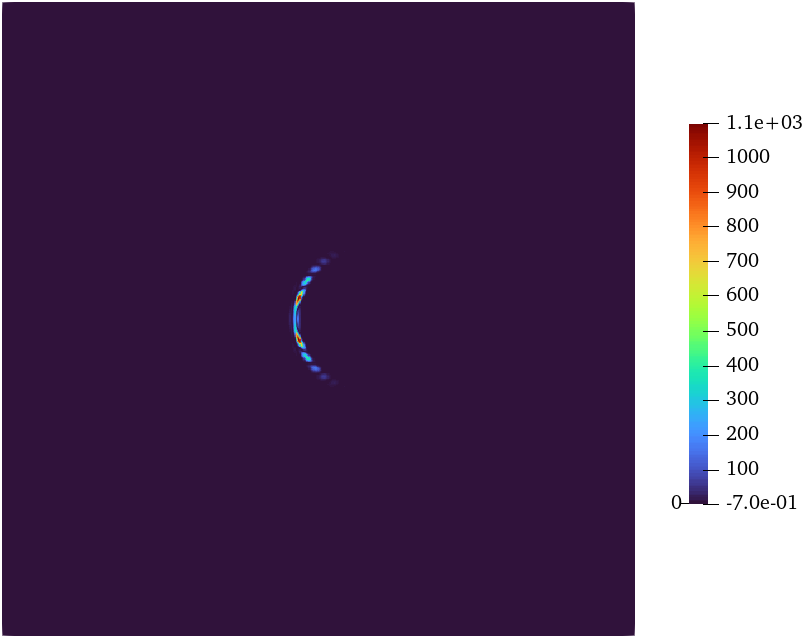}
		\includegraphics[width=0.24\textwidth]{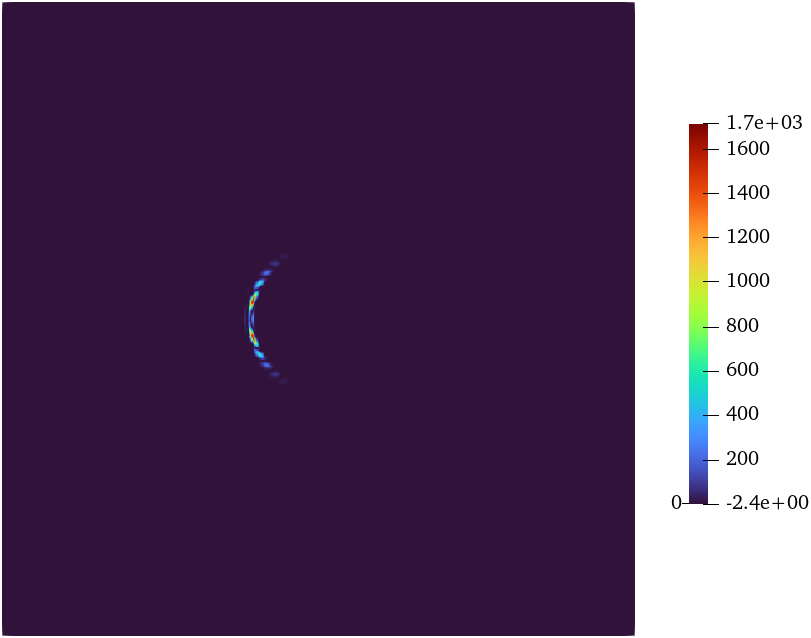}
		\includegraphics[width=0.24\textwidth]{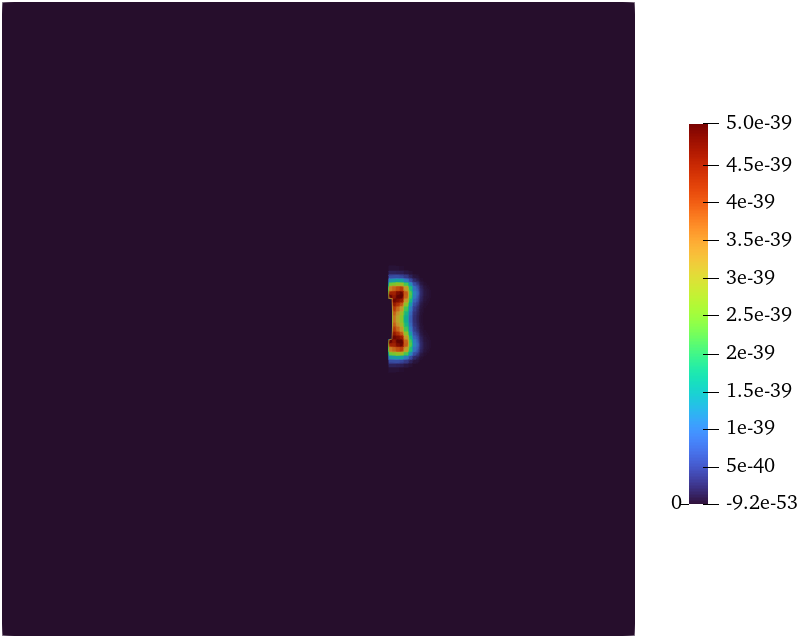}
		\includegraphics[width=0.24\textwidth]{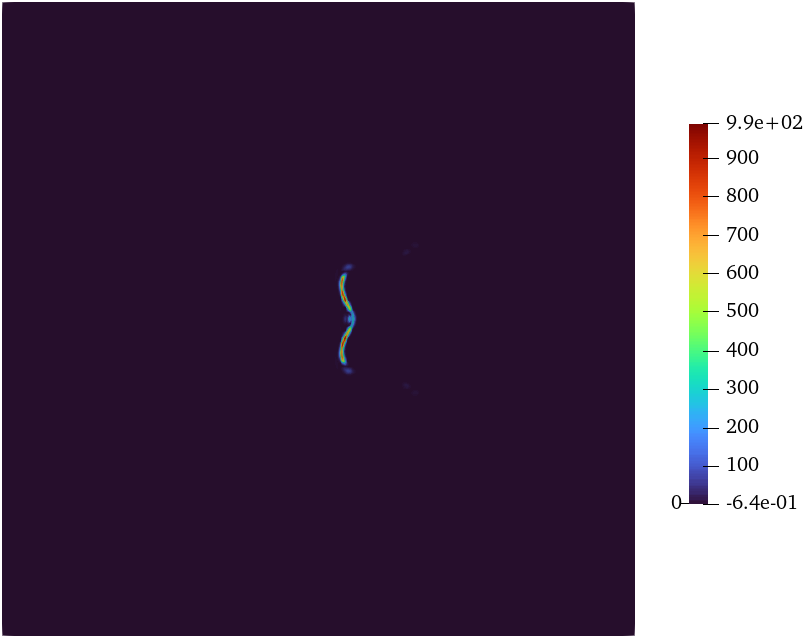}
		\includegraphics[width=0.24\textwidth]{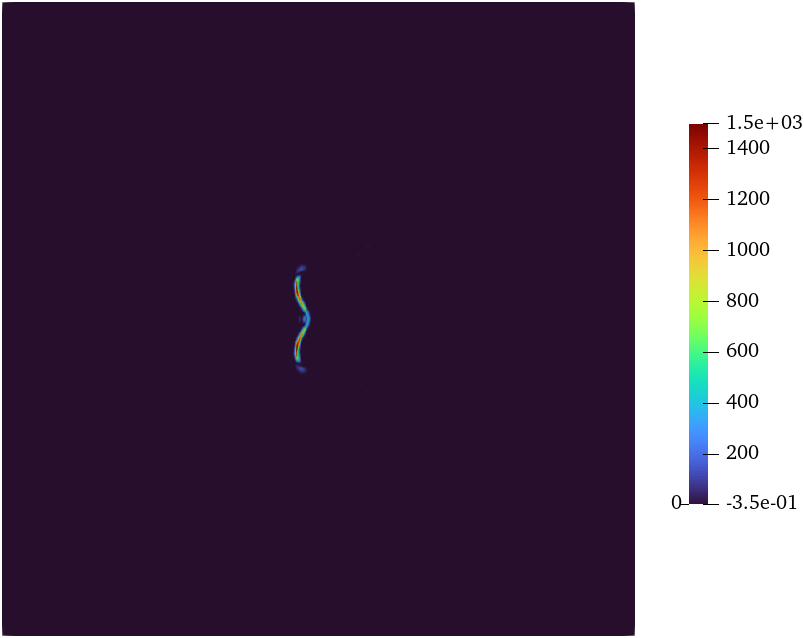}
		\includegraphics[width=0.24\textwidth]{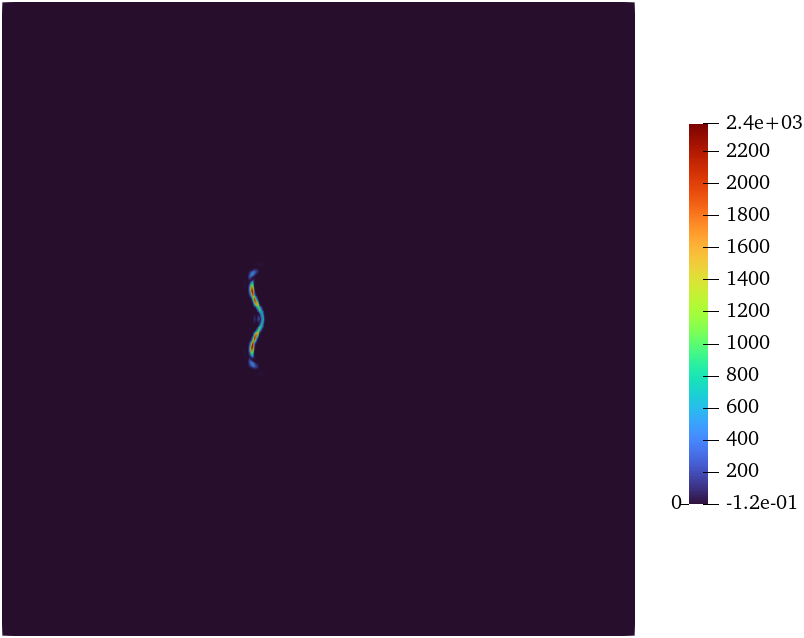}
	\end{center}
	\caption{Snapshots of the density of $\mathcal{R}_{h,{\rm embed}}^n$
		for the four sound-soft moving-object tests. From top to bottom:
		circular object in the moderate-frequency regime, star-shaped
		object in the moderate-frequency regime, circular object in the
		high-frequency regime, and star-shaped object in the
		high-frequency regime. Within each row, the times are $t=0.2, 0.6,
			0.8,$ and $1.0$ from left to right.}\label{fig:residual-density}
\end{figure}

Figure~\ref{fig:residual-density} provides an additional insight into the
remainder term related to the motion of the object and appeared in the fully discrete energy identity. In all
four sound-soft moving-object cases, the density of
$\mathcal{R}_{h,{\rm embed}}^n$ is concentrated in a narrow neighborhood
of the moving diffuse interface, while remaining negligible in the bulk
region and inside the PML. This behavior is consistent with
the analytical form of the remainder term, which is activated only
through time derivatives of $\psi_\varepsilon$ and
$W_\varepsilon$. The high-frequency cases exhibit larger localized
values because sharper wave gradients interact more strongly with the
transport of the interface profile. However, no uncontrolled spreading
of the density function is observed, which supports the numerical stability of
the adaptive fully discrete \PMLDE{} scheme in the moving-interface
setting.

Finally, we present a numerical experiment for a moving ship-shaped object. The
ship is initially centered at $(4.1,2.0)$ and undergoes the same rigid
translation with velocity $(-5,0)$ used in the previous examples. The same
two incident-wave parameter sets are used. Since the motion is uniform, the
sound-hard acceleration correction vanishes.

\begin{figure}[H]
	\begin{center}
		\includegraphics[width=0.24\textwidth]{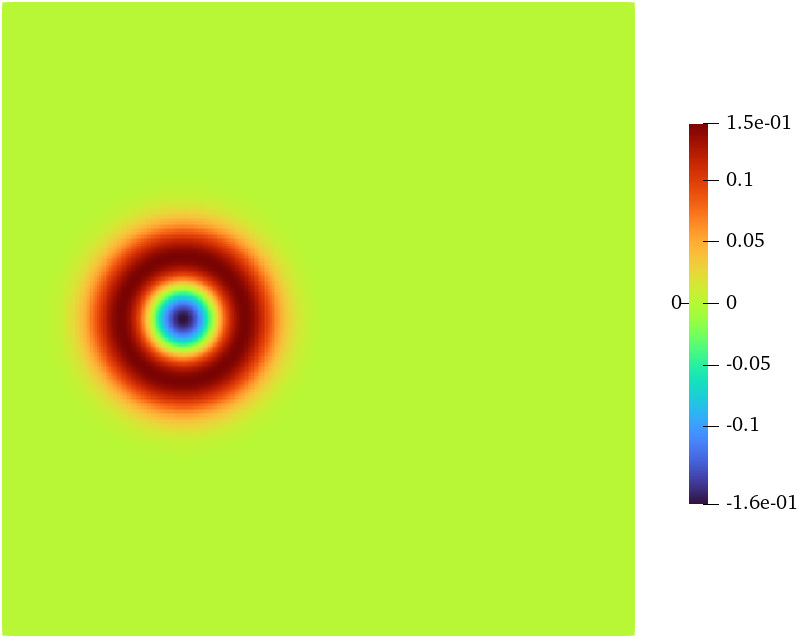}
		\includegraphics[width=0.24\textwidth]{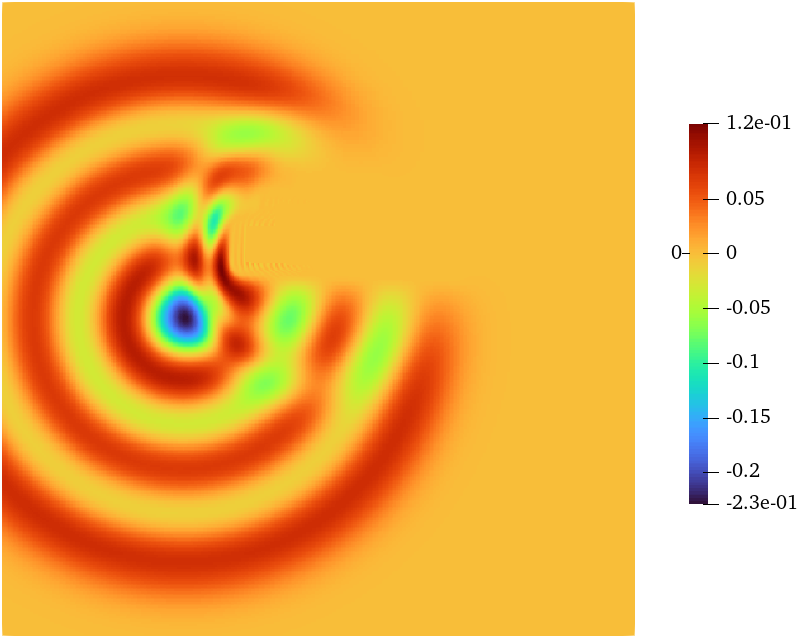}
		\includegraphics[width=0.24\textwidth]{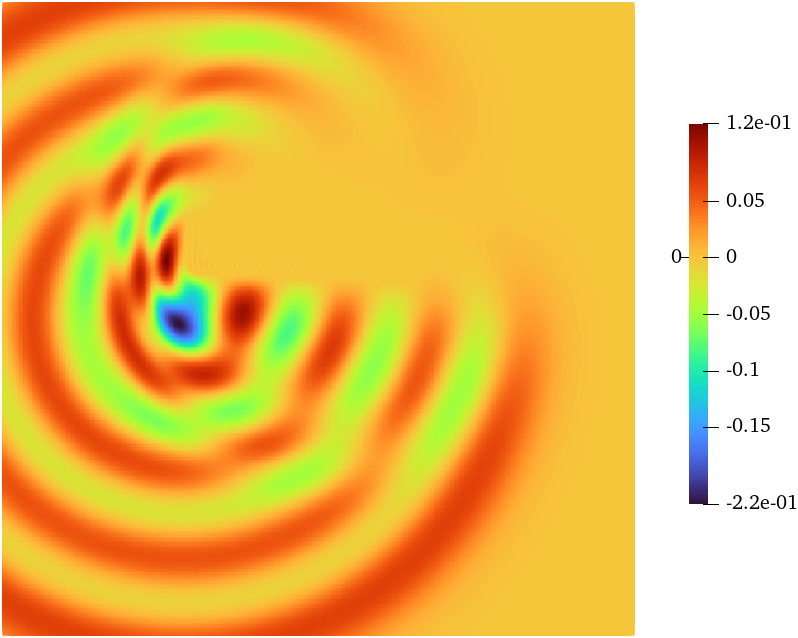}
		\includegraphics[width=0.24\textwidth]{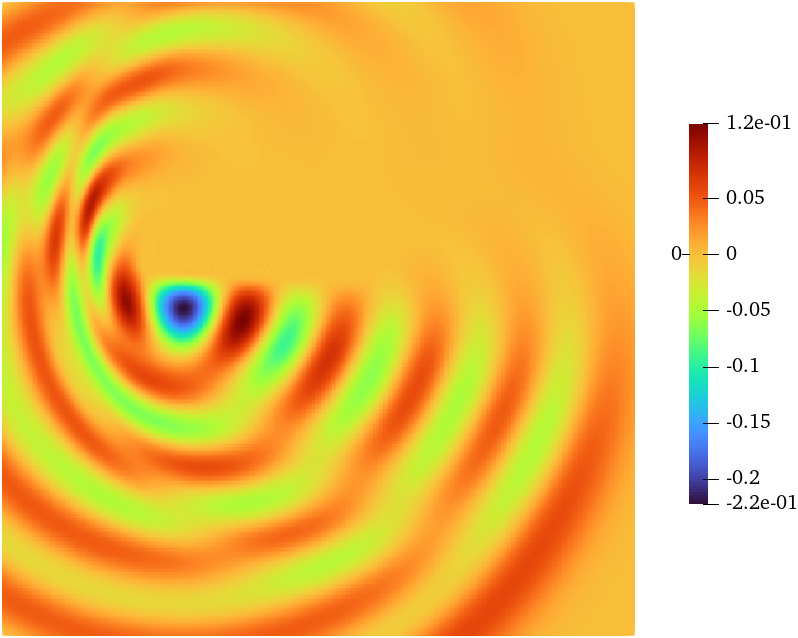}
		\includegraphics[width=0.24\textwidth]{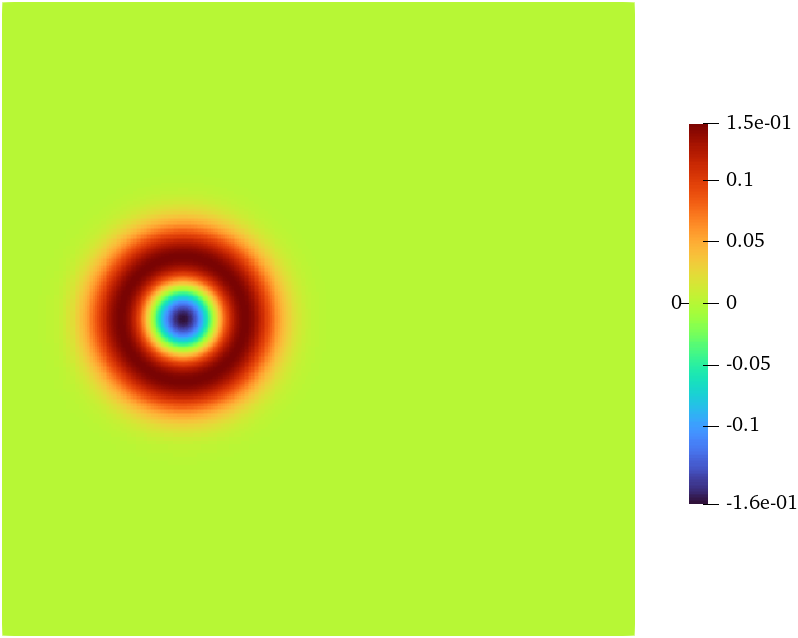}
		\includegraphics[width=0.24\textwidth]{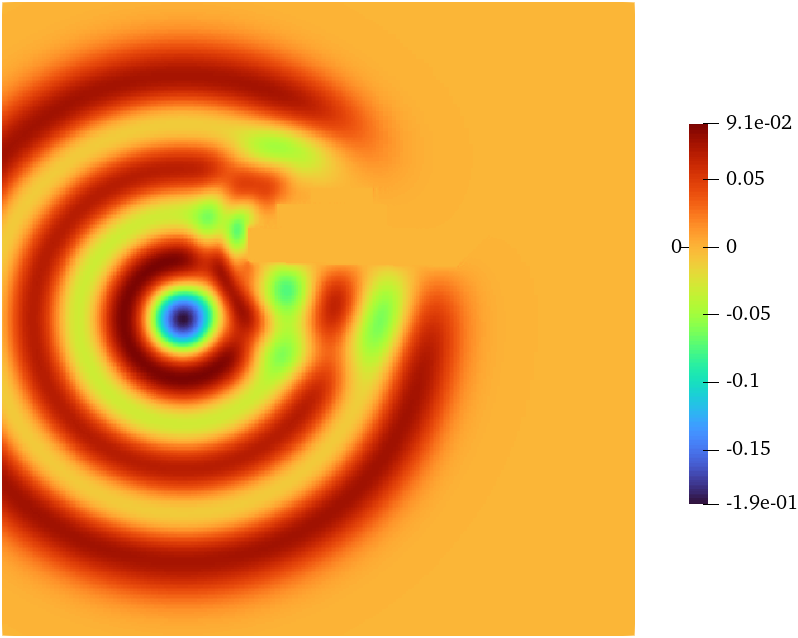}
		\includegraphics[width=0.24\textwidth]{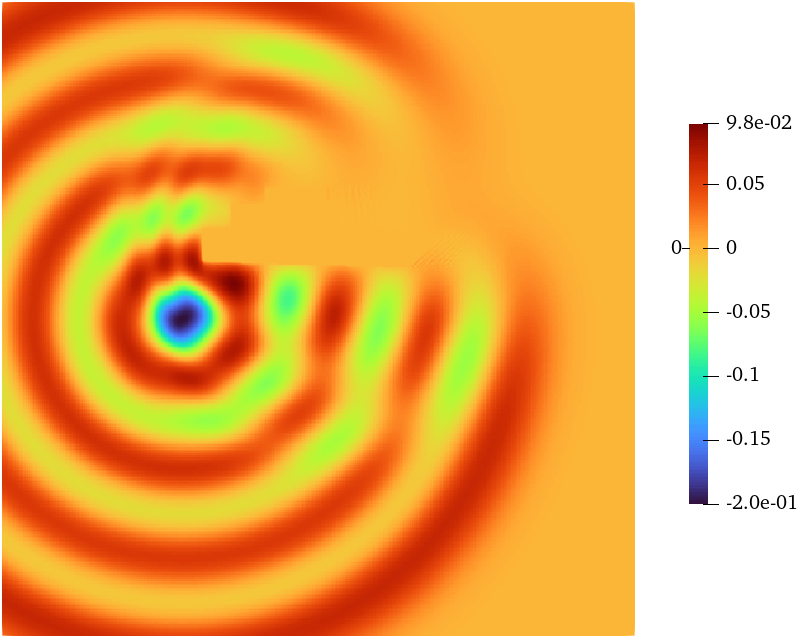}
		\includegraphics[width=0.24\textwidth]{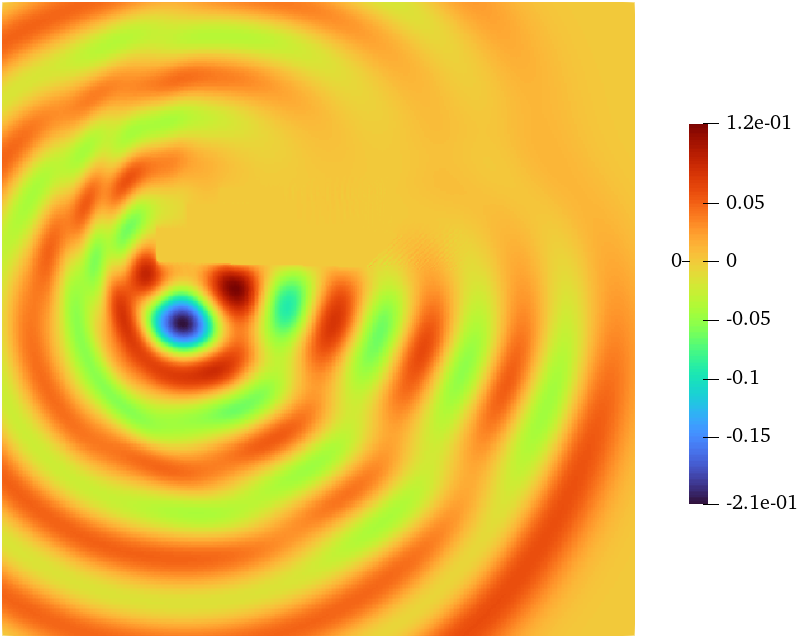}
		\begin{minipage}{0.235\textwidth}
			\hspace*{-0.25cm}
			\includegraphics[width=\linewidth]{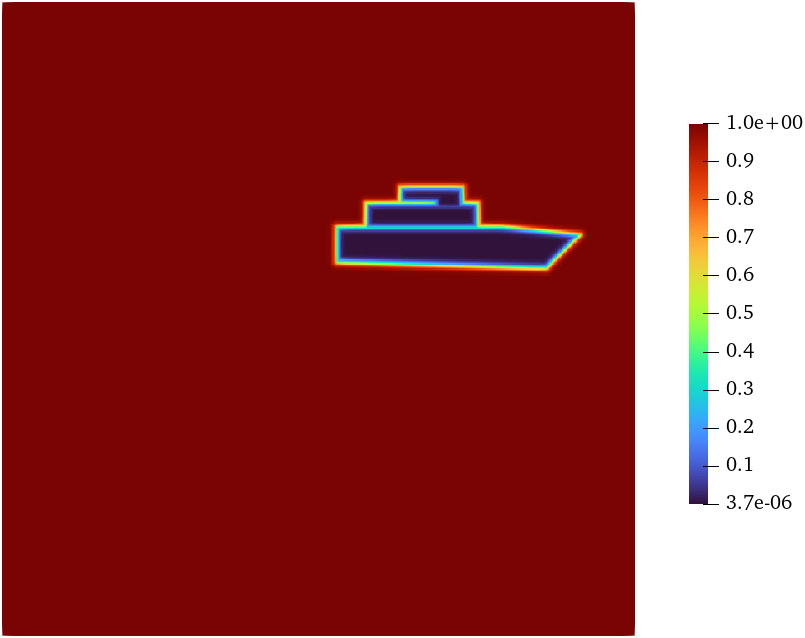}
		\end{minipage}
		\begin{minipage}{0.235\textwidth}
			\hspace*{-0.15cm}
			\includegraphics[width=\linewidth]{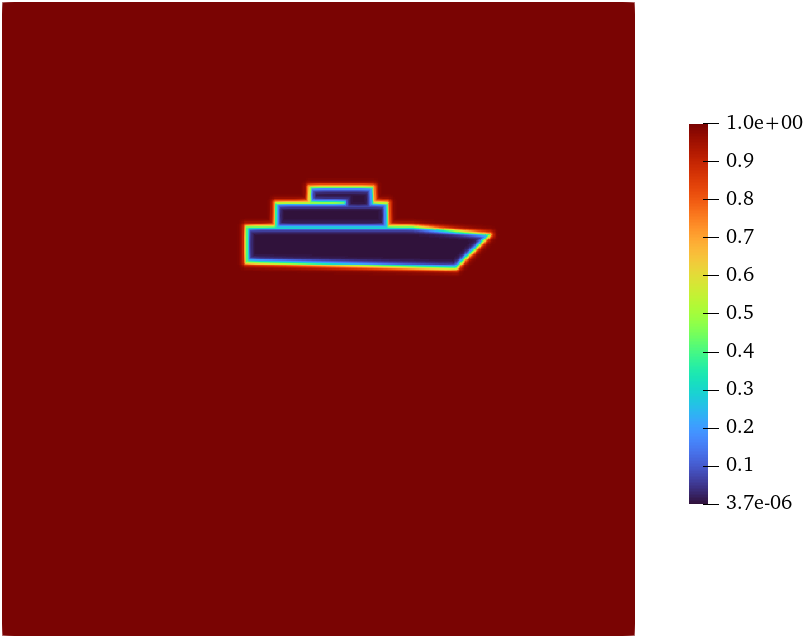}
		\end{minipage}
		\begin{minipage}{0.235\textwidth}
			\hspace*{-0.1cm}
			\includegraphics[width=\linewidth]{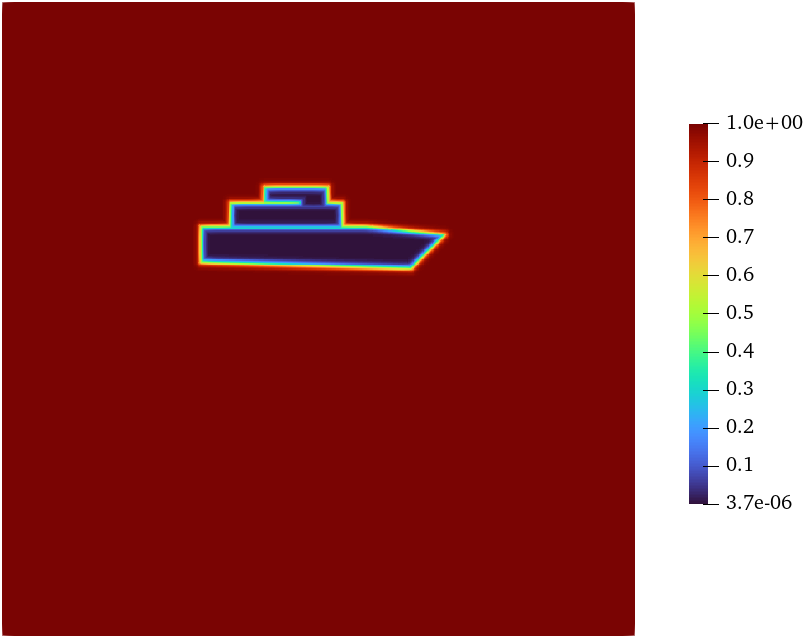}
		\end{minipage}
		\begin{minipage}{0.235\textwidth}
			\hspace*{0.05cm}
			\includegraphics[width=\linewidth]{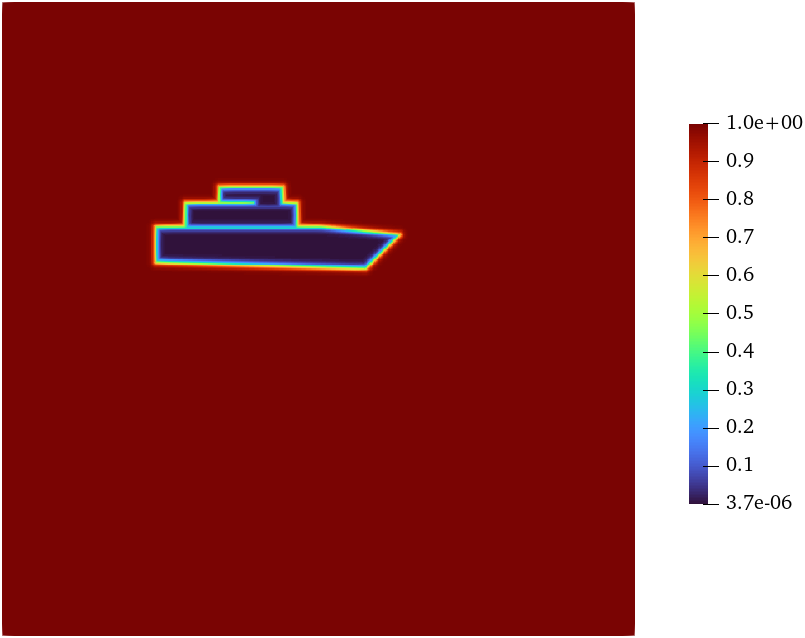}
		\end{minipage}
	\end{center}
	\caption{Pressure snapshots of a moving ship-shaped object in the
		moderate-frequency regime under sound-soft (top row) and sound-hard
		(middle row) boundary conditions at $t = 0.2, 0.6, 0.8,$ and $1.0$
		in $\Omega_{\rm phy}$; the bottom row shows the corresponding
		snapshots of $\psi_\varepsilon$.}\label{fig:moving-simple-ship.}
\end{figure}
\begin{figure}[H]
	\begin{center}
		\includegraphics[width=0.45\textwidth]{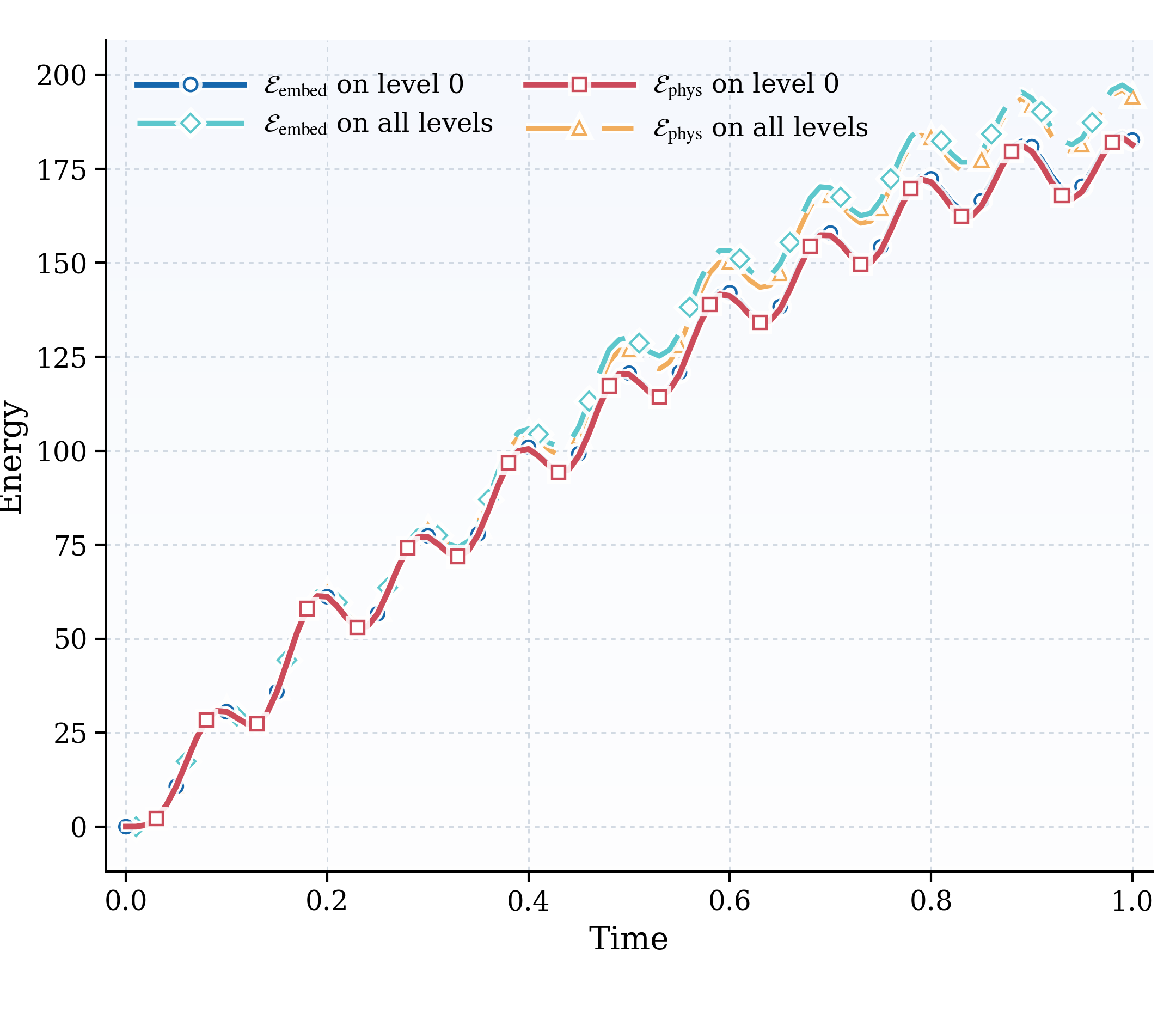}
		\includegraphics[width=0.45\textwidth]{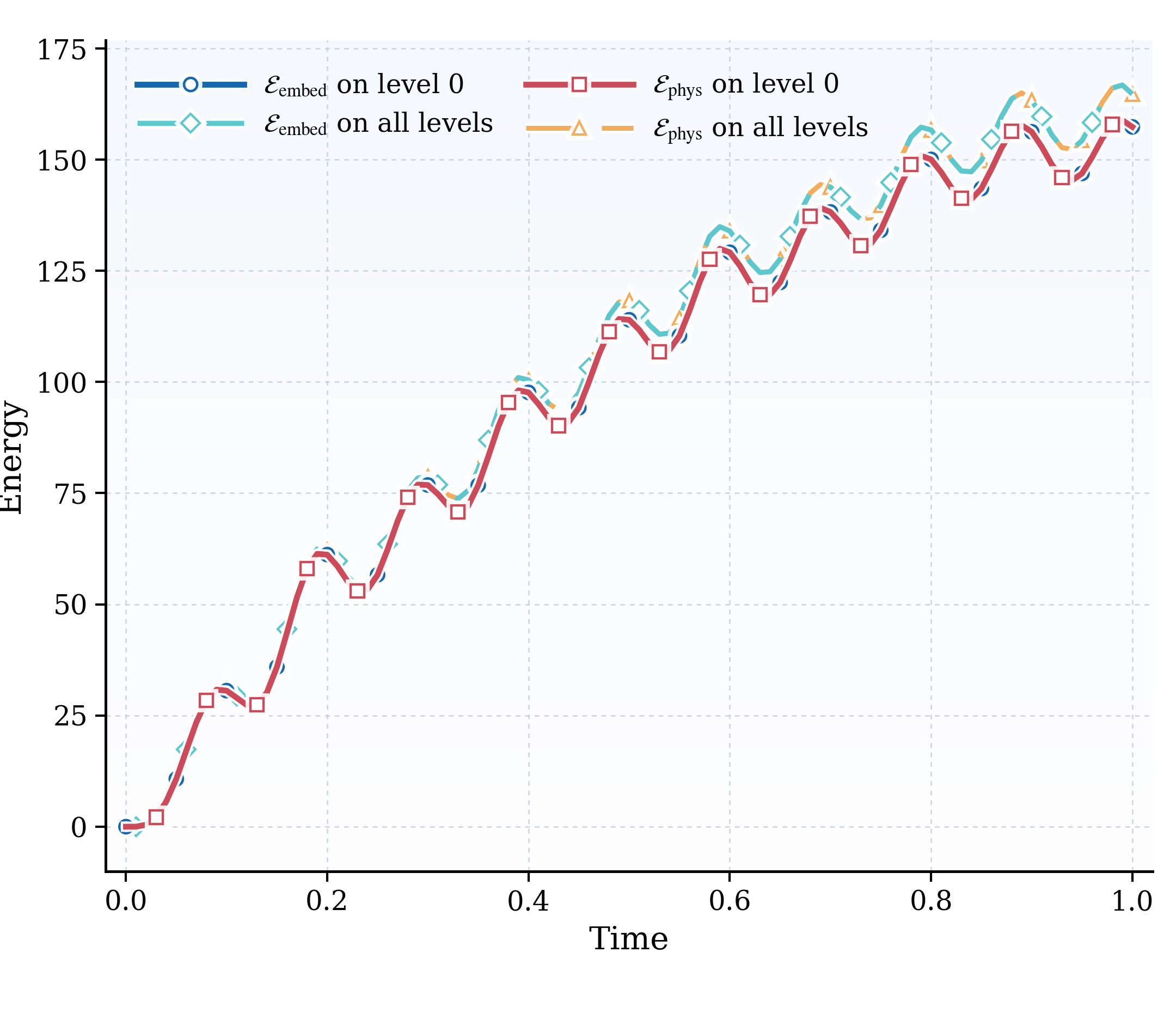}
	\end{center}
	\caption{Energy trajectories for the moving ship-shaped object in the
		moderate-frequency regime. Left: sound-soft case. Right: sound-hard
		case. In each panel, the weighted energy, $\mathcal{E}_{\rm
				embed}$, and the physical energy, $\mathcal{E}_{\rm phys}$, are
		reported both at the coarsest level and at the composite adaptive
		hierarchy.}\label{fig:energy-ship-low}
\end{figure}

\begin{figure}[H]
	\begin{center}
		\includegraphics[width=0.24\textwidth]{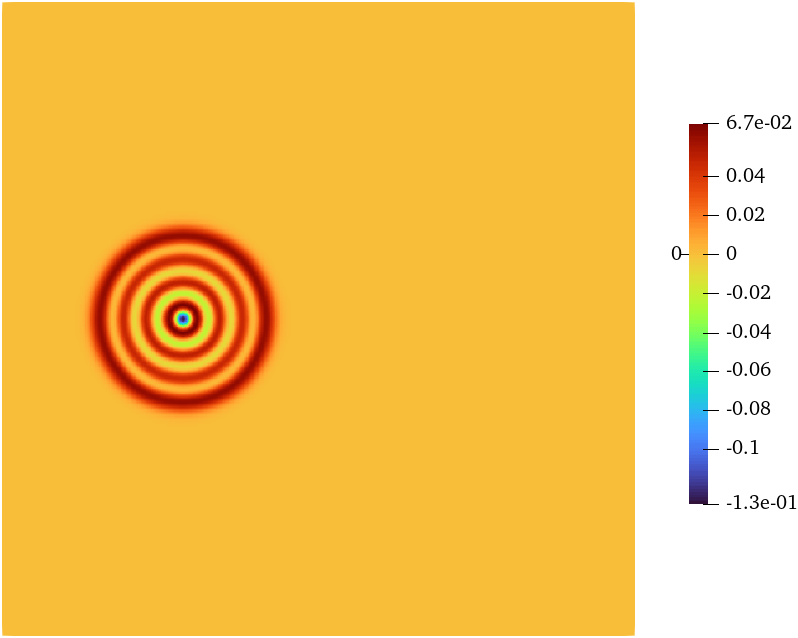}
		\includegraphics[width=0.24\textwidth]{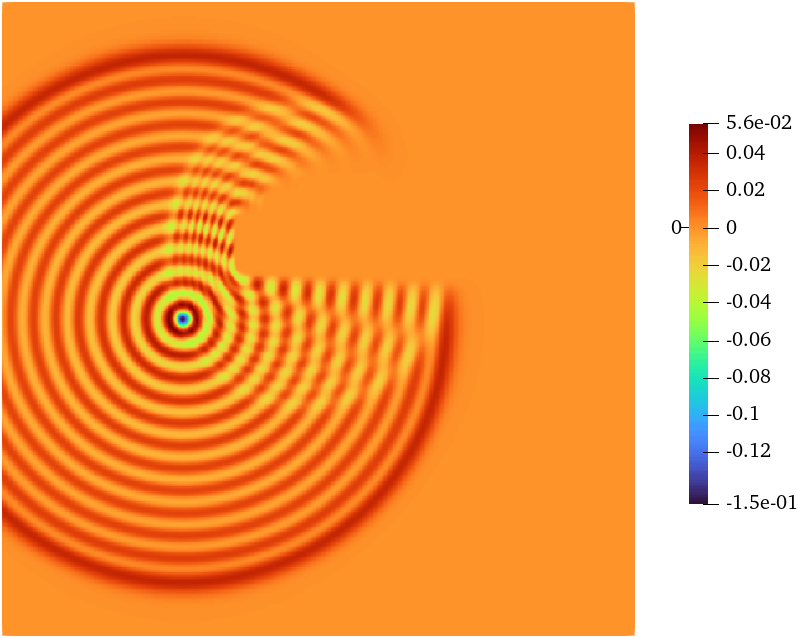}
		\includegraphics[width=0.24\textwidth]{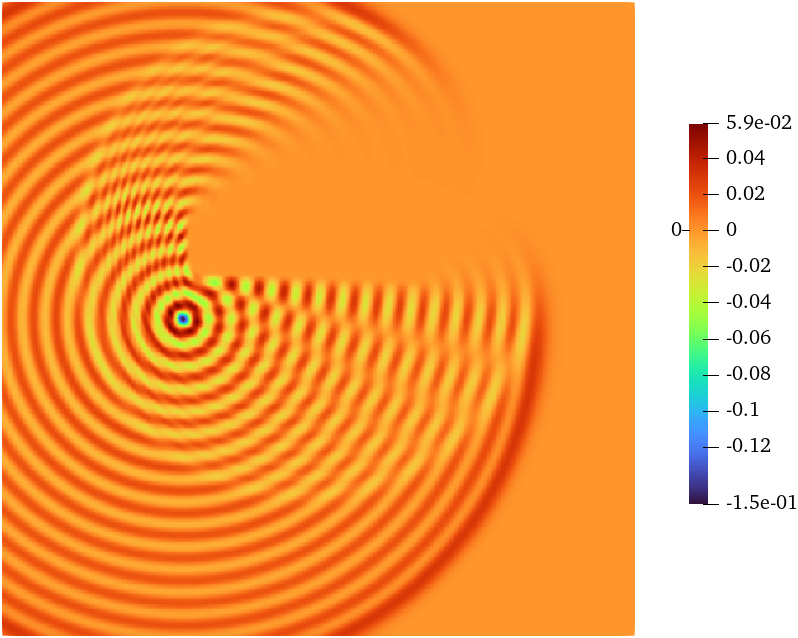}
		\includegraphics[width=0.24\textwidth]{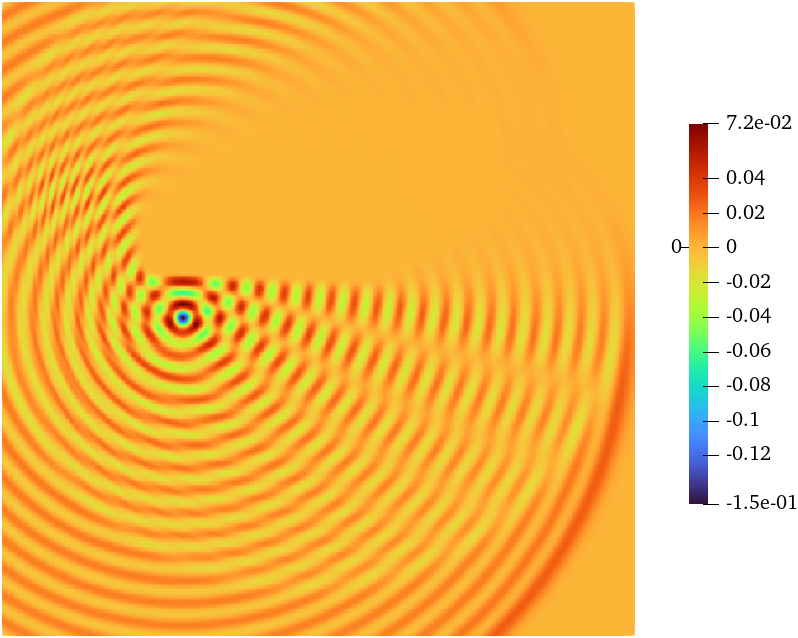}
		\includegraphics[width=0.24\textwidth]{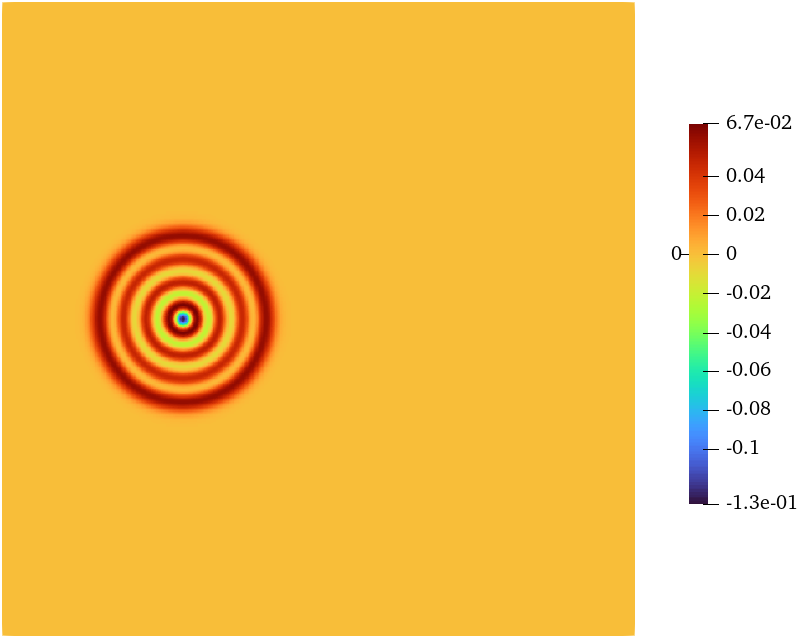}
		\includegraphics[width=0.24\textwidth]{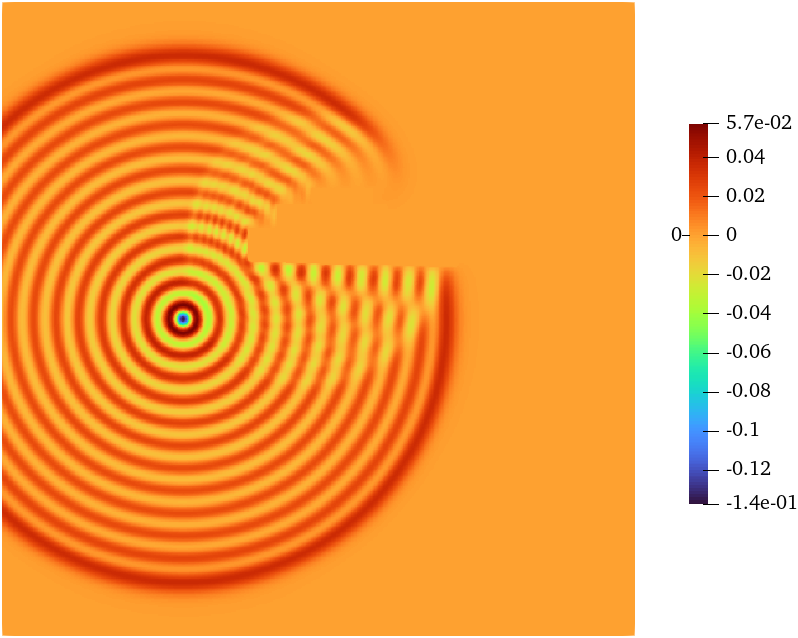}
		\includegraphics[width=0.24\textwidth]{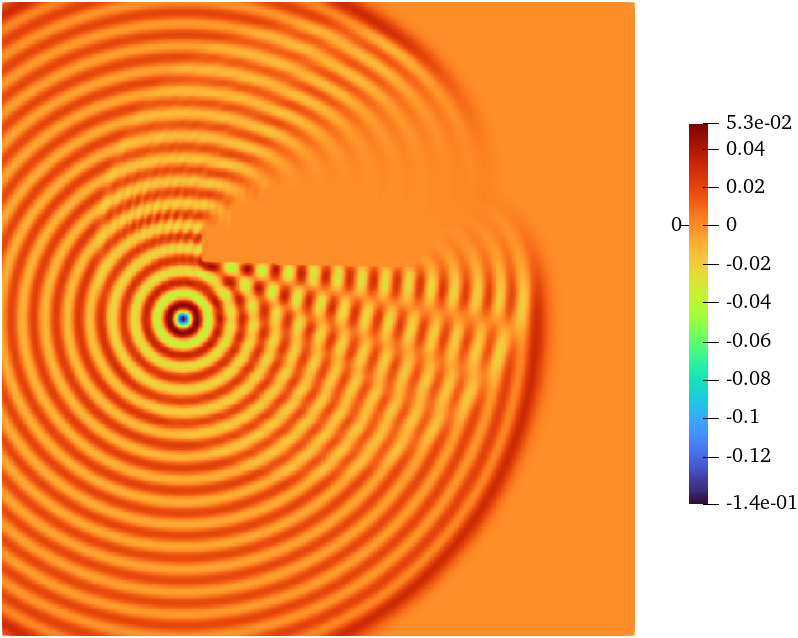}
		\includegraphics[width=0.24\textwidth]{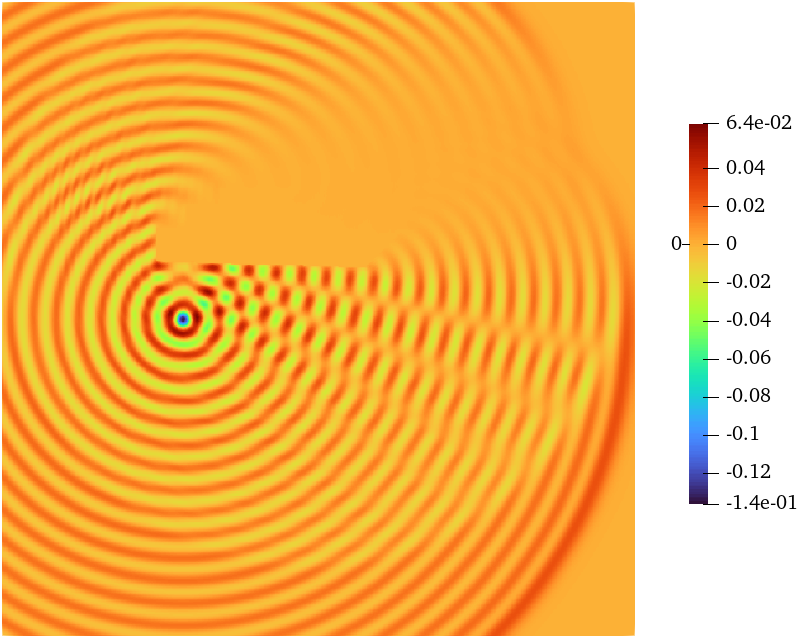}
		\begin{minipage}{0.235\textwidth}
			\hspace*{-0.25cm}
			\includegraphics[width=\linewidth]{./figs/psi_ship.0002.png}
		\end{minipage}
		\begin{minipage}{0.235\textwidth}
			\hspace*{-0.15cm}
			\includegraphics[width=\linewidth]{./figs/psi_ship.0006.png}
		\end{minipage}
		\begin{minipage}{0.235\textwidth}
			\hspace*{-0.1cm}
			\includegraphics[width=\linewidth]{./figs/psi_ship.0008.png}
		\end{minipage}
		\begin{minipage}{0.235\textwidth}
			\hspace*{0.05cm}
			\includegraphics[width=\linewidth]{./figs/psi_ship.0010.png}
		\end{minipage}
	\end{center}
	\caption{Pressure snapshots of a moving ship-shaped object in the
		high-frequency regime under sound-soft (top row) and sound-hard
		(middle row) boundary conditions at $t = 0.2, 0.6, 0.8,$ and $1.0$
		in $\Omega_{\rm phy}$; the bottom row shows the corresponding
		snapshots of $\psi_\varepsilon$.}\label{fig:moving-simple-ship_high.}
\end{figure}
\begin{figure}[H]
	\begin{center}
		\includegraphics[width=0.45\textwidth]{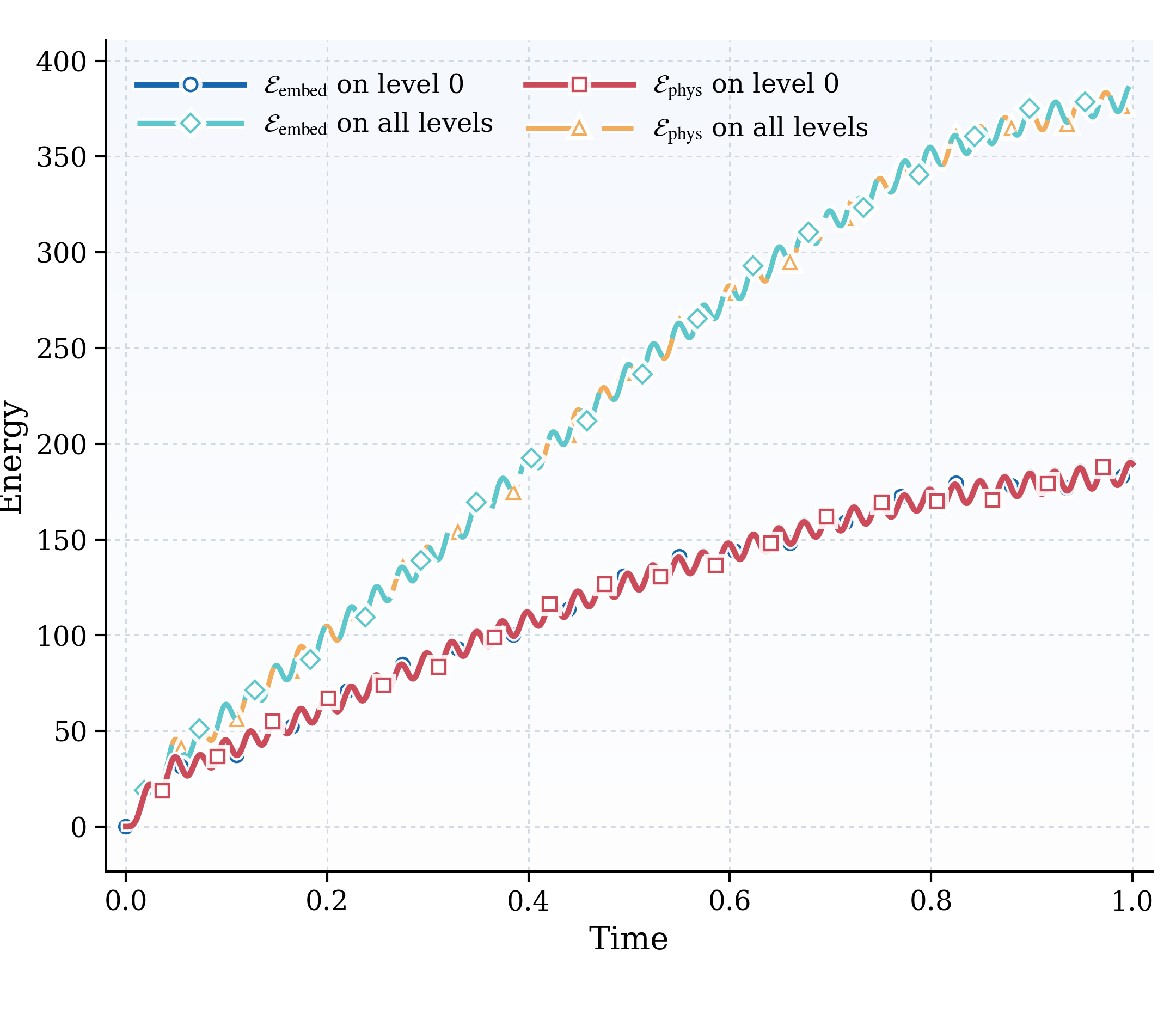}
		\includegraphics[width=0.45\textwidth]{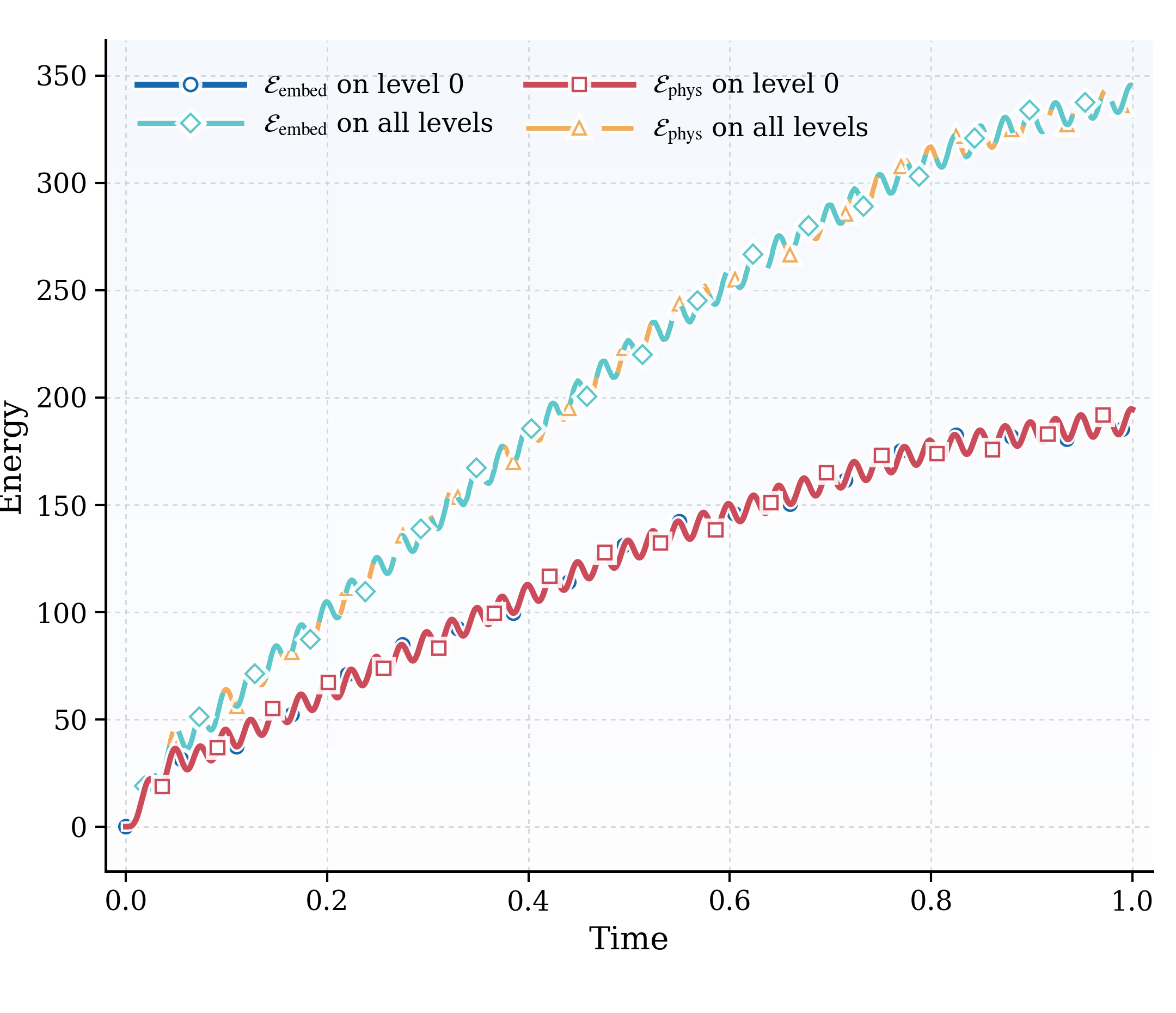}
	\end{center}
	\caption{Energy trajectories for the moving ship-shaped object in the
		high-frequency regime. Left: sound-soft case. Right: sound-hard
		case. In each panel, the weighted energy, $\mathcal{E}_{\rm
				embed}$, and the physical energy, $\mathcal{E}_{\rm phys}$, are
		reported both at the coarsest level and at the composite adaptive
		hierarchy.}\label{fig:energy-ship-high}
\end{figure}

Figures~\ref{fig:moving-simple-ship.} and \ref{fig:moving-simple-ship_high.}
show the scattered wave field for the moving ship-shaped object in the moderate-
and high-frequency regimes, respectively. In both cases, the asymmetric hull produces a
tilted reflected front near the bow and a visibly one-sided wake downstream.
Compared with the circular and star-shaped objects, the slanted bow and the
stepped upper profile lead to a stronger local distortion near the body. This
effect is more evident in the high-frequency regime, where the shorter
wavelength resolves the geometric features more clearly. The bottom rows show
that the embedded indicator $\psi_\varepsilon$ transports the ship smoothly on
the fixed grid, and no visible spurious reflection appears at the outer PML
boundary. The energy curves in Figures~\ref{fig:energy-ship-low} and
\ref{fig:energy-ship-high} follow the same trend as in the circular and
star-shaped tests: the gap between the level-0 and all-level energies is more
pronounced in the high-frequency regime, while $\mathcal{E}_{\rm embed}$ and
$\mathcal{E}_{\rm phys}$ remain close for both boundary conditions.

\section{Conclusion}\label{sec:conclusion}

\noindent \indent We have developed a structure-preserving computational framework for acoustic wave scattering by moving objects. The framework combines a PML reformulation of the wave-scattering problem, a domain-embedding description posed on a fixed computational domain, and a temporal leap-frog scheme derived from a midpoint discretization, together with a spatially adaptive algorithm.

The key analytical result is that the dissipative property of the fixed-geometry PML problem can be explicitly formulated as a gradient flow with a quadratic energy. This energy-dissipative structure is preserved under the proposed leap-frog discretization of the two-field reduced PML-DE system. After domain embedding, the moving-object formulation satisfies a weighted energy balance law in which the contribution of interface motion appears explicitly in the rate of change of the energy. This formulation clarifies the respective roles of PML damping, interface dynamics, and object-interior dissipation in the model.

The resulting  formulation is particularly advantageous for moving-object problems because it avoids body-fitted remeshing while remaining compatible with Cartesian-grid solvers. Through several numerical experiments, we have validated the expected qualitative behavior of the numerical approximation for both static and moving objects under sound-soft and special sound-hard boundary treatments. A broader investigation on the general sound-hard case will be pursued in a subsequent work.

\section*{Acknowledgements}
Xuelong Gu's research is  supported  by
NSF award  OIA-2242812.  and Qi Wang's research is  partially supported  by
NSF awards  OIA-2242812 and DMS-2038080, DOE award DE-SC0025229,
and an SC GAIN-CRP award.

\appendix

\section{Proofs of Lemmas~\ref{lem:constraint} and \ref{lem:reduction}}

\begin{proof}[Proof of Lemma~\ref{lem:constraint}]
	By definition,
	$D_t^+ \bm{r}^n = D_t^+ \bm{\chi}^n - c\nabla D_t^+ p^n - D_t^+
		\bm{\lambda}^n$.
	Substituting \eqref{eq:wave-pml-semi-discrete-component} yields
	\begin{equation}
		\begin{aligned}
			D_t^+ \bm{r}^n
			 & = A_t^+\big(c\nabla
			q^{n}-2\varGamma_1\bm{\chi}^{n}+\varGamma_1 \bm{\lambda}^{n}\big)
			-c\nabla A_t^+ \big(q^{n}-a p^{n}\big)
			-A_t^+
			\big(\varGamma_2\bm{\chi}^{n}-\widetilde{\varGamma}_1\bm{\lambda}^{n}\big) \\
			 & = c a \nabla A_t^+ p^n -
			(2\varGamma_1+\varGamma_2)A_t^+ \bm{\chi}^{n} +
			(\varGamma_1+\widetilde{\varGamma}_1)A_t^+ \bm{\lambda}^{n}                \\
			 & = aA_t^+ (c \nabla p^n - \bm{\chi}^n + \bm{\lambda}^n) = -a
			A_t^+ \bm{r}^n,
		\end{aligned}
	\end{equation}
	where we have used $2\varGamma_1+\varGamma_2=aI$ and
	$\varGamma_1+\widetilde{\varGamma}_1=aI$ in the last step. If
	$\bm{r}^0=0$, a straightforward inductive argument yields
	$\bm{r}^n\equiv 0$.
\end{proof}

\begin{proof}[Proof of Lemma~\ref{lem:reduction}]
	We first derive \eqref{eq:wave-pml-semi-discrete-leap-frog} from
	\eqref{eq:wave-pml-semi-discrete-component}. Rewriting the first equation of
	\eqref{eq:wave-pml-semi-discrete-component} at $t_{n-\frac12}$, applying
	$D_t^+$, and then applying $A_t^-$ to the second equation give
	\begin{equation}\label{eq:thm-reduction-pq}
		\begin{aligned}
			D_t^+ D_t^- p^n & = D_t^+ A_t^- q^n - a D_t^+ A_t^- p^n, \\
			A_t^- D_t^+ q^n & = -b A_t^-A_t^+ p^n + c\nabla\cdot
			A_t^-A_t^+\bm{\chi}^n.
		\end{aligned}
	\end{equation}
	Eliminating $q^n$ yields
	$D_t^2 p^n + aD_{2t}p^n + bA_t^2 p^n = c\nabla\cdot A_t^2\bm{\chi}^n$,
	where we used $D_t^+D_t^- = D_t^2$, $D_t^+A_t^- = D_{2t}$, and
	$A_t^-A_t^+ = A_t^2$. Invoking
	$\bm{\chi}^n=c\nabla p^n+\bm{\lambda}^n$ from Lemma~\ref{lem:constraint}
	in this identity gives the first equation of
	\eqref{eq:wave-pml-semi-discrete-leap-frog}; substituting the same identity
	into the last line of \eqref{eq:wave-pml-semi-discrete-component} gives the
	second equation.

	We next verify \eqref{eq:reconstruct_qn_chin}. The last two
	identities follow
	immediately from Lemma~\ref{lem:constraint}. Substituting
	$\bm{\chi}^n=c\nabla p^n+\bm{\lambda}^n$ into the second equation of
	\eqref{eq:wave-pml-semi-discrete-component} shows that
	$D_t^+ q^n = F^{n+\frac12}$ and $q^{n+\frac12}=A_t^+ q^n$, where the second
	identity follows from \eqref{eq:den-qF} and the first equation of
	\eqref{eq:wave-pml-semi-discrete-component}. These two relations give the
	first two formulas in \eqref{eq:reconstruct_qn_chin}.

	Finally, assume \eqref{eq:wave-pml-semi-discrete-leap-frog} and
	\eqref{eq:reconstruct_qn_chin}. Then these midpoint identities,
	\eqref{eq:den-qF}, and the third line of
	\eqref{eq:reconstruct_qn_chin} imply
	\begin{equation*}
		\begin{aligned}
			D_t^+ p^n                                              & = A_t^+ q^n - aA_t^+ p^n,                          \\
			D_t^+ q^n                                              & = -bA_t^+ p^n + c^2\Delta A_t^+ p^n + c\nabla\cdot
			A_t^+\bm{\lambda}^n,                                                                                        \\
			D_t^+ \bm{\lambda}^n + \varGamma_1 A_t^+\bm{\lambda}^n & =
			c\varGamma_2\nabla A_t^+ p^n,                                                                               \\
			\bm{\chi}^n                                            & = c\nabla p^n + \bm{\lambda}^n.
		\end{aligned}
	\end{equation*}
	Repeating at the discrete level the same algebraic manipulations used to
	derive \eqref{eq:first-order-system} from the continuous second-order model
	then recovers \eqref{eq:wave-pml-semi-discrete-component}. This
	completes the proof.
\end{proof}

\bibliographystyle{abbrv}
\bibliography{ref}

\end{document}